\documentclass[11pt]{amsart}
\usepackage[T1]{fontenc}
\usepackage[latin1]{inputenc}
\usepackage{a4}
\usepackage{setspace}
\usepackage{xypic}
\usepackage{amssymb}
\usepackage{amsthm}
\usepackage[english]{babel}
\usepackage{latexsym}
\date{}

\usepackage{srcltx}

\newtheorem{thm}{Theorem}[section]

\newtheorem{example}[thm]{Example}
\newtheorem{definition}[thm]{Definition}
\newtheorem{rem}[thm]{Remark}

\newtheorem{conjL'}[thm]{Conjecture $\textbf{L}(\overline{\mathcal{X}}_{et},d)_{\geq0}$}

\newtheorem{prop}[thm]{Proposition}
\newtheorem{prop-definition}[thm]{Proposition--Definition}

\newtheorem{lem}[thm]{Lemma}
\newtheorem{cor}[thm]{Corollary}
\newtheorem{notation}[thm]{Notation}
\newenvironment{f-proof}[1][\sc Proof.]{\begin{trivlist}
\item[\hskip \labelsep {\bfseries #1}]}{\hfill{$\square$}\end{trivlist}}
\newcommand{\Coprod}{\displaystyle\coprod}

\newcommand{\fonc}[5]{
 \begin{array}{cccc}
 #1: & #2 & \longrightarrow & #3\\
     & #4 & \longmapsto & #5
 \end{array}
}
\newcommand{\appl}[4]{
 \begin{array}{cccc}
  #1 & \longrightarrow & #2\\
  #3 & \longmapsto & #4
 \end{array}
}

\begin{document}
\title[Invariants of symmetric bundles]{The classifying topos of a group scheme and invariants of symmetric bundles}

\author{Ph. Cassou-Nogu\`es, T. Chinburg, B. Morin, M.J. Taylor  }

\maketitle

\begin{abstract}
Let $Y$ be a scheme in which $2$ is invertible and let $V$ be a rank $n$ vector bundle on $Y$ endowed with a non-degenerate symmetric bilinear form $q$. The orthogonal group ${\bf O}(q)$ of the form $q$ is a group scheme over $Y$ whose cohomology ring $H^*(B_{{\bf O}(q)},{\bf Z}/2{\bf Z})\simeq A_Y[HW_1(q),..., HW_n(q)]$ is a polynomial algebra over the \'etale cohomology ring $A_Y:=H^*(Y_{et},{\bf Z}/2{\bf Z})$ of the scheme $Y$. Here the $HW_i(q)$'s are Jardine's universal Hasse-Witt invariants and $B_{{\bf O}(q)}$ is the classifying topos of ${\bf O}(q)$ as defined by Grothendieck and Giraud. The cohomology ring $H^*(B_{{\bf O}(q)},{\bf Z}/2{\bf Z})$ contains canonical classes $\mathrm{det}[q]$ and $[C_q]$ of degree $1$ and $2$ respectively, which are obtained from the determinant map and the Clifford group of $q$. The classical Hasse-Witt invariants $w_i(q)$ live in the ring $A_Y$.

Our main theorem provides a computation of $\mbox{det}[q]$ and $[C_{q}]$ as polynomials in $HW_{1}(q)$ and $HW_{2}(q)$ with coefficients in $A_Y$ written in terms of $w_1(q),w_2(q)\in A_Y$. This result is the source of numerous standard comparison formulas for classical Hasses-Witt invariants of quadratic forms. Our proof is based on computations with  (abelian and non-abelian) Cech cocycles in the topos $B_{{\bf O}(q)}$. This requires a general study of the cohomology of the classifying topos of a group scheme, which we carry out in the first part of this paper.

\end{abstract}

\section{Introduction}

In \cite{Fro} and \cite{Serre}, Fr\"ohlich  and Serre proved some beautiful formulas which compared invariants associated to various kinds of Galois representations and quadratic forms defined over a field $K$ of characteristic different from  $2$. Their work has inspired numerous generalizations (see for example \cite{EKV} and \cite{CNET1}). The basic underlying idea may be summarized as follows. Let $(V, q)$  be  a symmetric bundle,  defined over a scheme $Y$ in which $2$ is invertible, and let ${\bf O}(q)$ be the orthogonal group of $(V,q)$ considered as a group scheme over $Y$. We may associate to any orthogonal representation  $\rho: G\rightarrow {\bf O}(q)$ of a finite discrete group $G$ and any $G$--torsor $X$ on $Y$ a cocycle in the cohomology set $H^1(Y_{et}, {\bf O}(q))$. Since this set classifies the isometry classes of symmetric bundles with the same rank of $q$,  we  may attach to $(\rho, X)$ a new symmetric bundle $(V_X, q_X)$,   known as the  Fr\"ohlich twist of $(V, q)$. The results consist of various comparison formulas,  in the \'etale   cohomology ring $H^*(Y_{et},{\bf Z}/2{\bf Z})$,  which relate the Hasse-Witt invariants of $(V,q)$ to those of its  twisted form $(V_X,q_X)$. One of the principal aims of this paper is to  show that all these comparison formulas, together with a number of new results,  can be immediately  deduced by pulling back from a single equation  which sits  in the cohomology ring $H^*(B_{{\bf O}(q)},{\bf Z}/2{\bf Z})$ of the classifying topos
$B_{{\bf O}(q)}$ and which is independent of any choice of particular orthogonal representation and particular torsor. 

In \cite{SGA4} and \cite{Giraud} Grothendieck and Giraud introduced  the notion of the classifying topos of a group object in given topos, and they suggested that it could be used in the theory of characteristic classes in algebraic geometry. Building on their insight, we shall prove our main theorem using both abelian and non-abelian Cech cohomology of the classifying topos $B_{{\bf O}(q)}$ of the group-scheme ${\bf O}(q)$. To this end, the first part of this paper, namely Sections 2 and 3,  is devoted to the study of some  basic properties of the classifying topos $B_G$ of a $Y$-group-scheme $G$ which is defined as follows. Let $Y_{fl}$ denote the category of sheaves of sets on the big fppf-site of $Y$ and let $yG$ denote the sheaf of groups of  $Y_{fl}$ represented by $G$. Then $B_G$ is simply the category of objects $\mathcal{F}$ of $Y_{fl}$ endowed with a left action of $yG$. We may view a $Y$-scheme as an object of $Y_{fl}$ and write $G$ for $yG$.

In Section 2 we prove the well-known result due to Grothendieck and Giraud  stating that there is a canonical equivalence
\begin{equation}\label{equiintro}
{\bf Homtop}_{Y_{fl}}(\mathcal{E},B_G)^{op}\stackrel{\sim}{\longrightarrow} {\bf Tors}(\mathcal{E},f^*G)
\end{equation} 
where $f:\mathcal{E}\rightarrow Y_{fl}$ is any topos over $Y_{fl}$, ${\bf Homtop}_{Y_{fl}}(\mathcal{E},B_G)$ is the category of morphisms of $Y_{fl}$-topoi from $\mathcal{E}$ to $B_G$ and ${\bf Tors}(\mathcal{E},f^*G)$ is  the groupoid of $f^*G$-torsors in $\mathcal{E}$. This result is proposed as an exercise in \cite{SGA4}. We include a proof here because the explicit description of the functor (\ref{equiintro}) and its quasi-inverse is used  repeatedly throughout this work. 

Section 3 is devoted to the study of the cohomology of $B_G$. First we show that there is a canonical isomorphism
\begin{equation}\label{compare-intro}
H^*(B_G,\mathcal{A})\simeq H^*({\bf B}G_{et},\mathcal{A})
\end{equation}
where the right-hand side denotes the \'etale cohomology of the simplicial scheme ${\bf B}G$ (as defined in \cite{F}) and $\mathcal{A}$ is an 
abelian object of $B_G$ which is representable by a smooth $Y$-scheme supporting a $G$-action.   For a constant group $G$, we show that the cohomology of $B_G$ (or more generally the cohomology of $B_G/X$ for any $Y$-scheme $X$ with a $G$-action) computes Grothendieck's mixed cohomology (see \cite{Grothendieck-chern-repres} 2.1). There are several interesting spectral sequences and exact sequences which relate the cohomology of $B_G$ to other kinds of cohomology. For example, for any commutative group scheme $A$ endowed with a left action of $G$, there is an exact sequence
\begin{equation}\label{ges-intro}
0\rightarrow H^0(B_G,\mathcal{A})\rightarrow H^0(Y_{fl},\mathcal{A})\rightarrow \mathrm{Crois}_Y(G,A)\rightarrow H^1(B_G,\mathcal{A})
\end{equation}
\begin{equation*}
\rightarrow H^1(Y_{fl},\mathcal{A})\rightarrow \mathrm{Ext}_{Y}(G,A)\rightarrow
H^2(B_G,\mathcal{A})\rightarrow H^2(Y_{fl},\mathcal{A})
\end{equation*}
where $\mathcal{A}=y(A)$, $\mathrm{Crois}_Y(G,A)$ is the group of crossed homomorphisms from $G$ to $A$ (which is just $\mathrm{Hom}_Y(G,A)$ if $G$ acts trivially on $A$) and $\mathrm{Ext}_{Y}(G,A)$ is the group of extensions $1\rightarrow A\rightarrow \tilde{G}\rightarrow G\rightarrow 1$ inducing the given $G$-action on $A$.

We shall also establish the existence of a Hochschild-Serre spectral sequence in this context. Let $1\rightarrow N\rightarrow G\rightarrow G/N\rightarrow 1$ be an exact sequence of $S$-group-schemes (w.r.t. the $\mathrm{fppf}$-topology). Then, for any abelian object $\mathcal{A}$ of $B_G$, there is a natural $G/N$--action on the cohomology $H^j_{S}(B_N,\mathcal{A})$ of $\mathcal{A}$ with values in $S_{fl}$ (see Definition \ref{def-relative-cohomology}) and we have the spectral sequence
$$H^i(B_{G/N},H^j_{S}(B_N,\mathcal{\mathcal{A}}))\Longrightarrow H^{i+j}(B_{G},\mathcal{A}).$$
The five-term exact sequence induced by this spectral sequence reads as follows:
\begin{equation}\label{es-hs}
0\rightarrow H^1(B_{G/N},\mathcal{A})\rightarrow H^1(B_G,\mathcal{A})\rightarrow
H^0(B_{G/N},\underline{\textrm{Hom}}(N,\mathcal{A}))
\end{equation}
$$\rightarrow H^2(B_{G/N},\mathcal{A})\rightarrow H^2(B_G,\mathcal{A})$$
where we assume for simplicity that $\mathcal{A}$ is given with trivial $G$-action.

This then concludes our description of the first part of the article, which is of a relatively general nature.

The aim of the second part of this paper, which starts from Section 4, is to apply the  general results of the first part to the study of symmetric bundles and their invariants. From Section 4 on, we fix a scheme $Y$ in which $2$ is invertible and a symmetric bundle $(V, q)$ on $Y$, i.e. a locally free $\mathcal{O}_Y$-module $V$ of rank $n$ endowed with a non-degenerate bilinear form $V\otimes_{\mathcal{O}_Y}V\rightarrow \mathcal{O}_Y$. A special case is given by $(\mathcal{O}_{Y}^{n}, t_{n}=x_{1}^{2}+...+x_{n}^{2})$, the standard form of rank $n$, and ${\bf O}(n)$ is defined as the orthogonal group for this form. The isomorphism (\ref{compare-intro}), together with a fundamental result of Jardine (see \cite{J5} Theorem 2.8), yields a canonical identification of $A$-algebras
$$H^*(B_{{\bf O}(n)}, {\bf Z}/2{\bf Z})\simeq H^*({\bf  BO}(n)_{et}, {\bf Z}/2{\bf Z})\simeq A[HW_{1},...,HW_{n}]$$
where $HW_{i}$ has degree $i$ and $A:=H^*(Y_{et},{\bf Z}/2{\bf Z})$ is the \'etale cohomology ring of $Y$. The symmetric bundle $(V,q)$ provides
us with the object ${\bf Isom}(t_n,q)$ of $Y_{fl}$, which naturally supports a right action of ${\bf O}(n)$ and a left action of ${\bf O}(q)$. It is easily seen that ${\bf Isom}(t_n,q)$ is in fact an ${\bf O}(n)$-torsor of $B_{{\bf O}(q)}$. It follows from (1) that  this torsor may be viewed  as a $Y_{fl}$-morphism
$$T_q:B_{{\bf O}(q)}\longrightarrow B_{{\bf O}(n)} . $$
This morphism is actually an equivalence of $Y_{fl}$-topoi. Note that, however, the groups ${\bf O}(n)$ and ${\bf O}(q)$ are not isomorphic in general. Indeed, $T_q$ has a quasi-inverse
$$T^{-1}_q:B_{{\bf O}(n)}\stackrel{\sim}{\longrightarrow} B_{{\bf O}(q)}$$
given by the ${\bf O}(q)$-torsor ${\bf Isom}(q,t_n)$ of $B_{{\bf O}(n)}$. This yields a canonical isomorphism of $A$-algebras
$$H^*(B_{{\bf O}(q)}, {\bf Z}/2{\bf Z})\simeq A[HW_{1}(q), ...,HW_{n}(q)]$$
where $HW_{i}(q):=T_q^*(HW_{i})$ has degree $i$.  The classes $HW_i(q), 1 \leq i\leq n$, will be called the universal Hasse-Witt invariants of $q$. The proof of the analogous fact in the simplicial framework seems more technical (see \cite{J4}, where it is shown that the simplicial sheaves associated to ${\bf O}(n)$ and ${\bf O}(q)$ are weakly equivalent). We may now view the object ${\bf Isom}(t_n,q)$ as an ${\bf O}(n)$-torsor of $Y_{fl}$; it therefore yields a map
$$\{q\}:Y_{fl}\longrightarrow B_{{\bf O}(n)}$$
which, incidentally,  determines $q$. The classical Hasse-Witt invariants of $q$ are defined by:
$$w_i(q):= \{q\}^*(HW_i)\in H^i(Y,{\bf Z}/2{\bf Z}).$$

 We can attach to   $(V,q)$  both a canonical map $\textrm{det}_{{\bf O}(q)} :{\bf O}(q)\rightarrow {\bf Z}/2{\bf Z}$ and also the central group-extension
\begin{equation}\label{Cq-Intro}
1\rightarrow {\bf Z}/2{\bf Z}\rightarrow\widetilde{{\bf O}}(q)\rightarrow {\bf O}(q)\rightarrow 1
\end{equation}
derived from the Clifford algebra and the Clifford group of $q$. It turns out, by considering the sequence  (\ref{ges-intro}), that the map $\mathrm{det}_{{\bf O}(q)}$ yields a cohomology class $\mathrm{det}[q]\in H^1(B_{{\bf O}(q)}, {\bf Z}/2{\bf Z})$, while the extension (\ref{Cq-Intro}) gives  us  a cohomology class $[C_q]\in H^2(B_{{\bf O}(q)}, {\bf Z}/2{\bf Z})$. The  main result of the second part of the paper provides an explicit expression of $\mbox{det}[q]$ and $[C_{q}]$ as polynomials in $HW_{1}(q)$ and $HW_{2}(q)$ with coefficients in $A$ expressed in terms of $w_1(q),w_2(q)\in A$.
To be more precise, we shall prove the following:
\begin{thm}\label{mainthmintro} Let $Y$ be a scheme in which $2$ is invertible and let $(V, q)$ be a symmetric bundle on $Y$. Assume that $Y$ is the disjoint union of its connected components. Then we have the equalities
$$\textrm{\emph{det}}[q]=w_{1}(q)+HW_{1}(q)$$
and
$$[C_q]=(w_{1}(q)\cdot w_{1}(q)+w_{2}(q))+w_{1}(q)\cdot HW_{1}(q)+HW_{2}(q)$$
in the polynomial ring $$H^*(B_{{\bf O}(q)}, {\bf Z}/2{\bf Z})\simeq A[HW_{1}(q), ...,HW_{n}(q)].$$
\end{thm}
Surprisingly, this theorem is new even in the case where  $Y=\mathrm{Spec}(K)$ is the spectrum of a field.

Section 5  is devoted to the proof of this result. The identity in degree $1$ is proved using simple computations with torsors. The proof of the identity in degree $2$ is more involved and is based on computations with Cech cocycles. A technical reduction makes use of the exactness of the  sequence (\ref{es-hs}) derived from the Hochschild--Serre spectral sequence for the group--extension (\ref{Cq-Intro}). 

Theorem \ref{mainthmintro} is the source of numerous comparison formulas, which are either new results or generalizations of known results (see \cite{EKV}, \cite{CNET1}) by use the following method: for any topos $\mathcal{E}$ given with an ${\bf O}(q)$--torsor, we have the canonical map $f:\mathcal{E}\rightarrow B_{{\bf O}(q)}$, and we obtain comparison formulas in $H^*(\mathcal{E},{\bf Z}/2{\bf Z})$ by applying the functor $f^*$ to the universal comparison formulas of Theorem \ref{mainthmintro}. For example, given an ${\bf O}(q)$--torsor $\alpha$ on $Y$, we consider the map $f:\mathcal{E}=Y_{fl}\rightarrow B_{{\bf O}(q)}$, which classifies $\alpha$, and we thereby obtain an identity in $H^i(Y_{fl},{\bf Z}/2{\bf Z})$ for $i=1,2$. Our result,  Corollary \ref{cor-one},  generalizes a result of Serre to any base scheme $Y$. A second example is provided by an orthogonal representation $\rho:G\rightarrow {\bf O}(q)$ of a $Y$-group-scheme $G$: here we consider the map $B_{\rho}:B_G\rightarrow B_{{\bf O}(q)}$ and the obtain identities in $H^*(B_G,{\bf Z}/2{\bf Z})$. The result that we obtain, Corollary \ref{cor-two},  is new - even in the case when $G$ is a constant (= discrete) group. A third example is provided by an orthogonal representation $\rho:G\rightarrow {\bf O}(q)$ and a $G$-torsor $X$ on $Y$: in this case we may consider the map
$$Y_{fl}\stackrel{X}{\longrightarrow} B_G\stackrel{B_{\rho}}{\longrightarrow} B_{{\bf O}(q)}$$
in order to obtain identities in $H^*(Y_{fl},{\bf Z}/2{\bf Z})$ (see Corollary \ref{cor-three}); the result that we obtain essentially generalizes the Theorem of  Fr\"ohlich-Kahn-Snaith (see \cite{J4} Theorem 2.4). Our result  is  general in the sense that $Y$ is an arbitrary scheme (except that $2$ must be invertible) and $G$ is not assumed to be constant. However, we should remark that we do not obtain  a  complete analogue of \cite{J4},  Theorem 1.6. ii), when $G$ is a non-constant group scheme (see remark of Section 6).  In Section 7 the twisting formulas of Corollary \ref{cor-three} are generalized to the case where $X$ is a $G$-cover of $Y$ with odd ramification (as defined in 7.1).  Twists of symmetric bundles by $G$-torsors (for non-constant group scheme $G$)  appear naturally in situations of arithmetic interest (for instance, the trace form of any finite and separable algebra is a  twist of the standard form).  The formulas in Corollary \ref{cor-three} provide us with tools to deal with the embedding problems associated to torsors; it is our intention to return to these questions in a forthcoming paper.

 We conclude this Introduction by remarking that, for the convenience of the reader, we have gathered together some of the basic material of topos theory in an appendix.

\tableofcontents

\section{The classifying topos of a group scheme}

\subsection{The definition of $B_G$}
Let $S$ be a scheme. We consider the category of $S$-schemes ${\bf Sch}/S$ endowed  with the \'etale topology or fppf-topology. Recall that a fundamental system of covering families for the fppf-topology is given by the surjective families $(f_i:X_i\rightarrow X)$ consisting of  flat morphisms which are locally finitely  presented.  The corresponding sites are denoted by $({\bf Sch}/S)_{et}$ and $({\bf Sch}/S)_{fppf}$.
The big flat topos and  the big \'etale topos of $S$ are defined as the categories of sheaves of sets on these sites:
$$S_{fl}:=\widetilde{({\bf Sch}/S)}_{fppf}\mbox{ and }S_{Et}:=\widetilde{({\bf Sch}/S)_{et}}.$$
Here $\widetilde C$ denotes the category of sheaves on a site $C$. The identity $({\bf Sch}/S)_{et}\rightarrow ({\bf Sch}/S)_{fppf}$ is a continuous functor. It yields a canonical morphism of topoi
$$i:S_{fl}\rightarrow S_{Et}.$$
This map is an embedding, i.e. $i_*$ is fully faithful;  hence $S_{fl}$ can be identified with the full subcategory of $S_{Et}$ consisting of  big \'etale sheaves on $S$ which are sheaves for the $\textrm{fppf}$-topology. The $\textrm{fppf}$-topology on the category ${\bf Sch}/S$ is subcanonical (hence so is the \'etale topology). In other words, any representable presheaf is a sheaf. It follows that the Yoneda functor yields a fully faithful functor
\begin{equation}\label{yoneda}
y:{\bf Sch}/S\rightarrow S_{fl}.
\end{equation}
For any $S$-scheme $Y$, we consider the slice topos $S_{fl}/yY$ (i.e. the category of maps $\mathcal{F}\rightarrow yY$ in $S_{fl}$). We have a canonical equivalence \cite{SGA4} [IV, Section 5.10]
$$S_{fl}/yY:=\widetilde{({\bf Sch}/S)_{fl}}/yY\simeq \widetilde{({\bf Sch}/Y)_{fl}}=:Y_{fl}.$$
The Yoneda functor (\ref{yoneda}) commutes with  projective limits. In particular it preserves products and the final object; hence a group scheme $G$ over $S$ represents a group object $yG$ in $S_{fl}$, i.e. a sheaf of groups on the site $({\bf Sch}/S)_{fl}$.
\begin{definition}
The \emph{classifying topos $B_G$} of the $S$-group scheme $G$ is defined as the category of objects in $S_{fl}$ given with a left action of $yG$.
The \emph{\'etale classifying topos $B^{et}_G$} of the $S$-group scheme $G$ is defined as the category of objects in $S_{Et}$ given with a left action of $yG$.
\end{definition}

More explicitly, an object of $B_G$ (resp. of $B_G^{et}$) is a sheaf $\mathcal{F}$ on ${\bf Sch}/S$ for the fppf-topology (resp. for the \'etale topology) such that, for any $S$-scheme $Y$, the set $\mathcal{F}(Y)$ is endowed  with a $G(Y)$--action
$$\textrm{Hom}_S(Y,G)\times\mathcal{F}(Y)\rightarrow \mathcal{F}(Y)$$
which is functorial  in $Y$. We have a commutative diagram (in fact a pull--back) of topoi
\[\xymatrix{
B_{G}\ar[d]\ar[r]& B_G^{et}\ar[d] \\
S_{fl}\ar[r]^i&S_{Et}
}\]
where the vertical morphisms are defined as in (\ref{map-pi}) below.

\subsection{Classifying torsors}
More generally, let $\mathcal{S}$ be any topos and let $G$ be any group in $\mathcal{S}$. We denote by ${\bf Tors}(\mathcal{S},G)$ the category of $G$--torsors in $\mathcal{S}$. Recall that a (right) $G$--torsor in $\mathcal{S}$ is an object $T$ endowed with a right action $\mu:T\times G\rightarrow T$ of $G$ such that:
\begin{itemize}
\item the map $T\rightarrow e_{\mathcal{S}}$ is an epimorphism, where $e_{\mathcal{S}}$ is the final object of $\mathcal{S}$;
\item the map $(p_1,\mu):T\times G\rightarrow T\times T$ is an isomorphism, where $p_1$ is the projection on the first component.
\end{itemize}
An object $T$ in $\mathcal{S}$,  endowed  with a right $G$--action,  is a $G$--torsor if and only if there exists an epimorphic family $\{U_i\rightarrow e_{\mathcal{S}}\}$ such that the base change $U_i\times T$ is isomorphic to the trivial $(U_i\times G)$--torsor in $\mathcal{S}/U_i$, i.e. if there is a $(U_i\times G)$--equivariant isomorphism
$$U_i\times T\simeq U_i\times G$$
defined over $U_i$, where $U_i\times G$ acts on itself by right multiplication.

The classifying topos $$B_G:=B_G(\mathcal{S})$$
is the category of left $G$-objects in $\mathcal{S}$. The fact that $B_G$ is a topos follows easily from Giraud's axioms;  the fact that $B_G$ classifies $G$--torsors is recalled below. We denote by
\begin{equation}\label{map-pi}
\pi: B_G\rightarrow \mathcal{S}
\end{equation}
the canonical map: the inverse image functor $\pi^*$ sends an object $\mathcal{F}$ in $\mathcal{S}$ to $\mathcal{F}$ with trivial $G$--action. Indeed, $\pi^*$ commutes with arbitrary inductive and projective limits;  hence  $\pi^*$ is the inverse image of a morphism of topoi $\pi$. In particular,  the group $\pi^*G$ is given by the trivial action of $G$ on itself. Let $E_G$ denote  the object of $B_G$ defined by the action of $G$ on itself by left multiplication. Then the map
$$E_G\times \pi^*G\rightarrow E_G,$$
given by right multiplication is a morphism of $B_G$ (i.e. it is $G$--equivariant). This action provides $E_G$ with the structure of a right $\pi^*G$--torsor in $B_G$.

If $f:\mathcal{E}\rightarrow\mathcal{S}$ and $f':\mathcal{E}'\rightarrow\mathcal{S}$ are topoi over the base topos $\mathcal{S}$,  then we denote by ${\bf Homtop}_{\mathcal{S}}(\mathcal{E},\mathcal{E}')$ the category of $\mathcal{S}$--morphisms from $\mathcal{E}$ to $\mathcal{E}'$. An object of this category is a pair $(a,\alpha)$ where $a:\mathcal{E}\rightarrow\mathcal{E}'$ is a morphism and $\alpha: f'\circ a\simeq f$ is an isomorphism, i.e. an isomorphism of functors $\alpha:f'_*\circ a_*\simeq f_*$, or equivalently, an isomorphism of functors $\alpha:f^*\simeq a^*\circ f'^* $. A map $\tau:(a,\alpha)\rightarrow(b,\beta)$ in the category ${\bf Homtop}_{\mathcal{S}}(\mathcal{E},\mathcal{E}')$ is a morphism (of morphism of topoi) $\tau:a\rightarrow b$ compatible with $\alpha$ and $\beta$. More precisely,  we require that  the following triangle commutes:
\[\xymatrix{
f'_*\circ a_*\ar[r]^{f'_*(\tau)}\ar[dr]_{\alpha}&f'_*\circ b_*\ar[d]^{\beta} \\
&f_*}\]

\begin{thm}\label{Thm-classifying--torsors}
Let $f:\mathcal{E}\rightarrow\mathcal{S}$ be a morphism of topoi and let $B_G$ be the classifying topos of a group $G$ in $\mathcal{S}$.
The functor
$$
\fonc{\Psi}{{\bf Homtop}_{\mathcal{S}}(\mathcal{E},B_G)^{op}}{{\bf Tors}(\mathcal{E},f^*G)}{(a,\alpha)}{a^*E_G}
$$
is an equivalence of categories.
\end{thm}

\begin{proof}
(1) Let $(a,\alpha)$ be an object of ${\bf Homtop}_{\mathcal{S}}(\mathcal{E},B_G)$. The functor $a^*$ is left exact and $E_G$ is a right $\pi^*G$--torsor in $B_G$. It follows that $a^*E_G$ is given with a right action of $a^*\pi^*G$ such that the map
$a^*E_G\times a^*\pi^*G\rightarrow a^*E_G\times a^*E_G$
is an isomorphism. Moreover, $a^*$ commutes with arbitrary inductive limits;  hence it preserves epimorphisms,  so that the map $a^*E_G\rightarrow e_{\mathcal{E}}$ is epimorphic. More generally, the inverse image of a morphism of topoi preserves torsors. Hence $a^*E_G$ is a right $a^*\pi^*G$--torsor in $\mathcal{E}$. Then $a^*E_G$ is a right $f^*G$--torsor via the given isomorphism $\alpha(G):f^*G\simeq a^*\pi^*G$. A morphism $\tau:(a,\alpha)\rightarrow(b,\beta)$ yields a morphism $\tau(E_G):b^*E_G\rightarrow a^*E_G$ which is $f^*G$--equivariant thanks to the above commutative triangle. Thus the functor $\Psi$ is indeed well-defined.

(2) Next we  define a functor
$$\Phi:{\bf Tors}(\mathcal{E},f^*G)\rightarrow {\bf Homtop}_{\mathcal{S}}(\mathcal{E},B_G)^{op}.
$$
Let $T$ be a right $f^*G$--torsor in $\mathcal{E}$. We denote by $p:\mathcal{S}\rightarrow B_G$ the canonical morphism ($p^*$ sends an object $\mathcal{F}$ with a $G$--action to $\mathcal{F}$). For any $\mathcal{F}$ in $B_G$, $p^*\mathcal{F}$ carries a left action of the group $G$  and so  $f^*G$ acts on $f^*p^*\mathcal{F}$. Consider the diagonal action of $f^*G$ on the product $T\times f^*p^*\mathcal{F}$ (here $f^*G$ acts on the right on the first factor and on the left on the second factor). The quotient
$$a_T^*\mathcal{F}:=T\wedge^{f^*G} f^*p^*\mathcal{F}:= (T\times f^*p^*\mathcal{F})/f^*G$$
is well defined in $\mathcal{E}$. We obtain a functor
$$a_T^*:B_G(\mathcal{S})\rightarrow\mathcal{E}. $$
Our aim is to show that $a_{T}^{*}$ is the inverse image functor of a morphism of topoi. For this it is enough to show that $a_T^*$ commutes with finite projective limits and arbitrary inductive limits. Inductive limits are universal in a topos (so that $T\times f^*p^*(-)$ commutes with inductive limits) and inductive limits commute with themselves (hence with $(-)/G$). Hence  $a_T^*$ commutes with inductive limits. It remains to show that $a_T^*$ commutes with finite projective limits. It is enough to prove that the functor
$$
\fonc{q_T^*}{B_{f^*G}(\mathcal{E})}{\mathcal{E}}{X}{T\wedge^{f^*G}X:=(T\times X)/f^*G}
$$
is left exact, i.e. that it commutes with finite projective limits. It is enough to establish this fact locally in $\mathcal{E}$. Since a torsor is locally trivial, one may assume that $T$ is the trivial torsor $T_0=f^*G$. The canonical map in $\mathcal{E}$ (induced by the left action $f^*G\times X\rightarrow X$ of $f^*G$ on $X$):
$$q_{T_0}^*X=f^*G\wedge^{f^*G} X:=(f^*G\times X)/f^*G\rightarrow X$$
is an isomorphism, which is functorial in the object $X$ of $B_{f^*G}(\mathcal{E})$. The functor $q_{T_0}^*$ forgets the $f^*G$--action, and thus coincides with the inverse image functor of the canonical map $\mathcal{E}\rightarrow B_G$.

Hence,  as required,  we see that $a_T^*$ commutes with finite projective limits and arbitrary inductive limits and therefore
is the inverse image of a morphism $a_T:\mathcal{E}\rightarrow B_G(\mathcal{S})$. For any object $\mathcal{L}$ of $\mathcal{S}$, one has a natural isomorphism
$$a_T^*\pi^*\mathcal{L}=(T\times f^*p^*\pi^*\mathcal{L})/f^*G=(T\times f^*\mathcal{L})/f^*G\simeq(T/f^*G)\times f^*\mathcal{L}\simeq f^*\mathcal{L}$$
since $f^*G$ acts trivially on $f^*\mathcal{L}$. We therefore obtain an isomorphism
$$\alpha_T:f^*\rightarrow a_T^*\pi^*$$
and the functor $\Phi:T\mapsto(a_T,\alpha_T)$ is well-defined. Indeed, an $f^*G$--morphism of torsors $T\rightarrow T'$ induces
a map
$$a_T^*(\mathcal{F})=T\wedge^{f^*G} f^*p^*\mathcal{F}\rightarrow T'\wedge^{f^*G} f^*p^*\mathcal{F}=a_{T'}^*(\mathcal{F})$$ which is
functorial in $\mathcal{F}$ and which in turn gives a morphism $(a_{T'},\alpha_{T'})\rightarrow(a_T,\alpha_T)$.

(3) It  is straightforward to show that $\Psi$ and $\Phi$ are quasi-inverse to each other.
\end{proof}

\begin{cor}\label{Diaconescu-Thm}
Let $f:\mathcal{E}\rightarrow\mathcal{S}$ be a morphism of topoi and let $G$ be a group in $\mathcal{S}$.
Then the following square
\[\xymatrix{
B_{f^*G}(\mathcal{E})\ar[d]\ar[r]& B_G(\mathcal{S})\ar[d]_{\pi} \\
\mathcal{E}\ar[r]^f&\mathcal{S}}\]
is a pull--back.
\end{cor}
\begin{proof}
 The square of the proposition is clearly commutative. Hence for any topos $\mathcal{E}'$
there is a canonical functor 
\begin{equation}\label{functhere}
{\bf Homtop}(\mathcal{E}',B_{f^*G}(\mathcal{E}))\longrightarrow
{\bf Homtop}(\mathcal{E}',B_{G}(\mathcal{S}))\times_{{\bf Homtop}(\mathcal{E}',\mathcal{S})}
{\bf Homtop}(\mathcal{E}',\mathcal{E}).
\end{equation}
This functor has a quasi-inverse which is defined as follows. Let $a:\mathcal{E}'\rightarrow B_G(\mathcal{S})$ and $b:\mathcal{E}'\rightarrow\mathcal{E}$ be morphisms,  and let $\sigma:\pi\circ a\simeq f\circ b$ be an isomorphism. It follows from Theorem \ref{Thm-classifying--torsors}  that the $\mathcal{S}$--morphism
 $$(a,\sigma):\mathcal{E}'\rightarrow B_G(\mathcal{S})$$ provides  a $b^*(f^*G)$--torsor in $\mathcal{E}'$,
 which in turn gives a morphism
$$a\times b: \mathcal{E}'\rightarrow B_{f^*G}(\mathcal{E})\ .$$
It is straightforward to check that this yields a quasi-inverse to (\ref{functhere}). Hence the functor (\ref{functhere}) is an equivalence for any $\mathcal{E}'$. This then proves the result.
\end{proof}

\subsection{Torsors under group scheme actions}

\begin{cor}
Let $G$ be a group scheme over $S$ and let $Y$ be an $S$-scheme. There are  canonical equivalences
\begin{eqnarray*}
{\bf Tors}(Y_{fl},G_Y)^{op}&\simeq&{\bf Homtop}_{{Y}_{fl}}(Y_{fl},B_{G_Y})\\
&\simeq&{\bf Homtop}_{{S}_{fl}}(Y_{fl},B_G)\\
&\simeq&{\bf Homtop}_{{S}_{Et}}(Y_{fl},B^{et}_G).\\
\end{eqnarray*}
\end{cor}
\begin{proof}
The first equivalence follows directly from the previous theorem, and so does the second equivalence,  since the inverse image of $y(G)$ along the morphism
$Y_{fl}\rightarrow S_{fl}$ is the sheaf on $Y$ represented by $G_Y=G\times_SY$.
The third equivalence follows from the canonical equivalence
$$B_G:=B_G(S_{fl})\simeq S_{fl}\times_{S_{Et}}B_G(S_{Et})$$
given by Corollary \ref{Diaconescu-Thm}.
\end{proof}

The key case of interest is provided by an $S$-group-scheme $G$ which is flat and locally of finite presentation over $S$.
For an $S$-scheme $Y$, denote by ${\bf Tors}(Y,G_Y)$ the category of $G_Y$--torsors of the scheme $Y$;
i.e. the category of maps $T\rightarrow Y$ which are faithfully flat and locally of finite presentation, supporting  a right action $T\times_YG_Y\rightarrow T$
such that the morphism $T\times_YG_Y\rightarrow T\times T$ is an isomorphism of $T$-schemes. The Yoneda embedding yields a fully faithful functor
$$y:{\bf Tors}(Y,G_Y)\rightarrow {\bf Tors}(Y_{fl},G_Y).$$
This functor is not an equivalence (i.e. it is not essentially surjective) in general. However,  it is an equivalence in certain special cases, see  \cite{Milne}[III, Theorem 4.3].   In particular,  this is the case  when $G$ is affine over $S$.
\begin{cor}
Let $Y$ be an $S$-scheme. Let $G$ be a flat   group scheme over $S$ which is locally of finite type. Assume that $G$ is affine over $S$.  Then
we have an  equivalence of categories
\begin{eqnarray*}
{\bf Tors}(Y,G_Y)^{op}&\simeq&{\bf Homtop}_{{S}_{fl}}(Y_{fl},B_G).\\
\end{eqnarray*}
\end{cor}

\begin{notation}
Let $G$ be a flat   affine  group scheme over $S$ which is of finite type and let $Y$ be an $S$-scheme. We have  canonical equivalences
\begin{eqnarray*}
{\bf Tors}(Y,G_Y)^{op}&\simeq&{\bf Tors}(Y_{fl},G_Y)^{op}\\
&\simeq&{\bf Homtop}_{{S}_{fl}}(Y_{fl},B_G(S_{fl}))\\
&\simeq&{\bf Homtop}_{{S}_{fl}}(Y_{fl},B_{G_Y}(Y_{fl}))
\end{eqnarray*}
If a $Y$-scheme $T$ is a $G_Y$--torsor over $Y$, then  we again denote by $T$ the object of ${\bf Tors}(Y_{fl},G_Y)^{op}$, and also denote by $T$ 
$$T:Y_{fl}\rightarrow B_G(S_{fl})$$
the corresponding object of ${\bf Homtop}_{{S}_{fl}}(Y_{fl},B_G(S_{fl}))$; and similarly we denote by $T$ 
$$T:Y_{fl}\rightarrow B_{G_Y}(Y_{fl})$$
the corresponding object of  ${\bf Homtop}_{{Y}_{fl}}(Y_{fl},B_{G_Y}(Y_{fl}))$.
\end{notation}

\subsection{The big topos $B_G/X$ of $G$--equivariant sheaves}

Let $X$ be an $S$-scheme endowed with a left action (over $S$) of $G$. Then $yX$ is a sheaf on $({\bf Sch}/S)_{fppf}$ with a left action of $yG$ (since $y$ commutes with finite projective limits). The resulting object of $B_G$ will be denoted by $y(G,X)$, or just by $X$ if there is no risk of ambiguity. The slice category $B_G/X$ is a topos, which we refer to as  the topos of $G$--equivariant sheaves on $X$. This terminology is justified by the following observation: an object of $B_G/X$ is given by an object $\mathcal{F}\rightarrow X$ of $S_{fl}/X\simeq X_{fl}$ (i.e. a sheaf on the fppf-site of $X$), endowed with an action of $yG$ such that the structure map $\mathcal{F}\rightarrow X$ is $G$--equivariant. We have a (localization) morphism
$$f:B_G/X\rightarrow B_G$$
whose inverse image maps an object $\mathcal{F}$ of $B_G$ to the ($G$--equivariant) projection $\mathcal{F}\times X\rightarrow X$, where $yG$ acts diagonally on $\mathcal{F}\times X$.

Let $Y$ be an $S$-scheme with trivial $G$--action, and consider the topos $B_G/Y$. We denote by $G_Y:=G\times_SY$ the base change of the $S$-group scheme $G$ to $Y$ and we consider the classifying topos $B_{G_Y}$ of the $Y$-group scheme ${G_Y}$. Recall that $B_{G_Y}$ is the category of $y(G_Y)$--equivariant sheaves on $({\bf Sch}/Y)_{fl}$. The following result shows that the classifying topos $B_G$ behaves well with respect to base change.
\begin{prop}\label{prop-BGY}
If $G$ acts trivially on an $S$-scheme $Y$, then there is a canonical equivalence
$$B_{G_Y}\simeq B_G/Y.$$
\end{prop}
\begin{proof} Let $\pi:B_G\rightarrow S_{fl}$  denote  the canonical map. On the one hand by  \cite{SGA4}[IV, Section 5.10], the square
\[\xymatrix{
B_G/\pi^*(yY)\ar[r]\ar[d]&S_{fl}/yY\ar[d] \\
B_G\ar[r]^\pi&S_{fl}}\]
is a pull--back. Note that $\pi^*(yY)$ is given by the trivial action of $G$ on $Y$ so that $B_G/\pi^*(yY)=B_G/Y$. On the other hand, the square
\[\xymatrix{
B_{g^*(yG)}(S_{fl}/yY)\ar[r]\ar[d]&B_G \ar[d] \\
S_{fl}/yY\ar[r]^g&S_{fl}}\]
is also a  pull--back by Corollary \ref{Diaconescu-Thm}. Hence we have  canonical equivalences
$$B_G/Y\simeq B_G\times_{S_{fl}}S_{fl}/yY\simeq B_{S_{fl}/yY}(g^*(yG)).$$
Here the first equivalence (respectively the second) is induced by the first (respectively the second) pull--back square above.
Finally, we have  $g^*(yG)=y(G\times_SY)=y(G_Y)$ in the topos $S_{fl}/yY\simeq Y_{fl}$; hence we obtain
$$B_{S_{fl}/yY}(g^*(yG))\simeq B_{Y_{fl}}(yG_Y)=B_{G_Y}.$$
\end{proof}

\section{Cohomology of group schemes}

The cohomology of a Lie group can be defined as the cohomology of its classifying space. Analogously, Grothendieck and
Giraud defined the cohomology of a group object $G$ in a topos as the cohomology of its classifying topos $B_G$.

\begin{definition} Let $G$ be an $S$-group scheme and let $\mathcal{A}$ be an abelian object of
 $B_G=B_{yG}(S_{fl})$. The cohomology of the $S$-group scheme $G$ with coefficients in  $\mathcal{A}$ is defined as
$$H^i(G,\mathcal{A}):=H^i(B_G,\mathcal{A}).$$
\end{definition}
Note that any commutative group scheme $\mathcal{A}$ over $S$,  endowed with an action of $G$,  gives rise to an abelian object in $B_G$.
Note also that, in the case where the $S$-group scheme $G$ is trivial (i.e. $G=S$) the cohomology of $G$ is reduced
to the flat cohomology of $S$.

In this section we show that the cohomology of the classifying topos $B_G$ of a group-scheme with coefficients in a smooth commutative group scheme coincides with the \'etale cohomology of the simplicial classifying scheme ${\bf B}G$. This fact holds in the more general situation given by the action of $G$ on a scheme $X$ over $S$.

\begin{definition}  Let $X$ be an $S$-scheme endowed with a left $G$--action. We define the
equivariant cohomology of the pair $(G,X)$ with coeficients in an abelian object $\mathcal{A}$ of $B_G/X$ by:
$$H^i(G,X,\mathcal{A}):=H^i(B_G/X,\mathcal{A}) . $$

\end{definition}

Note that, if $X=S$ is trivial, the equivariant cohomology of the pair $(G,X)$ is just the cohomology of the $S$-group
scheme $G$ as defined before. If the group scheme $G$ is trivial, then the equivariant cohomology of the pair $(G,X)$ is the flat cohomology of the scheme $X$.

\subsection{Etale cohomology of simplicial schemes}
After recalling the notions of simplicial schemes and simplicial topoi we observe that the big and the small \'etale sites of a simplicial scheme have the same cohomology. References for this section are \cite{F}[I,II] and \cite{I2}[VI, Section 5].

The category  $\Delta$  of standard simplices  is the category whose objects are the finite ordered sets $[0, n]=\{0<1<...<n\}$ and whose morphisms  are non-decreasing functions. Any morphism $[0, n]\rightarrow [0, m]$, other than identity, can be written as a composite of degeneracy maps $s^{i}$ and face maps $d^{i}$. Here recall that $s^{i}:[0, n+1]\rightarrow [0, n]$ is the is the unique surjective map with two elements mapping to $i$ and that $d^{i}: [0, n-1]\rightarrow [0, n]$ is the unique injective map avoiding $i$. A simplicial scheme is  a functor $X_\bullet:\Delta^{op}\rightarrow{\bf Sch}$.
As usual we write $X_n:=X_\bullet([0,n])$, $d_i=X_\bullet(d^i)$ for
the face map and $s_i=X_\bullet(s^i)$ for the degeneracy map.  From the functor $X_{\bullet}$ we deduce a simplicial topos
$$\fonc{X_{\bullet,\,et}}{\Delta^{op}}{\bf Top}{[0,n]}{X_{n,\,et}}$$
where $X_{\bullet,\,et}([0,n])=X_{n,\,et}$ is the small \'etale topos of the scheme $X_n$, i.e. the category of sheaves on the category of \'etale $X_n$-schemes endowed with the \'etale topology. Strictly speaking, $X_{\bullet,\,et}$ is a pseudo--functor from $\Delta^{op}$ to the 2-category of topoi (see Section \ref{subsect-2catoftopos}).

Finally  we consider the
total topos $\textsc{Top}(X_{\bullet,\,et})$ associated to this simplicial topos (see \cite{I2} VI. 5.2). Recall that an object
of $\textsc{Top}(X_{\bullet,\,et})$ consists of  the data of objects $F_n$ of $X_{n,\,et}$ together with maps
 $\alpha^*F_m\rightarrow F_n$ in $X_{n,\,et}$ for each $\alpha:[0,m]\rightarrow[0,n]$ in $\Delta$ satisfying
 the natural transitivity condition for a composite map in $\Delta$. The arrows in $\textsc{Top}(X_{\bullet,\,et})$
 are defined in the obvious way. We observe that this category is equivalent to the category of sheaves on the etale site $Et(X_{\bullet})$
 as defined in \cite{F}[I, Def. 1.4].

In a similar way we define the big \'etale simplicial topos associated to $X_{\bullet}$ as follows:
$$
\fonc{X_{\bullet,\,Et}}{\Delta^{op}}{{\bf Top}}{[0,n]}{X_{n,\,Et}}$$
 here $X_{n, Et}$ is the big \'etale topos of the scheme $X_{n}$, i.e. the category of sheaves on the category ${\bf Sch}/X_n$ endowed with the \'etale topology. Then we denote by $\textsc{Top}(X_{\bullet, Et})$ the total topos associated to $X_{\bullet,\, Et}$.

\begin{lem}\label{lem-1}
For any simplicial scheme $X_{\bullet}$, there is a canonical morphism of topoi
$$\iota:\textsc{Top}(X_{\bullet,\,Et})\rightarrow \textsc{Top}(X_{\bullet,\,et}) $$
such that the map
$$H^i(\textsc{Top}(X_{\bullet,\,et}),\mathcal{A})\rightarrow H^i(\textsc{Top}(X_{\bullet,\,Et}),\iota^*\mathcal{A})$$
is an isomorphism for any $i\geq0$ and for any abelian sheaf $\mathcal{A}$ of $\textsc{Top}(X_{\bullet,\,et})$.
\end{lem}
\begin{proof}
The canonical morphism $Y_{Et}\rightarrow Y_{et}$, from the big \'etale topos of a scheme $Y$ to its small \'etale topos, is pseudo--functorial in $Y$;  this follows immediately from the description of this morphism in terms of morphism of sites. Hence we have a morphism of simplicial topoi
$$\iota_{\bullet}:X_{\bullet,\,Et}\rightarrow X_{\bullet,\,et}$$
inducing a morphism between total topoi:
$$\iota:\textsc{Top}(X_{\bullet,\,Et})\rightarrow \textsc{Top}(X_{\bullet,\,et}).$$
Note that we have a commutative diagram of topoi
\[\xymatrix{
X_{n,Et}\ar[r]^{\iota_n}\ar[d]_{f_n}&X_{n,et}\ar[d]_{g_n} \\
\textsc{Top}(X_{\bullet,\,Et})\ar[r]^{\iota}&\textsc{Top}(X_{\bullet,\,et}) \\
}\]
for any object $[0,n]$ of $\Delta$. Here the inverse image $g_n^*$ (resp.  $f_n^*$) of the vertical morphism
$g_n:X_{n,et}\rightarrow X_{\bullet,\,et}$ (resp.
$f_n:X_{n,Et}\rightarrow X_{\bullet,\,Et}$) maps an object $F=(F_n;\,\alpha^*F_m\rightarrow F_n)$
of the total topos $\textsc{Top}(X_{\bullet,\,et})$ (resp. of $\textsc{Top}(X_{\bullet,\,Et})$)
to $F_n\in X_{n,\,et}$ (resp. to $F_n\in X_{n,\,Et}$). Recall that the functors  $g_n^*$ and  $f_n^*$
preserve injective objects. This leads to spectral sequences (see \cite{SGA4} VI Exercice 7.4.15)
\begin{equation}\label{spectral}
E_1^{i,j}=H^j(X_{i,et},\mathcal{A}_i)\Rightarrow H^{i+j}(\textsc{Top}(X_{\bullet,\,et}),\mathcal{A})
\end{equation}
and
\begin{equation}\label{sequence}
'E_1^{i,j}=H^j(X_{i,Et},(\iota^*\mathcal{A})_i)\Rightarrow H^{i+j}(\textsc{Top}(X_{\bullet,\,Et}),\iota^*\mathcal{A})
\end{equation}
for any abelian object $\mathcal{A}$ of $\textsc{Top}(X_{\bullet,\,et})$. The morphism $\iota_{\bullet}$ of simplicial topoi
induces a morphism of spectral sequences from (\ref{spectral}) to (\ref{sequence}). This morphism of spectral sequences
is an isomorphism since the natural map
\begin{equation}\label{an-iso}
H^j(X_{i,et},\mathcal{A}_i)\rightarrow H^j(X_{i,Et},(\iota^*\mathcal{A})_i)=H^j(X_{i,Et},\iota_{i}^*(\mathcal{A}_i))
\end{equation}
is an isomorphism, where the equality on the right-hand side follows from the previous commutative square. Then the map (\ref{an-iso})
is the natural morphism from the cohomology of the small \'etale site of $X_i$ to the cohomology of its big \'etale site, which is
well known to be an isomorphism. Therefore, the induced morphism on abutments
$$H^i(\textsc{Top}(X_{\bullet,\,et}),{\mathcal{A}})\rightarrow H^i(\textsc{Top}(X_{\bullet,\,Et}),\iota^*{\mathcal{A}})$$
is an isomorphism.
\end{proof}
\subsection{Classifying topoi and classifying simplicial schemes}

Let $S$ be a scheme, let $G$ be an $S$-group scheme and let $X$ be an $S$-scheme which supports  a left $G$--action
$G\times_SX\rightarrow X$. We consider the classifying simplicial scheme ${\bf B}(G,X)$ as defined in (\cite{F} Example 1.2). Recall that
$${\bf B}(G,X)_n=G^{n}\times X$$
where  $G^{n}$ is the $n$-fold fiber product of $G$ with itself over $S$ and the product $G^{n}\times X$ is taken over $S$,
with structure maps given in the usual way by using  the multiplication in $G$, the action of $G$ on $X$ and the unit section $S\rightarrow G$.
We consider  the  big \'etale simplicial topos
$$
\fonc{{\bf B}(G,X)_{Et}}{\Delta^{op}}{{\bf Top}}{[0,n]}{(G^n\times X)_{\,Et}}$$
and the total topos $\textsc{Top}({\bf B}(G,X)_{Et})$ as defined
in the previous subsection.
\begin{lem}\label{lem-2}
There is a canonical morphism of topoi
$$\kappa :\textsc{Top}({\bf B}(G,X)_{Et})\rightarrow B^{et}_G/X.$$
\end{lem}
\begin{proof} We let    $\textsc{Desc}({\bf B}(G,X)_{Et})$ be the descent topos. It  is defined as the category of
objects $L$ of $X_{Et}={\bf B}(G,X)_{0,Et}$ endowed with  descent data, i.e. an isomorphism
$a:d_1^*L\rightarrow d_0^*L$ such that
\begin{itemize}
\item $s_0^*(a)=Id_L$
\item $d_0^*(a)\circ d_2^*(a)=d_1^*(a)$ (neglecting the transitivity isomorphisms).
\end{itemize}
Then there is an equivalence of categories
\begin{equation}\label{descent=BG}
\textsc{Desc}({\bf B}(G,X)_{Et})\rightarrow B_G^{et}/X .
\end{equation}
Indeed, for any object $L$ of $S_{Et}/X\simeq X_{Et}$,  descent data on $L$
is equivalent to a left action of $G$ on $L$ such that the structure map $L\rightarrow X$ is $G$--equivariant (see \cite{I2}[VI, Section 8]).
We define the  functor
$$
\fonc{\textrm{Ner}}{\textsc{Desc}({\bf B}(G,X)_{Et})}{\textsc{Top}({\bf B}(G,X)_{Et})}{(L,a)}{\textrm{Ner}(L,a)}
$$
as follows. Let $(L,a)$ be an object of $\textsc{Desc}({\bf B}(G,X)_{Et})$. We consider $$\textrm{Ner}(L,a)_n=(d_0...d_0)^*L$$ in the topos $S_{Et}/(G^n\times X)\simeq (G^n\times X)_{Et}$. The map
$$d_i^*\textrm{Ner}_{n-1}(L,a)\rightarrow\textrm{Ner}_{n}(L,a)$$
is $Id$ for $i<n$ and $(d_0...d_0)^*(a)$ for $i=n$. Finally the map
$$s_i^*\textrm{Ner}_n(L,a)\rightarrow \textrm{Ner}_{n-1}(L,a)$$ is the identity for any $i$. The functor $\textrm{Ner}$ commutes with inductive limits and finite projective limits,  since the inverse image of a morphism of topoi commutes with such limits and since these limits are computed component-wise in both $\textsc{Desc}({\bf B}(G,X)_{Et})$ and $\textsc{Top}({\bf B}(G,X)_{Et})$. Hence $\textrm{Ner}$ is the inverse image of a morphism of topoi
$$\textsc{Top}({\bf B}(G,X)_{Et})\rightarrow\textsc{Desc}({\bf B}(G,X)_{Et}).$$
Composing this map with the equivalence (\ref{descent=BG}), we obtain the desired morphism
$$\kappa :\textsc{Top}({\bf B}(G,X)_{Et})\rightarrow \textsc{Desc}({\bf B}(G,X)_{Et})\simeq B^{et}_G/X.$$
\end{proof}

\begin{lem}\label{lem-3}
The canonical map
$$H^i(B^{et}_G/X,\mathcal{A})\rightarrow H^i(\textsc{Top}({\bf B}(G,X)_{Et}), \kappa ^*\mathcal{A})$$
is an isomorphism for any $i$ and any abelian sheaf $\mathcal{A}$ on $B^{et}_G/X$.
\end{lem}
\begin{proof}
We shall prove  this lemma as follows: we describe spectral sequences converging to $H^*(B^{et}_G/X,\mathcal{A})$ and $H^*(\textsc{Top}({\bf B}(G,X)_{Et}), \kappa ^*\mathcal{A})$ respectively; then we show that these spectral sequences are isomorphic at $E_1$. Let $e_G$ be the final object of $B^{et}_G$. Since the map $E_G\rightarrow e_G$ has a section,  it  is an epimorphism;
hence so is $E_G\times X\rightarrow X$,  since epimorphisms are universal in a topos. We obtain a  covering $\mathcal{U}=(E_G\times X\rightarrow X)$
of the final object in $B^{et}_G/X$. This covering  gives rise to the Cartan-Leray spectral sequence (see \cite{SGA4} [V Cor. 3.3])
$$\Check{H}^i(\mathcal{U},\underline{H}^j(\mathcal{A}))\Rightarrow H^{i+j}(B^{et}_G/X,\mathcal{A})$$
where $\underline{H}^j(\mathcal{A})$ denotes the presheaf on $B^{et}_G/X$
$$\underline{H}^j(\mathcal{A}): (\mathcal{F}\rightarrow X)\rightarrow H^{j}((B^{et}_G/X)/\mathcal{F}, \mathcal{F}\times_{X}\mathcal{A})$$
and $\check{H}^i(\mathcal{U},-)$ denotes Cech cohomology. By \cite{SGA4}[IV, 5.8.3] we have a canonical equivalence
$$(B^{et}_G/X)/(E_G\times X)\simeq S_{Et}/X.$$
Consider more generally the $n$-fold product of $(E_G\times X)$ with itself over the final object in $B^{et}_G/X$
$$(E_G\times X)^n=(E_G\times X)\times_X...\times_X(E_G\times X)=(E_G^n\times X).$$
Here $E_G^n$ is the object of $B^{et}_G$ represented by the scheme $G^n=G\times_S...\times_SG$ on which $G$ acts diagonally. Then we have an equivalence
$$(B^{et}_G/X)/(E_G^{n+1}\times X)=(B^{et}_G/X)/(E_G\times(E_G^{n}\times X))\simeq S_{Et}/(G^{n}\times X)$$
for any $n\geq0$. Therefore the term $E_1^{i,j}$ of the Cartan-Leray spectral sequence takes the following form:
\begin{eqnarray*}
E_1^{i,j}&=&H^{j}((B^{et}_G/X)/(E_G^{i+1}\times X),(E_G^{i+1}\times X)\times_X \mathcal{A})\\
&=& H^{j}((S_{Et}/(G^{i}\times X),G^{i}\times \mathcal{A})\\
&=& H^{j}((G^{i}\times X)_{Et},G^{i}\times \mathcal{A})
\end{eqnarray*}
for any abelian object $\mathcal{A}\rightarrow X$ of $B^{et}_G/X$.  We conclude that   the spectral sequence can be written as follows:
\begin{equation}\label{spectralsequence-deplus}
E_1^{i,j}= H^{j}((G^{i}\times X)_{Et},G^{i}\times \mathcal{A})\Rightarrow H^{i+j}(B^{et}_G/X, \mathcal{A}).
\end{equation}
We also have a spectral sequence (see (\ref{sequence}))
\begin{equation}\label{spectralsequence-3plus}
'E_1^{i,j}={H}^j({\bf B}(G,X)_{i,Et},(\kappa^*\mathcal{A})_i)\Rightarrow H^{i+j}(\textsc{Top}({\bf B}(G,X)_{Et}),\kappa^*\mathcal{A})
\end{equation}
where $\kappa:\textsc{Top}({\bf B}(G,X)_{Et})\rightarrow B^{et}_G/X$ is the map of Lemma \ref{lem-2}. For any $i\geq0$, the following square
\[\xymatrix{
{\bf B}(G,X)_{i,Et}\ar[r]\ar[d]&(B^{et}_G/X)/(E_G^{i+1}\times X)\ar[d] \\
\textsc{Top}({\bf B}(G,X)_{Et})\ar[r]^{\kappa}&B^{et}_G/X \\
}\]
is commutative, where the top horizontal map is the canonical equivalence
$${\bf B}(G,X)_{i,Et}=(G^i\times X)_{Et}\simeq  S_{Et}/(G^{i}\times X)\simeq(B^{et}_G/X)/(E_G^{i+1}\times X). $$
Note that this last equivalence is precisely the equivalence from which we have deduced the isomorphism
$E_1^{i,j}= H^{j}((G^{i}\times X)_{Et},G^{i}\times A)$. We obtain a morphism of spectral sequences from (\ref{spectralsequence-deplus}) to (\ref{spectralsequence-3plus}). This morphism of spectral sequences is an isomorphism since $(\kappa^*\mathcal{A})_i=G^{i}\times \mathcal{A}$, which in turn follows from the fact that the square above commutes. The result follows.
\end{proof}

We now consider the flat topos $S_{fl}$,  the big \'etale topos $S_{Et}$ and their   classifying topoi
$$B_G^{et}=B_{yG}(S_{Et})\mbox{ and } B_G^{fl}=B_{yG}(S_{fl}).$$ It follows from Corollary \ref{Diaconescu-Thm} 
that the canonical morphism
$i:S_{fl}\rightarrow S_{et}$ induces a morphism $B_G\rightarrow B_G^{et}$ such that the following square
\[\xymatrix{
B_G\ar[r]\ar[d]&B_G^{et}\ar[d] \\
S_{fl}\ar[r]^{i}&S_{Et}
}\]
is a pull--back. This morphism induces a morphism \cite{SGA4}[IV, Section 5.10]:
$$\gamma:B_{G}/X\rightarrow B_{G}^{et}/X.$$
The following is an equivariant refinement of the classical comparison theorem \cite{BrauerIII}[Theorem 11.7] between \'etale and flat cohomology.
\begin{lem}\label{compra-etale-flat}
Let $\mathcal{A}=yA$ be an abelian object of $B_G/X$ represented by a smooth $X$-scheme $A$. Then the canonical morphism
$$\gamma^*:H^i(B^{et}_G/X,yA)\rightarrow H^i(B_G/X, yA)$$
is an isomorphism for any $i\geq 0$.
\end{lem}

\begin{proof}
Consider the spectral sequence (a special case of (\ref{spectralsequence-deplus}))
\begin{equation}\label{spectral-encore}
H^{j}((G^{i}\times X)_{Et},G^{i}\times \mathcal{A})\Rightarrow H^{i+j}(B^{et}_G/X,\mathcal{A})
\end{equation}
associated to the covering $(E_G\times X\rightarrow X)$ in $B^{et}_G/X$. Applying the functor $\gamma^*$ to $\mathcal{U}=(E_G\times X\rightarrow X)$, we obtain the covering $\gamma^*\mathcal{U}=(E_G\times X\rightarrow X)$ in $B_G/X$, and we get a morphism of spectral sequences from (\ref{spectral-encore}) to
$$H^{j}((G^{i}\times X)_{fl},G^{i}\times \mathcal{A})\Rightarrow H^{i+j}(B_G/X, \mathcal{A}).$$
But the canonical maps
$$H^{j}((G^{i}\times X)_{Et},G^{i}\times \mathcal{A})\rightarrow H^{j}((G^{i}\times X)_{fl},G^{i}\times \mathcal{A})$$
are isomorphisms \cite{BrauerIII}[Theorem 11.7]. 
It therefore follows that the maps
$$H^i(B^{et}_G/X,\mathcal{A})\rightarrow H^i(B_G/X,\mathcal{A})$$
are also isomorphisms.
\end{proof}

\begin{thm}\label{thm-cohomologyBG}
Let $\mathcal{A}=yA$ be an abelian object of $B_G/X$ represented by a smooth $X$-scheme $A$. Then there is a canonical isomorphism
$$H^i(B_G/X,yA)\simeq H^i(Et({\bf B}(G,X)),yA)$$
for any $i\geq 0$, where $B_G$ denotes the classifying topos of $G$ and $Et({\bf B}(G,X))$ denotes the (small) \'etale site of the simplicial scheme  ${\bf B}(G,X)$ as defined in \cite{F}.
\end{thm}
\begin{proof}
This follows from Lemmas \ref{lem-1}, \ref{lem-2}, \ref{lem-3} and \ref{compra-etale-flat}.
\end{proof}

\subsection{Constant group schemes}
Here we consider a constant $S$-group scheme $G_S=\coprod_GS$ where $G$ is an abstract  group. In what follows we do not distinguish the abstract group $G$ and  the group scheme $G_S$. Let $X$ be an $S$-scheme endowed with a left action of $G$. We denote by $\mathcal{S}_{et}(G,X)$ the topos of $G$--equivariant (small) \'etale sheaves on $X$: an object $\mathcal{F}$ of $\mathcal{S}_{et}(G,X)$ is a sheaf of sets $\mathcal{F}$ on the small \'etale site of $X$ (i.e. on the category of \'etale $X$-schemes given with the \'etale topology) together with a family of isomorphisms
$\{\phi_g:\mathcal{F}\rightarrow g^*\mathcal{F},\,g\in G\}$ such that $\phi_{hg}=g^*(\phi_h)\circ\phi_g$ and $\phi_1=\textrm{Id}_{\mathcal{F}}$.
The topos $\mathcal{S}_{et}(G,X)$ is the algebraic analogue of the category of $G$--equivariant \'etal\'e spaces on a $G$-topological space. This definition is due to Grothendieck (see \cite{Grothendieck-chern-repres} 2.1) who calls \ the cohomology of $\mathcal{S}_{et}(G,X)$ \emph{mixed cohomology}.
\begin{prop}There is a canonical morphism
$$h:B_{G}/X\rightarrow\mathcal{S}_{et}(G,X)$$
inducing isomorphisms
$$H^i(\mathcal{S}_{et}(G,X),yA)\stackrel{\sim}{\rightarrow} H^i(B_{G}/X,yA)$$
for any abelian object $\mathcal{A}=yA$ of $B_G/X$ represented by a smooth $X$-scheme $A$.
\end{prop}
\begin{proof}
Consider the (small \'etale) simplicial topos
$$\fonc{{\bf B}(G,X)_{\,et}}{\Delta^{op}}{\bf Top}{[0,n]}{(G^n\times X)_{\,et}}.$$
It is easy to see that the topos of equivariant sheaves is in fact a descent topos:
$$\mathcal{S}_{et}(G,X)\simeq \textsc{Desc}({\bf B}(G,X)_{\,et}).$$
The morphism of simplicial topoi ${\bf B}(G,X)_{\,Et}\rightarrow {\bf B}(G,X)_{\,et}$ therefore induces a canonical morphism
$$B_{G}/X:=B_{G}(S_{fl})/X\longrightarrow B_{G}(S_{Et})/X\simeq \textsc{Desc}({\bf B}(G,X)_{\,Et})$$
$$\longrightarrow \textsc{Desc}({\bf B}(G,X)_{\,et})\simeq \mathcal{S}_{et}(G,X).$$
Here the equivalence $B_{G}(S_{Et})/X\simeq \textsc{Desc}({\bf B}(G,X)_{\,Et})$ is (\ref{descent=BG}).  One can also define the morphism of the proposition directly.

Consider the morphisms $B_{G}/X\rightarrow \mathcal{S}_{et}(G,X)\rightarrow B_G({\bf Sets})$, where $G$ is considered as a discrete group. Localizing over $EG$ (an object of $B_G({\bf Sets})$), we obtain the localization pull--back (see Section \ref{subsect-2catoftopos})
\[\xymatrix{
X_{fl}\ar[r]^{h'}\ar[d]^{l'}&X_{et}\ar[d]^l \\
B_G/X\ar[r]^h&  \mathcal{S}_{et}(G,X)}\]
since $\mathcal{S}_{et}(G,X)/EG\simeq X_{et}$ is the small \'etale topos of $X$ and $(B_G/X)/(EG\times X)\simeq X_{fl}$ is its big flat topos. Indeed, the natural transformation
$$l^*\circ h_*\rightarrow h'_*\circ l'^*$$
is easily seen to be an isomorphism. Moreover, $l'^*$ is exact and preserves abelian injective objects. We obtain
$$l^*\circ R(h_*)\simeq R(l^*\circ h_*)\simeq  R(h'_*\circ l'^*)\simeq R(h'_*)\circ l'^*.$$
In particular, we have
$$l^*\circ R(h_*)(yA)= R(h'_*)\circ l'^*(yA)= R(h'_*)(yA)$$
hence by \cite{BrauerIII}[Theorem 11.7]
$$l^*\circ R^j(h_*)(yA)= R^j(h'_*)(yA)=0$$
for any $j\geq1$ and any smooth $A$ (note that $R^j(h'_*)(yA)$ is the sheaf associated with the presheaf sending an \'etale map $U\rightarrow X$ to $H^j(U_{fl},yA)=H^j(U_{et},yA)$). But $l^*$ is conservative, hence $R^j(h_*)(yA)=0$ for any $j\geq1$ and any smooth $A$. Recall that $l^*$ conservative means that, for any map $\phi$, $l^*(\phi)$ is an  isomorphism if and only if $\phi$ is an isomorphism (see \cite{SGA4}[I, Definition 6.1]).

The result now follows from the fact that the spectral sequence
$$H^i(\mathcal{S}_{et}(G,X),R^jh_*(yA))\Rightarrow H^{i+j}(B_G/X,yA)$$
degenerates and yields the desired isomorphism.

\end{proof}

\subsection{Equivariant cohomology spectral sequences}

We consider a left action of $G$ on $X$ over the base scheme $S$ such that $Y=X/G$ exists as a scheme.
The following diagram:
\[\xymatrix{
B_G/X\ar[r]^f\ar[d]_g&B_G\ar[d] \\
Y_{fl}\ar[r]& {\bf Sets}}\]
provides two different decompositions of the unique map $B_G/X\rightarrow{\bf Sets}$. This yields two Leray spectral sequences converging to the cohomology of $B_G/X$. Here the map $f$ is the localization map and $g$ is defined as follows: For an object $\mathcal{F}\rightarrow Y$ of $Y_{fl}\simeq S_{fl}/Y$, we have $g^*(\mathcal{F}\rightarrow Y):=\mathcal{F}\times_Y X$ where $G$ acts on $\mathcal{F}\times_Y X$ via its action on $X$. We adopt the following notation.
\begin{definition}\label{def-relative-cohomology}
Let $t:\mathcal{E}\rightarrow\mathcal{S}$ be a topos over a base topos $\mathcal{S}$ and let $\mathcal{A}$ be an abelian object in $\mathcal{E}$. The cohomology of $\mathcal{A}$ with values in $\mathcal{S}$ in degree $n$ is the following abelian object in $\mathcal{S}$:
$$H_{\mathcal{S}}^n(\mathcal{E},\mathcal{A}):=R^n(t_*)\mathcal{A}.$$
\end{definition}
For any Grothendieck topos $\mathcal{E}$  we have
$$H_{{\bf Sets}}^n(\mathcal{E},\mathcal{A})=H^n(\mathcal{E},\mathcal{A}):=R^n(e_*)\mathcal{A}, $$
where  $e:\mathcal{E}\rightarrow {\bf Sets}$ is the unique morphism from $\mathcal{E}$ to the final topos.
For any topos $t:\mathcal{E}\rightarrow\mathcal{S}$ over $\mathcal{S}$, the Leray spectral sequence reads as follows:
$$H^i(\mathcal{S},H_{\mathcal{S}}^j(\mathcal{E},\mathcal{A}))\Longrightarrow H^{i+j}(\mathcal{E},\mathcal{A}).$$

For a topos $\mathcal{T}$ over $S_{fl}$ we write $H^j_{S}(\mathcal{T},-)$ for $H^j_{S_{fl}}(\mathcal{T},-)$. First we consider the morphism $f:B_G/X\rightarrow B_G$.
\begin{prop}\label{localproof}
Let $G$ be an $S$-group scheme acting on a scheme $X$.
For any abelian object $\mathcal{A}$ of $B_G/X$, there is a natural $G$--action on $H_{S}^j(X_{fl},\mathcal{A})$ and one has a spectral sequence
$$H^i(B_G,H_{S}^j(X_{fl},\mathcal{A}))\Rightarrow H^{i+j}(G,X,\mathcal{A}).$$
\end{prop}
\begin{proof}

The composite map
$$B_G/X\rightarrow B_G\rightarrow {\bf  Sets}$$
provides us with the spectral sequence
$$H^i(B_G,R^j(f_*)\mathcal{A})\Rightarrow H^{i+j}(B_G/X,\mathcal{A})$$
which is functorial in the abelian object $\mathcal{A}$ of $B_G/X$. In view of the canonical equivalences
$$B_G/E_G\simeq S_{fl} \mbox{ and }(B_G/X)/(E_G\times X)\simeq S_{fl}/X\simeq X_{fl},$$
we have a pull--back square
\[\xymatrix{
X_{fl}\ar[r]^{f'}\ar[d]_{p'}&S_{fl}\ar[d]_{p} \\
B_G/X\ar[r]^f&B_G
}\]
where the vertical maps are localization maps. For any such localization pull--back, there is a natural transformation
$$p^*f_*\rightarrow f'_*p'^*$$
which is an isomorphism. Moreover the functor $p'^*$, being a localization functor, preserves injective abelian objects. Thus  we obtain an isomorphism
$$p^*R^n(f_*)\simeq R^n(f'_*)p'^*.$$
The result follows from the fact that the functors $p^*:B_G\rightarrow S_{fl}$ and $p'^*:B_G/X\rightarrow X_{fl}$ are the forgetful functors (where  we forget the $G$--action).
\end{proof}

\begin{prop}
Let $G$ be an $S$-group scheme acting on a scheme $X$ such that $Y=X/G$ exists as a scheme. There is a spectral sequence
$$H^i(Y_{fl},R^j(g_*)\mathcal{A})\Rightarrow  H^{i+j}(G,X,\mathcal{A})$$
functorial in the abelian object $\mathcal{A}$ of $B_G/X$, where $g:B_G/X\rightarrow Y_{fl}$ is the morphism induced by the $G$--equivariant map $X\rightarrow Y$.
\end{prop}
\begin{proof}
This is the Leray spectral sequence associated with the composite morphism
$$B_G/X\rightarrow Y_{fl}\rightarrow {\bf Sets}.$$
\end{proof}

\subsection{Giraud's exact sequence}

Let $A$ be a commutative $S$--group scheme endowed with a left action of $G$. We denote by  $\textrm{Ext}_{S}(G,A)$ the abelian group of extensions of $G$ by $A$ in the topos $S_{fl}$. More precisely, $\textrm{Ext}_{S}(G,A)$ is the group of equivalence classes of exact sequences in $S_{fl}$
$$1\rightarrow yA\rightarrow\mathcal{G}\rightarrow yG \rightarrow 1$$
where $yA$ and $yG$ denote the sheaves in $S_{fl}$ represented by $A$ and $G$, and
such that the action of $\mathcal{G}$ on $yA$ by inner automorphisms induces the given action of $G$ on $A$. Note that $\mathcal{G}$ is not a scheme in general. We denote by $\textrm{Crois}_S(G,A)$ the abelian group of crossed morphisms $f:G\rightarrow A$. Recall that a crossed morphism is a map of $S$-schemes $f:G\rightarrow A$ such that
\begin{equation}\label{crossed}
f(g g')=f(g)+g\cdot f(g').
\end{equation}
This identity makes sense on points. Equivalently (\ref{crossed}) can be seen as a commutative diagram in ${\bf Sch}/S$. Note that, if $G$ acts trivially on $A$, then $\textrm{Crois}_S(G,A)=\textrm{Hom}_S(G,A)$.

More generally, for any abelian object $\mathcal{A}$ of $B_G$, one defines $\textrm{Ext}_{S}(G,\mathcal{A})$ and $\textrm{Crois}_S(G,\mathcal{A})$ in the very same way (of course one has $\textrm{Ext}_{S}(G,A)=\textrm{Ext}_{S}(G,yA)$ and $\textrm{Crois}_S(G,A)=\textrm{Crois}_S(G,yA)$).

\begin{prop}\label{prop-Giraud-sequence}
We have an exact sequence of abelian groups
$$0\rightarrow H^0(B_G,\mathcal{A})\rightarrow H^0(S_{fl},\mathcal{A})\rightarrow \emph{\textrm{Crois}}_S(G,\mathcal{A})\rightarrow H^1(B_G,\mathcal{A})$$
$$\rightarrow H^1(S_{fl},\mathcal{A})\rightarrow \emph{\textrm{Ext}}_{S}(G,\mathcal{A})\rightarrow
H^2(B_G,\mathcal{A})\rightarrow H^2(S_{fl},\mathcal{A}) .$$
\end{prop}
\begin{proof}
This is a special case of \cite{Giraud}[VIII.7.1.5].
\end{proof}

\subsection{The sheaves $H^i_{S}(B_G,\mathcal{A})$ for $i=0,1,2$}
Recall that we denote by
$$\pi: B_G\rightarrow S_{fl}$$
the canonical map and,  for an abelian object $\mathcal{A}$ of $B_G$, we denote by
$$H^i_{S}(B_G,\mathcal{A}):=R^i(\pi_*)\mathcal{A}$$
the cohomology of $B_G$ with values in the topos $S_{fl}$. The sheaf $H^i_{S}(B_G,\mathcal{A})$ may be described as follows: 
\begin{prop}\label{oneprop}
For any abelian object $\mathcal{A}$ of $B_G$ and for any $i\geq0$, the sheaf $H^i_{S}(B_G,\mathcal{A})$ is the sheaf associated to the presheaf
$$\appl{{\bf Sch}/S}{{\bf Ab}}{T}{H^i(G_{T},\mathcal{A}_{T})}$$
where $G_{T}$ is the $T$-group scheme $G\times_ST$ and $\mathcal{A}_{T}$ is the abelian object of $B_{G_{T}}$
induced by $\mathcal{A}$.
\end{prop}
\begin{proof}
The sheaf $H^i_{S}(B_G,\mathcal{A})$ is the sheaf associated to the presheaf
$$
{T}\rightarrow {H^i(B_{G}/T,\mathcal{A}\times{T}), }
$$
but by virtue of Proposition \ref{prop-BGY} we have
$$H^i(B_{G}/T,\mathcal{A}\times T)=H^i(B_{G_{T}},\mathcal{A}_{T})=:H^i(G_{T},\mathcal{A}_{T}).$$
\end{proof}

For any abelian sheaf $\mathcal{A}$ of $B_G$, we denote by $\mathcal{A}^G$ the largest sub-object of $\mathcal{A}$ on which $G$ acts trivially. Then we consider the abelian  presheaf
\begin{equation}\label{sheaf-Crois}
\appl{S_{fl}}{{\bf Ab}}{F}{\textrm{Crois}_{S_{fl}/F}(G\times F,\mathcal{A}\times F)}
\end{equation}
This presheaf is easily seen to be a subsheaf of the sheaf of homomorphisms  $\underline{\textrm{Map}}(G, \mathcal{A})$ (here the group structure is not taken into account) in the topos $S_{fl}$ endowed with the canonical topology. The sheaf (\ref{sheaf-Crois}) is therefore representable by an abelian object $\underline{\textrm{Crois}}_{S}(G,\mathcal{A})$ of $S_{fl}$ (recall that any sheaf on a topos endowed with the canonical topology is representable). There is a morphism:
$$\fonc{\tau}{\mathcal{A}}{\underline{\textrm{Crois}}_{S}(G,\mathcal{A})}{a}{g\mapsto g\cdot a-a} . $$
Finally we consider the presheaf:
$$\appl{S_{fl}}{{\bf Ab}}{F}{\textrm{Ext}_{S_{fl}/F}(G\times F,\mathcal{A}\times F)} . $$
Basic descent theory in topoi shows that this is a sheaf for the canonical topology. We denote by $\underline{\textrm{Ext}}_{S}(G,\mathcal{A})$ the corresponding abelian object of $S_{fl}$.
\begin{cor}
We have $$H^0_{S}(B_G,\mathcal{A})\simeq \mathcal{A}^{G}\,,\,\,\, H^2_{S}(B_G,\mathcal{A})\simeq\underline{\textrm{\emph{Ext}}}_S(G,\mathcal{A})$$
and an exact sequence (of abelian objects in $S_{fl}$)
$$0\rightarrow\mathcal{A}^G\rightarrow \mathcal{A}\stackrel{\tau}{\rightarrow}\underline{\textrm{\emph{Crois}}}_S(G,\mathcal{A})\rightarrow H_S^1(B_{G},\mathcal{A})\rightarrow 0.$$
In particular, if $G$ acts trivially on $\mathcal{A}$,  then
$$H_S^1(B_{G},\mathcal{A})=\underline{\textrm{\emph{Hom}}}(G,\mathcal{A}).$$
\end{cor}
\begin{proof}
By Prop. \ref{prop-Giraud-sequence}, for any $T$ over $S$ we have the exact sequence
$$0\rightarrow H^0(B_{G_T},\mathcal{A}_{T})\rightarrow H^0(T_{fl},\mathcal{A}_{T})\rightarrow\textrm{Crois}_T(G_T,\mathcal{A}_{T})\rightarrow H^1(B_{G_T},\mathcal{A}_{T})$$
$$\rightarrow H^1(T_{fl},\mathcal{A}_{T})\rightarrow\textrm{Ext}_{T}(G_T,\mathcal{A}_{T})\rightarrow
H^2(B_{G_T},\mathcal{A}_{T})\rightarrow H^2(T_{fl},\mathcal{A}_{T}).$$
This exact sequence is functorial in $T$,  so that it may be viewed  as an exact sequence of abelian presheaves on ${\bf Sch}/S$. Applying the associated sheaf functor together with Proposition \ref{oneprop}, we obtain the exact sequence of sheaves
$$0\rightarrow\mathcal{A}^G\rightarrow \mathcal{A}\rightarrow\underline{\textrm{Crois}}_S(G,\mathcal{A})\rightarrow H_S^1(B_{G},\mathcal{A})$$
$$\rightarrow 0\rightarrow\underline{\textrm{Ext}}_{S}(G,\mathcal{A})\rightarrow
H_S^2(B_{G},\mathcal{A})\rightarrow 0$$
since the sheafification of the presheaf $T\mapsto H^i(T_{fl},\mathcal{A}_{T})$ is trivial for $i\geq1$ (as it follows from \cite{SGA4}[V, Proposition 5.1] applied to the identity map $Id:S_{fl}\rightarrow S_{fl}$).
\end{proof}

\subsection{The Hochschild-Serre spectral sequence}

We consider a sequence of $S$-group schemes
$$1\rightarrow N\rightarrow G\rightarrow G/N\rightarrow 1$$
 which is exact with respect to the fppf-topology.

\begin{prop}\label{propHS}
For any abelian object $\mathcal{A}$ of $B_G$, there is a natural $G/N$--action on $H^j_{S}(B_N,\mathcal{A})$ and there is  a spectral sequence
$$H^i(G/N,H^j_{S}(B_N,\mathcal{\mathcal{A}}))\Longrightarrow H^{i+j}(G,\mathcal{A}).$$
\end{prop}

\begin{proof}
The quotient map
$G\rightarrow G/N$ induces a morphism of classifying topoi
$$f:B_G\rightarrow B_{G/N}, $$
and hence a spectral sequence
$$H^i(B_{G/N},R^j(f_*)\mathcal{A})\Longrightarrow H^{i+j}(B_{G},\mathcal{A})$$
which is functorial in the abelian object $\mathcal{A}$ of $B_G$.
In view of the canonical equivalences
$$B_{G/N}/E_{G/N}\simeq S_{fl} \mbox{ and }B_G/(G/N)\simeq B_N,$$
we have a pull--back square
\[\xymatrix{
B_N\ar[r]^{\pi}\ar[d]_{p'}&S_{fl}\ar[d]_{p} \\
B_G\ar[r]^f&B_{G/N}
}\]
where the vertical maps are localization maps. As previously this "localization pull--back" yields isomorphisms for all $n\geq 0$
$$p^*R^n(f_*)\simeq R^n(\pi_*)p'^*.$$
Moreover, the functor $p'^*:B_G\rightarrow B_N$ maps a $G$-object $F$ to $F$ on which $N$ acts via $N\rightarrow G$, so that, for $\mathcal{A}$ an abelian object of $B_G$, $R^n(\pi_*)p'^*\mathcal{A}$ is really what we (slightly abusively) denote by $H_S^n(B_N,\mathcal{A})$.
\end{proof}

Recall that there is a canonical map $\tau:\mathcal{A}\rightarrow\underline{\textrm{Crois}}_{S}(G,\mathcal{A})$, and that this map is the zero map if $G$ acts trivially on $\mathcal{A}$.
\begin{cor}\label{cor-short-exact-HS-spectralsequ}
There is an exact sequence
$$0\rightarrow H^1(B_{G/N},\mathcal{A}^N)\rightarrow H^1(B_G,\mathcal{A})\rightarrow
H^0(B_{G/N},\underline{\emph{Crois}}_{S}(N,\mathcal{A})/\textrm{\emph{Im}}(\tau))$$
$$\rightarrow H^2(B_{G/N},\mathcal{A}^N)\rightarrow H^2(B_G,\mathcal{A}).$$
If $N$ acts trivially on $\mathcal{A}$, we obtain an exact sequence
$$0\rightarrow H^1(B_{G/N},\mathcal{A})\rightarrow H^1(B_G,\mathcal{A})\rightarrow
H^0(B_{G/N},\underline{\emph{\textrm{Hom}}}(N,\mathcal{A}))$$
$$\rightarrow H^2(B_{G/N},\mathcal{A})\rightarrow H^2(B_G,\mathcal{A}).$$
\end{cor}
\begin{proof} This is the five--term  exact sequence given by the Hochschild--Serre spectral sequence of Proposition \ref{propHS} above.
\end{proof}

\section{Invariants of symmetric bundles}

In the next  sections we fix a scheme  $Y\rightarrow \mathrm{Spec}(\mathbb{Z}[1/2])$ in which $2$ is invertible. The principal goal of this section is to associate to any symmetric bundle over $Y$ cohomological invariants which generalize the
classical Hasse-Witt invariants associated to quadratic forms on vector spaces over fields.

\subsection{Symmetric bundles}

A bilinear form on $Y$ consists of a locally free $\mathcal{O}_Y$-module $V$ (which one may see as a vector bundle on $Y$) and a morphism of $\mathcal{O}_Y$-modules
$$B: V\otimes_{\mathcal{O}_Y} V\rightarrow \mathcal{O}_{Y}$$
such that for any affine open subscheme $Z$ of $Y$ the induced map
$$B_{Z}: V(Z)\times V(Z)\rightarrow \mathcal{O}_{Y}(Z)$$ is a symmetric
bilinear form on the $\mathcal{O}_{Y}(Z)$-module $V(Z)$. Let $V^{\vee}$ be the dual of $V$. The form $B$ induces a
morphism of
bundles
$$\varphi_{B}: V\rightarrow V^{\vee}$$ which is self-adjoint. We call $B$ \emph{non-degenerate} or \emph{unimodular}
if $\varphi_{B}$  is an isomorphism. A \emph{symmetric bundle} is a pair $(V, B)$ consisting of a $Y$-symmetric bundle $V$ endowed with a
unimodular form $B$. In general we will denote such a bundle by $(V, q)$ where $q$ is the quadratic form associated to $B$. Since
$2$ is invertible in $Y$, we will refer to $(V, q)$  either as a symmetric bundle or as a quadratic form over $Y$. In the case where  $Y=\mbox {Spec}(R)$ is affine, a symmetric bundle $(V, q)$ is given by a pair $(M, B)$ where $M$
is a locally free $R$-module and $B$ is a unimodular,  symmetric, bilinear form on $M$.

Let $(V, q)$ be a symmetric bundle over $Y$ and let $f: T\rightarrow Y$ be a morphism of schemes.  We define the pull back of
$(V, q)$ by $f$ as the symmetric bundle $(f^{*}(V), f^{*}(q))$ on $T$ where $f^{*}(V)$ is the pull back of
$V$ endowed with the form  $f^{*}(q)$   defined
on any affine open subsets $U'$ and $U$ of $T$ and $Y$,  such that $f(U') \subset U$,   by scalar extension from $q$.
We denote by $(V_T, q_T)$ the resulting symmetric bundle on $T$.

An isometry of symmetric bundles  $u:(V, q)\rightarrow (E, r)$ on $Y$ is an isomorphism of locally free $\mathcal{O}_Y$-modules $u: V\rightarrow E$
such that $r(u(x))=q(x)$ for any open affine subscheme $U$ of $Y$ and any $x$ in $V(U)$. We denote this set by  $\mbox{Isom}(q, r)$.  It follows
from (\cite{DG} III, Section 5, n° 2) that $T\rightarrow \mbox{Isom}(q_{T}, r_{T})$ is a sheaf of sets on ${\bf Sch}/Y$ endowed with the fppf--topology. We denote this sheaf by ${\bf Isom}(q, r)$, or by ${\bf Isom}_Y(q, r)$. We define the \emph{orthogonal group} ${\bf O}(q)$ as the group ${\bf Isom}(q, q)$ of $Y_{fl}$. This sheaf is representable by a smooth algebraic group scheme over $Y$, which we also denote by ${\bf O}(q)$. We denote by $(\mathcal{O}^{n}_{Y}, t_{n}=x_{1}^{2}+...+x_{n}^{2})$ the standard form over $Y$ of rank $n$ and by
 ${\bf O}(n)$ (or by ${\bf O}(n)_Y$ if we wish to stress  that the base scheme is $Y$) the orthogonal group of $t_{n}$.

\subsection{Twisted forms}\label{sect-twisted-forms}
Let $(V, q)$ be a symmetric bundle over $Y$. A  symmetric bundle $(F, r)$ is  called a {\it twisted form of $(V, q)$} if there exists
a fppf-covering $\{U_{i}\rightarrow Y,\,i\in I\}$ such that there exists an isometry
$$(V\otimes_{Y}U_{i}, q\otimes_{Y}U_{i})\simeq (F\otimes_{Y}U_{i}, r\otimes_{Y}U_{i})\ \ \forall i\in I.$$
Recall that a \emph{groupoid} is a small category whose morphisms are all isomorphisms. We denote by ${\bf Twist}(q)$ the groupoid whose objects are twisted forms of $(V,q)$ and morphisms are isometries. Let ${\bf Twist}(q)/_\sim$ be the set of isometry classes of twists of $(V, q)$, which we consider as a set pointed by the class of $(V, q)$. If  $(F, r)$ is a twist of $(V, q)$ then ${\bf Isom}(q, r)$ is an ${\bf O}(q)$--torsor of $Y_{fl}$. We denote by ${\bf Tors}(Y_{fl}, {\bf O}(q))/_\sim$ the pointed set of isometry classes of ${\bf Tors}(Y_{fl}, {\bf O}(q))$, pointed by the class of the trivial torsor.

\begin{prop}\label{propone}
The canonical functor
$$\appl{{\bf Twist}(q)}{{\bf Tors}(Y_{fl}, {\bf O}(q))}{(F, r)}{{\bf Isom}(q, r)}
$$
is an equivalence.
\end{prop}
\begin{proof}  By \cite{DG}[III, § 5, n°2, 2.1] the functor above induces an isomorphism of pointed sets:
$$\appl{{\bf Twist}(q)/_\sim}{{\bf Tors}(Y_{fl}, {\bf O}(q))/_\sim}{(F, r)}{[{\bf Isom}(q, r)]}
$$
where $[{\bf Isom}(q, r)]$ is the isometry class of the torsor ${\bf Isom}(q, r)$. Since ${\bf Twist}(q)$ and ${\bf Tors}(Y_{fl}, {\bf O}(q))$ are both groupoids, we are  reduced to showing that the automorphism group of ${\bf Isom}(q, r)$ in ${\bf Tors}(Y_{fl}, {\bf O}(q))$ is in bijection with the automorphism group of $(F, r)$ in  ${\bf Twist}(q)$. In other words, one has to show that the map $${\bf O}(r)(Y)\longrightarrow \mathrm{Hom}_{{\bf Tors}(Y_{fl}, {\bf O}(q))}({\bf Isom}(q, r),{\bf Isom}(q, r))$$
is bijective. But this follows from the fact that ${\bf Isom}(q, r)$ is a left ${\bf O}(r)$-torsor.

\end{proof}

Since ${\bf O}(q)$ is smooth,  the functor ${\bf Tors}(Y_{et}, {\bf O}(q))\rightarrow {\bf Tors}(Y_{fl}, {\bf O}(q))$
is an equivalence.  Hence any twist of $q$ is already split by an \'etale covering family.
This can also be seen as follows: let $(\mathcal{O}^{n}_{Y}, t_{n}=x_{1}^{2}+...+x_{n}^{2})$ be the standard form over $Y$ of rank $n$ and let
 ${\bf O}(n)$ be the orthogonal group of $t_{n}$.  Since on any strictly
henselian local ring (in which $2$ is invertible) a quadratic form of rank $n$ is isometric to the standard form,  any symmetric bundle $(V,q)$ on $Y$ is locally isometric to the standard form $t_{n}$ for the \'etale topology. We denote by ${\bf Quad}_{n}(Y)$  the groupoid whose objects are symmetric bundles of rank $n$ over $Y$ and whose morphisms are isometries. There are canonical equivalences of categories
\begin{eqnarray}
{\bf Quad}_{n}(Y)&\simeq& {\bf Twist}(t_{n})\\
&\simeq&{\bf Tors}(Y_{et}, {\bf O}(n))\\
&\simeq&{\bf Tors}(Y_{fl}, {\bf O}(n))\\
&\simeq&{\bf Homtop}_{Y_{fl}}(Y_{fl}, B_{{\bf O}(n)})^{op}. 
\end{eqnarray}
Given a symmetric bundle $(V,q)$ on $Y$, we denote by $\{q\}: Y_{fl}\rightarrow B_{{\bf O}(n)}$ the
morphism of topoi associated to the quadratic form $q$ by this equivalence.
\begin{prop}\label{prop-define-Quad-as-Homtop}
There is an equivalence of categories
$$\appl{{\bf Quad}_{n}(Y)}{{\bf Homtop}_{Y_{fl}}(Y_{fl}, B_{{\bf O}(n)})^{op}}{(V,q)}{\{q\}} .$$
\end{prop}

\subsection{Invariants in low degree: $\mathrm{det}[q]$ and $[C_q]$}\label{subsect-invariants}

For a symmetric bundle $(V, q)$ over $Y$, there are canonical cohomology classes in $H^1(B_{{\bf O}(q)}, {\bf Z}/2{\bf Z})$ and $H^2(B_{{\bf O}(q)}, {\bf Z}/2{\bf Z})$.

\subsubsection{Degree $1$}
The determinant map
$$\mbox{det}_{{\bf O}(q)}: {\bf O}(q)\rightarrow \mu_2\stackrel{\sim}{\rightarrow}{\bf Z}/2{\bf Z}$$
is a morphism of $Y$-group schemes. By Proposition \ref{prop-Giraud-sequence}, there is a canonical map
$$\mbox{Hom}_Y({{\bf O}(q)}, {\bf Z}/2{\bf Z})\rightarrow H^1(B_{{\bf O}(q)}, {\bf Z}/2{\bf Z}).$$
This yields a class $\mbox{det}[q]\in H^1(B_{{\bf O}(q)}, {\bf Z}/2{\bf Z})$.
\begin{definition}
The class $\mbox{\emph{det}}[q]\in H^1(B_{{\bf O}(q)}, {\bf Z}/2{\bf Z})$ is the class defined by the morphism $\mbox{\emph{det}}_{{\bf O}(q)}:{\bf O}(q)\rightarrow {\bf Z}/2{\bf Z}$.
\end{definition}
The cohomology class $\mbox{det}[q]$ is represented by the morphism
$$B_{\mbox{det}_{{\bf O}(q)}}:B_{{\bf O}(q)}\rightarrow B_{{\bf Z}/2{\bf Z}}.$$
The ${\bf Z}/2{\bf Z}$--torsor of $B_{{\bf O}(q)}$ corresponding to this morphism is given by ${\bf O}(q)/{\bf SO}(q)$ with its natural left ${\bf O}(q)$--action and its right ${\bf Z}/2{\bf Z}$--action via ${\bf O}(q)/{\bf SO}(q)\simeq{\bf Z}/2{\bf Z}$. In other words, we may write
$$\mbox{det}[q]=[{\bf O}(q)/{\bf SO}(q)]\in H^1(B_{{\bf O}(q)}, {\bf Z}/2{\bf Z}).$$

\subsubsection{Degree $2$}
One can define the Clifford algebra of the  symmetric bundle $(V, q)$; this is a sheaf of algebras over $Y$.
This  leads us to consider the  group $\widetilde {\bf O}(q)$ which, in this context, is the generalization of the group $\mbox{Pin} (q)$ (see [5], Section 1.9 and [17], Appendix 4). 
The group $\widetilde {\bf O}(q)$ is
a smooth group scheme over $Y$ which  is an extension of ${\bf O}(q)$ by ${\bf Z}/2{\bf Z}$, i.e. there is
an exact sequence of groups in $Y_{fl}$
\begin{equation}\label{ext-Cq}
1\rightarrow {\bf Z}/2{\bf Z}\rightarrow \widetilde {\bf O}(q)\rightarrow {\bf O}(q)\rightarrow 1.
\end{equation}
Such a sequence defines a class $C_q\in \mbox{Ext}_{Y}({\bf O}(q), {\bf Z}/2{\bf Z})$ and therefore (see Proposition \ref{prop-Giraud-sequence}) a cohomology class in $H^{2}(B_{{\bf O}(q)}, {\bf Z}/2{\bf Z})$ that we denote by $[C_{q}]$.
\begin{definition}
The class $[C_{q}]\in H^2(B_{{\bf O}(q)}, {\bf Z}/2{\bf Z})$ is the class defined by the extension $C_q\in \mbox{\emph{Ext}}_{Y}({\bf O}(q), {\bf Z}/2{\bf Z})$.
\end{definition}

\subsection{The fundamental morphisms $T_q$ and $\Theta_q$}
Let $(\mathcal{O}^{n}_{Y}, t_{n}=x_{1}^{2}+...+x_{n}^{2})$ denote the standard form over $Y$ of rank $n$ and let
${\bf O}(n):={\bf Isom}(t_{n}, t_n)$ be the orthogonal group of $t_{n}$. Let $\pi: B_{{\bf O}(q)}\rightarrow Y_{fl}$ be the morphism of topoi associated to the group homomorphism ${\bf O}(q)\rightarrow \{e\}$.
The sheaf ${\bf Isom}(t_{n}, q)$ has a natural left action of ${\bf O}(q)$ and a natural right action of ${\bf O}(n)$:
$$\appl{{\bf O}(q)\times{\bf Isom}(t_{n}, q)\times {\bf O}(n)}{{\bf Isom}(t_{n}, q)}{(\tau,f,\sigma)}{\tau\circ f\circ\sigma}$$
These actions are compatible; more precisely, ${\bf Isom}(t_{n}, q)$ is naturally an object of $B_{{\bf O}(q)}$ which carries a right action of $\pi^*{\bf O}(n)$.

\begin{lem} The sheaf ${\bf Isom}(t_{n}, q)$ is a $\pi^{*}{\bf O}(n)$--torsor of $B_{{\bf O}(q)}$.
\end{lem}
\begin{proof}
On the one hand, ${\bf Isom}(t_{n}, q)$ is an ${\bf O}(n)$--torsor of $Y_{fl}$,  since there exists
a fppf-covering (or equivalently, an \'etale covering) $U\rightarrow Y$  and an isometry
$$(V\otimes_{Y}U, q\otimes_{Y}U)\simeq (\mathcal{O}_Y^n\otimes_{Y}U, t_n\otimes_{Y}U).$$
We obtain a covering $U\rightarrow Y$ in $Y_{fl}$ and an ${\bf O}(n)$--equivariant isomorphism in $U_{fl}\simeq Y_{fl}/U$:
$$U\times{\bf Isom}(t_{n}, q)\stackrel{\sim}{\rightarrow}U\times {\bf O}(n).$$
On the other hand, there is a canonical equivalence $Y_{fl}\simeq B_{{\bf O}(q)}/E_{{\bf O}(q)}$ such that
the composite morphism $$f:Y_{fl}\stackrel{\sim}{\rightarrow} B_{{\bf O}(q)}/E_{{\bf O}(q)}\rightarrow B_{{\bf O}(q)}$$
is the map induced by the morphism of groups $1\rightarrow {\bf O}(q)$; in other words, the inverse image functor $f^*$ forgets the ${\bf O}(q)$--action. We may therefore view  respectively  the right $(E_{{\bf O}(q)}\times \pi^*{\bf O}(n))$-objects $E_{{\bf O}(q)}\times{\bf Isom}(t_{n}, q)$ and $E_{{\bf O}(q)}\times \pi^*{\bf O}(n)$ of the topos $B_{{\bf O}(q)}/E_{{\bf O}(q)}$ as the right ${\bf O}(n)$-objects ${\bf Isom}(t_{n}, q)$ (with no ${\bf O}(q)$--action) and ${\bf O}(n)$ of the topos $Y_{fl}$. We obtain an ${\bf O}(n)$--equivariant isomorphism
$$U\times E_{{\bf O}(q)}\times{\bf Isom}(t_{n}, q)\stackrel{\sim}{\rightarrow}U\times E_{{\bf O}(q)}\times \pi^*{\bf O}(n)$$
in the topos $$B_{{\bf O}(q)}/(U\times E_{{\bf O}(q)})\simeq (B_{{\bf O}(q)}/E_{{\bf O}(q)})/(U\times E_{{\bf O}(q)})\simeq Y_{fl}/f^*U\simeq U_{fl}.$$
The result follows since $U\times E_{{\bf O}(q)}\rightarrow *$ covers the final object  of $B_{{\bf O}(q)}$.
\end{proof}

\begin{definition}\label{defTq}
We denote by $T_{q}$ the sheaf ${\bf Isom}(t_{n}, q)$ endowed with its structure of $\pi^{*}({\bf O}(n))$--torsor of
$B_{{\bf O}(q)}$, and we define
$$T_q:B_{{\bf O}(q)}\rightarrow B_{{\bf O}(n)}$$
to be the morphism associated to the torsor $T_q$.
\end{definition}

\begin{prop}\label{prop-fundamental-equivalence}
The map $T_q:B_{{\bf O}(q)}\rightarrow B_{{\bf O}(n)}$ is an equivalence.
\end{prop}
\begin{proof} Let $X_q:B_{{\bf O}(n)}\rightarrow B_{{\bf O}(q)}$ be the map associated to the ${\bf O}(q)$--torsor of $B_{{\bf O}(n)}$ given by
${\bf Isom}(q,t_{n})$, and consider
$$X_q\circ T_q:B_{{\bf O}(q)}\rightarrow B_{{\bf O}(n)}\rightarrow B_{{\bf O}(q)}.$$
We have
$$(X_q\circ T_q)^*E_{{\bf O}(q)}=T_q^* X_q^*(E_{{\bf O}(q)})=T_q^*({\bf Isom}(q,t_{n}))={\bf Isom}(t_{n},q)\wedge^{{\bf O}(n)}{\bf Isom}(q,t_n). $$
But the map
$$\appl{{\bf Isom}(t_{n},q)\times {\bf Isom}(q,t_n)}{{\bf Isom}(q,q)}{(f,g)}{f\circ g}
$$
induces a $\pi^*{\bf O}(q)$--equivariant isomorphism
$${\bf Isom}(t_{n},q)\wedge^{{\bf O}(n)}{\bf Isom}(q,t_n)\simeq {\bf Isom}(q,q)=E_{{\bf O}(q)}. $$
Hence we have a canonical isomorphism
$$(X_q\circ T_q)^*E_{{\bf O}(q)}={\bf Isom}(t_{n},q)\wedge^{{\bf O}(n)}{\bf Isom}(q,t_n)\simeq E_{{\bf O}(q)}$$
of $\pi^*{\bf O}(q)$--torsors in $B_{{\bf O}(q)}$. By Theorem \ref{Thm-classifying--torsors}, we obtain an isomorphism $X_q\circ T_q\simeq \mbox{Id}_{B_{{\bf O}(q)}}$ of morphisms of topoi.
Similarly, we have
$$(T_q\circ X_q)^*E_{{\bf O}(n)}={\bf Isom}(q,t_{n})\wedge^{{\bf O}(q)}{\bf Isom}(t_n,q)\simeq E_{{\bf O}(n)}$$
hence $T_q\circ X_q\simeq\mbox{Id}_{B_{{\bf O}(n)}}.$
\end{proof}

Let $\eta : Y_{fl}\rightarrow B_{{\bf O}(q)}$ be the morphism of topoi induced by the morphism of groups $1\rightarrow {\bf O}(q)$, while
$\pi: B_{{\bf O}(q)}\rightarrow Y_{fl}$ is induced by the morphism of groups ${\bf O}(q)\rightarrow 1$. Note that
$$\mbox{Id}_{Y_{fl}}\cong \pi\circ\eta:Y_{fl}\rightarrow B_{{\bf O}(q)}\rightarrow Y_{fl}. $$
However, $\eta\circ\pi\ncong Id_{B_{{\bf O}(q)}}$, since $(\eta\circ\pi)^*$ sends an object of $B_{{\bf O}(q)}$ given by a sheaf $\mathcal{F}$ of $Y_{fl}$ endowed with a (possibly non-trivial) ${\bf O}(q)$--action to the object of $B_{{\bf O}(q)}$ given by $\mathcal{F}$ with trivial ${\bf O}(q)$--action.  The pull--back $\eta^*{\bf Isom}(t_{n}, q)$ is an ${\bf O}(n)$--torsor of $Y_{fl}$, and $\pi^*\eta^*({\bf Isom}(t_{n}, q))$ is a $\pi^{*}({\bf O}(n))$--torsor of $B_{{\bf O}(q)}$. We denote this torsor by $\Theta_{q}$ . This is the sheaf ${\bf Isom}(t_{n}, q)$ endowed with the trivial left action of ${\bf O}(q)$ and its natural right action of ${\bf O}(n)$.
\begin{definition}
We consider the $\pi^{*}({\bf O}(n))$--torsor of
$B_{{\bf O}(q)}$ given by $$\Theta_q:=\pi^*\eta^*({\bf Isom}(t_{n}, q)),$$
and we define
$$\Theta_q:B_{{\bf O}(q)}\rightarrow B_{{\bf O}(n)}$$
to be the morphism associated to the torsor $\Theta_q$.
\end{definition}
We have a commutative diagram:
\[
\xymatrix{
&&B_{{\bf O}(q)}\ar[d]^{\pi}\ar[rrrd]^{\Theta_q}&&&\\
Y_{fl}\ar[rr]_{\small{\mbox{Id}}}\ar[rru]^{\eta}\ar[rrd]_{\eta}&&Y_{fl}\ar[rrr]_{\{q\}=\eta^*{\bf Isom}(t_{n}, q)\,\,\,\,\,\,\,\,\,\,\,\,}&&&B_{{\bf O}(n)}\\
&&B_{{\bf O}(q)}\ar[rrru]_{T_q}&&&}\]
In other words, we have canonical isomorphisms:
$$ \pi\circ\eta\simeq id,\ \{q\}\circ \pi\simeq \Theta_{q}, \ T_{q}\circ \eta\simeq \Theta_{q}\circ \eta\simeq \{q\} .$$

\subsection{Hasse-Witt invariants}\label{subsect-HW}

In  Theorem 2.8 of \cite{J5},  Jardine  proved the following result,  which is a basic source for the  definitions of our invariants:

 \begin{thm}\label{thm-Jardine} Let $Y$ be a scheme in which $2$ is invertible and let $A$ denote the algebra $H^{*}_{et }(Y, {\bf Z}/2{\bf Z})$. Assume that $Y$ is the disjoint union of its connected components. Then there is a canonical isomorphism of graded $A$-algebras of the form
 $$H^*(\textsc{Top}(B({\bf O}(n),Y)_{Et}), {\bf Z}/2{\bf Z})\simeq A[HW_{1}, ...,HW_{n}], $$
 where the polynomial generator $HW_{i}$ has degree $i$.
 \end{thm}
By Theorem \ref{thm-cohomologyBG}  there is a canonical isomorphism:
$$H^{*}(B_{{\bf O}(n)}, {\bf Z}/2{\bf Z})\simeq H^*(\textsc{Top}(B({\bf O}(n),Y)_{Et}), {\bf Z}/2{\bf Z}).$$
We use this isomorphism to identify these two groups and from now on we view  $HW_{i}$ as an element of
$H^{i}(B_{{\bf O}(n)}, {\bf Z}/2{\bf Z})$. 

\begin{definition}{\it The universal Hasse-Witt $i$th-invariant of the quadratic form $q$ is
$$HW_{i}(q)=T_{q}^{*}(HW_{i})\in H^{i}(B_{{\bf O}(q)}, {\bf Z}/2{\bf Z}).$$ }
\end{definition}
When $q$ is the standard form $t_{n}$,  the invariants $HW_{i}(q)$ coincide in degree $1$ and $2$ with the invariants we introduced in subsection \ref{subsect-invariants}, i.e. we have
$$HW_{1}=HW_{1}(t_{n})=\mbox {det}[t_{n}]\ \mbox{and}\ HW_{2}=HW_{2}(t_{n})=[C_{n}]. $$ 
\begin{cor}\label{cor-cohomologyofOq}
There is a canonical isomorphism of graded $A$-algebras of the form
$$H^*(B_{{\bf O}(q)}, {\bf Z}/2{\bf Z})\simeq A[HW_{1}(q), ...,HW_{n}(q)], $$
where the polynomial generator $HW_{i}(q)$ has degree $i$.
\end{cor}
\begin{proof}
This follows from Theorem \ref{thm-Jardine} and Proposition \ref{prop-fundamental-equivalence}.
\end{proof}

We note that the "usual" $i$-th Hasse-Witt invariant of $q$ is the class of $H^{i}(Y_{fl}, {\bf Z}/2{\bf Z})$ obtained by pulling-back $HW_{i}(q)$  by $\eta$. To be more precise, the $i$-th Hasse-Witt invariant of $q$ is defined by:
$$w_{i}(q)=\eta^{*}(HW_{i}(q))=\{q\}^{*}(HW_{i}).$$
Since $\eta^{*}\circ \pi^{*}\simeq\mbox{Id}$, the group homomorphism
$$\eta^{*}: H^{i}(B_{{\bf O}(q)}, {\bf Z}/2{\bf Z})\rightarrow H^{i}(Y_{fl}, {\bf Z}/2{\bf Z})  $$ is split for any $i\geq0$.
Therefore we may identify via $\pi^{*}$ the group $ H^{i}(Y_{fl}, {\bf Z}/2{\bf Z})$ as
a direct factor of $H^{i}(B_{{\bf O}(q)}, {\bf Z}/2{\bf Z})$. Since $\Theta_{q}\simeq \{q\}\circ \pi$,  we note that
under this identification  ${\Theta_{q}}^{*}(HW_{i})$ identifies with $w_{i}(q)$. This leads to the following
\begin{definition}
The $i$-th Hasse-Witt invariant of $q$ is defined by:
$$w_{i}(q)=\Theta_q^*(HW_{i})\in H^{i}(B_{{\bf O}(q)}, {\bf Z}/2{\bf Z}). $$
\end{definition}

Let $G_Y$ be a $Y$-group scheme and let  $(V, q, \rho )$ be a $G_Y$--equivariant symmetric bundle on $Y$.
The group-scheme homomorphism $\rho:G_Y\rightarrow{\bf O}(q)$ induces a morphism of topoi $\rho: B_{G_Y}\rightarrow B_{{\bf O}(q)}$. We obtain $$T_{q}\circ \rho:B_{G_{Y}}\rightarrow B_{{\bf O}(q)}\rightarrow B_{{\bf O}(n)}.$$
This morphism corresponds to the ${\bf O}(n)$--torsor $\rho^{*}(T_{q})$ of $B_{G_Y}$ which is given by the sheaf ${\bf Isom}(t_{n}, q)$ endowed with a left action of $G_Y$ via $\rho$ and a right action of ${\bf O}(n)$.

\begin{definition}{\it The $i$-th equivariant Hasse-Witt invariant of $(V, q, \rho)$ is defined by:
$$w_{i}(q, \rho)=\rho^{*}(HW_{i}(q))=\rho^{*}T_{q}^{*}(HW_{i})\in H^{i}(B_{G_{Y}}, {\bf Z}/2{\bf Z}).$$}

\end{definition}

\section{Universal Comparison Formulas}

Let $(V, q)$  be a symmetric bundle on the scheme $Y$. We assume that $2$ is invertible in $Y$ and that $Y=\Coprod_{\alpha\in A}Y_{\alpha}$ is the disjoint union of its connected components. This second condition is rather weak but not automatic. 

By Corollary \ref{cor-cohomologyofOq} we have a canonical isomorphism
$$H^*(B_{{\bf O}(q)}, {\bf Z}/2{\bf Z})\simeq A[HW_{1}(q), ...HW_{n}(q)]$$
where $A=H_{et}^*(Y_{fl},{\bf Z}/2{\bf Z})$. Under this identification, we have classes
$$w_1(q),\,w_2(q)\in A\mbox{ and } \mbox{det}[q],\,[C_{q}]\in A[HW_{1}(q), ...HW_{n}(q)]$$
defined in subsection \ref{subsect-HW}. Theorems \ref{mainthm1} and \ref{mainthm2} provide an explicit expression of $\mbox{det}[q]$ and $[C_{q}]$ as polynomials in $HW_{1}(q)$ and $HW_{2}(q)$ with coefficients in $A$ written in terms of $w_1(q),w_2(q)\in A$.
More precisely, we shall prove the following
\begin{thm}\label{mainthm} Let $(V, q)$ be a symmetric bundle of rank $n$ on the scheme $Y$ in which $2$ is invertible. Assume that $Y$ is the disjoint union of its connected components. Then we have
$$\textrm{\emph{det}}[q]=w_{1}(q)+HW_{1}(q)$$
and
$$[C_q]=(w_{1}(q)\cdot w_{1}(q)+w_{2}(q))+w_{1}(q)\cdot HW_{1}(q)+HW_{2}(q)$$
in the polynomial ring $$H^*(B_{{\bf O}(q)}, {\bf Z}/2{\bf Z})\simeq A[HW_{1}(q), ...HW_{n}(q)].$$
\end{thm}
These formulas are the source for many other comparison formulas that either
have been proved in previous paper (\cite{EKV},  \cite{CNET1}) or
that we shall  establish, using the following principle. For any topos $\mathcal{E}$ given with an ${\bf O}(q)$--torsor, we have a canonical map $f:\mathcal{E}\rightarrow B_{{\bf O}(q)}$, and we obtain comparison formulas in $H^*(\mathcal{E},{\bf Z}/2{\bf Z})$ by applying the functor $f^*$ to the universal comparison formulas of Theorem \ref{mainthm}.

We split Theorem \ref{mainthm} in two theorems according to the degree. 

 \begin{thm}\label{mainthm1} Let $(V, q)$ be a symmetric bundle of rank $n$ on the scheme $Y$ in which $2$ is invertible. Then
$$HW_{1}(q)=w_{1}(q)+ \textrm{\emph{det}}[q]$$
in $H^{1}(B_{{\bf O}(q)}, {\bf Z}/2{\bf Z})$.
 \end{thm}
\begin{proof} The  group $H^{1}(B_{{\bf O}(q)}, {\bf Z}/2{\bf Z})$ can be understood as the group of isomorphism classes of ${\bf Z}/2{\bf Z}$--torsors
of the topos $B_{{\bf O}(q)}$. Hence the theorem may be proved  by describing an isomorphism between the  torsors representing both sides of the required equality. The cohomology class $HW_{1}(q)= T_q^*(HW_{1})=T_q^*(\mbox{det}[t_{n}])$ is represented by the morphism
$$B_{\mbox{det}_{{\bf O}(n)}}\circ T_q:B_{{\bf O}(q)}\rightarrow B_{{\bf O}(n)}\rightarrow B_{{\bf Z}/2{\bf Z}}$$
where $B_{\mbox{det}_{{\bf O}(n)}}$ is the map of classifying topoi induced by the morphism of groups $\mbox{det}_{{\bf O}(n)}:{\bf O}(n)\rightarrow{\bf Z}/2{\bf Z}$. Therefore, $HW_{1}(q)$ is represented by the ${\bf Z}/2{\bf Z}$--torsor
$$(B_{\mbox{det}_{{\bf O}(n)}}\circ T_q)^*E_{{\bf Z}/2{\bf Z}}=T_q^* B_{\mbox{det}_{{\bf O}(n)}}^*E_{{\bf Z}/2{\bf Z}}
=T_q^* ({\bf O}(n)/{\bf SO}(n))= T_{q}\wedge^{{\bf O}(n)}({\bf O}(n)/{\bf SO}(n)).$$
Similarly, $w_{1}(q)$ is represented by the ${\bf Z}/2{\bf Z}$--torsor $\Theta_{q}\wedge^{{\bf O}(n)}({\bf O}(n)/{\bf SO}(n))$.
Notice that ${\bf O}(q)$ acts on  $T_{q}\wedge^{{\bf O}(n)}{\bf O}(n)/{\bf SO}(n)$ via its left action on $T_{q}$ while it acts trivially on $\Theta_{q}\wedge^{{\bf O}(n)}{\bf O}(n)/{\bf SO}(n)$. In both cases
${\bf Z}/2{\bf Z}$ acts by right multiplication on  ${\bf O}(n)/{\bf SO}(n)\simeq {\bf Z}/2{\bf Z}$. The group  ${\bf O}(q)$ acts on
$T_{q}\wedge^{{\bf O}(n)}({\bf O}(n)/{\bf SO}(n))$ as follows:
$$\appl{{\bf O}(q)\times (T_{q}\wedge^{{\bf O}(n)}({\bf O}(n)/{\bf SO}(n)))}{T_{q}\wedge^{{\bf O}(n)}({\bf O}(n)/{\bf SO}(n)).}{(f,[\sigma,\overline{g}])}{[f\circ\sigma,\overline{g}]}$$
We now consider ${\bf Z}/2{\bf Z}$ as an object of $Y_{fl}$ endowed with a right action of ${\bf O}(q)$ via $\mbox{det}_{{\bf O}(q)}$ and right multiplication on the one hand, and with a left action of ${\bf O}(n)$ via $\mbox{det}_{{\bf O}(n)}$ and left multiplication on the other. 
Then ${\bf O}(q)$ acts on the left on
$\Theta_{q}\wedge^{{\bf O}(n)}{\bf Z}/2{\bf Z}$ as follows:
$$\appl{{\bf O}(q)\times (\Theta_{q}\wedge^{{\bf O}(n)}{\bf Z}/2{\bf Z})}{\Theta_{q}\wedge^{{\bf O}(n)}{\bf Z}/2{\bf Z}.}{(f,[\sigma,\epsilon])}{[\sigma,\epsilon\cdot\mbox{det}_{{\bf O}(q)}(f)^{-1}]}$$
Let us show that the map
$$\fonc{\iota}{T_{q}\wedge^{{\bf O}(n)}{\bf O}(n)/{\bf SO}(n)}{\Theta_{q}\wedge^{{\bf O}(n)}{\bf Z}/2{\bf Z}}{[\sigma,\bar{g}]}{[\sigma,\mbox{det}_{{\bf O}(n)}(g)]}$$
is an isomorphism of ${\bf Z}/2{\bf Z}$--torsors of $B_{{\bf O}(q)}$. The map $\iota$ is an isomorphism of ${\bf Z}/2{\bf Z}$--torsors of $Y_{fl}$ since $\mbox{det}_{{\bf O}(n)}$ induces ${\bf O}(n)/{\bf SO}(n)\simeq{\bf Z}/2{\bf Z}$. It remains to check that $\iota$ respects the left action of ${\bf O}(q)$. On points, we have
$$\iota(f*[\sigma,\bar{g}])=\iota[f\circ\sigma,\bar{g}]= \iota[\sigma\circ(\sigma^{-1}\circ f^{-1}\circ \sigma)^{-1}, \overline{\sigma^{-1}f^{-1}\sigma}\cdot\overline{g}]=[\sigma,\mbox{det}_{t_n}(g)\mbox{det}_{{\bf O}(q)}(f)^{-1}]=f*\iota[\sigma,\bar{g}]$$
since $\sigma^{-1}\circ f\circ \sigma $ is a section of ${\bf O}(n)$. We remark that
$${\bf O}(n)/{\bf SO}(n)
\wedge^{{\bf Z}/2{\bf Z}}{\bf O}(q)/{\bf SO}(q)$$
is canonically isomorphic to ${\bf Z}/2{\bf Z}$ and is naturally given with a right action of ${\bf O}(q)$ and a left action of ${\bf O}(n)$. Hence $\iota$ yields a canonical isomorphism of ${\bf Z}/2{\bf Z}$--torsors
in the topos $B_{{\bf O}(q)}$:
\begin{equation}\label{iota-iso}
T_{q}\wedge^{{\bf O}(n)}{\bf O}(n)/{\bf SO}(n)\simeq\Theta_{q}\wedge^{{\bf O}(n)}({\bf O}(n)/{\bf SO}(n)\wedge^{{\bf Z}/2{\bf Z}}{\bf O}(q)/{\bf SO}(q)).
\end{equation}
Recall that the class $\mbox{det}[q]\in H^{1}(B_{{\bf O}(q)}, {\bf Z}/2{\bf Z})$ is represented by the ${\bf Z}/2{\bf Z}$--torsor
${\bf O}(q)/{\bf SO}(q)$.  We have
\begin{eqnarray}
w_{1}(q)+ \mbox{det}[q]&=&[(\Theta_q\wedge^{{\bf O}(n)}{\bf O}(n)/{\bf SO}(n))
\wedge^{{\bf Z}/2{\bf Z}}{\bf O}(q)/{\bf SO}(q)]\\
\label{iso1}&=&[ \Theta_{q}\wedge ^{{\bf O}(n)}({\bf O}(n)/{\bf SO}(n)
\wedge^{{\bf Z}/2{\bf Z}}{\bf O}(q)/{\bf SO}(q))]\\
\label{iso2}&=&[ T_{q}\wedge^{{\bf O}(n)}{\bf O}(n)/{\bf SO}(n)]\\
&=&HW_1(q)
\end{eqnarray}
Here (\ref{iso1}) is given by associativity of the contracted product (see \cite{Giraud} 1.3.5)  and the isomorphism (\ref{iso2}) is just (\ref{iota-iso}).
\end{proof}

\begin{thm}\label{mainthm2} Let $(V, q)$ be a symmetric bundle of rank $n$ on the scheme $Y$ in which $2$ is invertible. Assume that $Y$ is the disjoint union of its connected components. Then
$$HW_{2}(q)=w_{2}(q)+w_{1}(q)\cup \mbox{\emph{det}}[q]+[C_{q}]$$
in $H^{2}(B_{{\bf O}(q)}, {\bf Z}/2{\bf Z})$.
\end {thm}
\begin{proof}
For the convenience of the reader we split the proof into several steps.

\bigskip
{\bf Step 0: Cech cocycles.}
\bigskip

Let $\mathcal{E}$ be a topos, let $A$ be an abelian object in $\mathcal{E}$ and let $U$ be a covering of the final object of $\mathcal{E}$, i.e. $U\rightarrow*$ is an epimorphism. We consider the covering $\mathcal{U}:=\{U\rightarrow*\}$. The Cech complex $\check{C}^*(\mathcal{U},A)$ with value in $A$ with respect to the covering $\mathcal{U}$ is
$$0\longrightarrow A(U)\stackrel{d_0}{\longrightarrow} A(U\times U)\stackrel{d_1}{\longrightarrow} A(U\times U\times U)\stackrel{d_2}{\longrightarrow}... \stackrel{d_n}{\longrightarrow} A(U^{n+2})\stackrel{d_{n+1}}{\longrightarrow}...$$
where, for $f\in A(U^{n+1})=\textrm{Hom}_{\mathcal{E}}(U^{n+1},A)$, we have $$d_n(f)=\sum_{1\leq i\leq n+2}(-1)^{i-1}f\circ p_{1,2,...,\widehat{i},...,n+2}\in A(U^{n+2})=\textrm{Hom}_{\mathcal{E}}(U^{n+2},A).$$
Here $p_{1,2,...,\widehat{i},...,n+2}:U^{n+2}\rightarrow U^{n+1}$ is the projection obtained by omitting the i-th coordinate. We denote by $\mathcal{Z}^n(\mathcal{U},A):=\textrm{Ker}(d_n)$ the group of $n$-cocycles, and Cech cohomology with respect to  the covering $U$ with coefficients in $A$ is defined as follows:
$$\check{H}^n(\mathcal{U},A):=\textrm{Ker}(d_n)/\textrm{Im}(d_{n-1}).$$
For a refinement $V\rightarrow U$, i.e. $V$ covers the final object and is given with a map to $U$, we have an induced map of complexes
$\check{C}^*(\mathcal{U},A)\rightarrow \check{C}^*(\mathcal{V},A)$ hence a map $\check{H}^n(\mathcal{U},A)\rightarrow \check{H}^n(\mathcal{V},A)$ for any $n$, which can be shown to be independent of the map $V\rightarrow U$, using the fact that the cohomology $\check{H}^*(\mathcal{U},A)=\check{H}^*(\mathcal{R}_U,A)$ (resp. $\check{H}^*(\mathcal{V},A)=\check{H}^*(\mathcal{R}_{V},A)$) only depends on the sieve $\mathcal{R}_U$ (resp. $\mathcal{R}_V$) generated by $U$ (resp. by $V$). We  then define
$$\check{H}^n(\mathcal{E},A):=\underrightarrow{\textrm{lim}}\,\check{H}^n(\mathcal{U},A).$$
There are  always  canonical maps $\check{H}^n(\mathcal{E},A)\rightarrow H^n(\mathcal{E},A)$ (given by the Cartan-Leray spectral sequence) but this map is not an isomorphism in general. However, the map
$$\check{H}^1(\mathcal{E},A)\rightarrow H^1(\mathcal{E},A)$$
is an isomorphism for any topos $\mathcal{E}$. Fix a ring $R$ in the topos $\mathcal{E}$, and let $A$ and $B$ be $R$-modules. We have cup-products $$\check{H}^n(\mathcal{E},A)\times \check{H}^m(\mathcal{E},B)\stackrel{\cup}{\longrightarrow}  \check{H}^{n+m}(\mathcal{E},A\otimes_R B)$$
induced by the maps:
$$\appl{\check{C}^n(\mathcal{U},A)\times\check{C}^m(\mathcal{U},B)}{\check{C}^{n+m}(\mathcal{U},A\otimes_{R}B)}{(f:U^{n+1}\rightarrow A,g:U^{m+1}\rightarrow B)}
{f\circ p_{1,...,n+1}\otimes g\circ p_{n+1,...,n+m+1}}
$$
where $p_{1,...,n+1}$ is the projection on the $(1,...,n+1)$-components. We obtain a cup-product
$$H^1(\mathcal{E},A)\times H^1(\mathcal{E},B)\stackrel{\sim}{\rightarrow}\check{H}^1(\mathcal{E},A)\times \check{H}^1(\mathcal{E},B)\rightarrow \check{H}^{2}(\mathcal{E},A\otimes_R B)\rightarrow H^{2}(\mathcal{E},A\otimes_R B). $$
For $A=B=R$,  composing with the multiplication map $R\times R\rightarrow R$, we obtain the following
\begin{lem}
For any ringed topos $(\mathcal{E},R)$ there is a cup-product
$$H^1(\mathcal{E},R)\times H^1(\mathcal{E},R)\longrightarrow H^{2}(\mathcal{E},R)$$
compatible with cup-product of Cech cocycles.
\end{lem}
For a (not necessarily commutative) group $G$ in $\mathcal{E}$, we denote by $H^1(\mathcal{E},G)$ the pointed set of isomorphism classes of $G$--torsors in $\mathcal{E}$. Let $\mathcal{U}=\{U\rightarrow *\}$ be a  covering. The definitions of $\mathcal{C}^1(\mathcal{U},G)$, $\mathcal{Z}^1(\mathcal{U},G)$ and $\check{H}^1(\mathcal{U},G)$ extend to the non-abelian case. Indeed, we set
$\mathcal{C}^1(\mathcal{U},G):=G(\mathcal{U}\times \mathcal{U})$ and
$$\mathcal{Z}^1(\mathcal{U},G):=\{s\in G(U\times U),\,(s\circ p_{23})(s\circ p_{13})^{-1}(s\circ p_{12})=1\},$$
where $p_{ij}:U\times U \times U\rightarrow U \times U$ is the projection on the $(i,j)$-components. There is a natural action of $\mathcal{C}^0(\mathcal{U},G):=G(U)$ on $\mathcal{Z}^1(\mathcal{U},G)$ which is defined as follows: if $\sigma\in G(U)$ and $s\in \mathcal{Z}^1(\mathcal{U},G)$ then
$$\sigma\star s=(\sigma \circ p_1)\cdot s\cdot (\sigma \circ p_2)^{-1}\in \mathcal{Z}^1(\mathcal{U},G)$$
where $p_1,p_2:U \times U\rightarrow U$ are the projections. Then one defines
$$\check{H}^1(\mathcal{U},G):=\mathcal{Z}^1(\mathcal{U},G)/G(U)$$
to be the quotient of $\mathcal{Z}^1(\mathcal{U},G)$ by this group action. Notice that $\mathcal{Z}^1(\mathcal{U},G)$ hence $\check{H}^1(\mathcal{U},G)$ both have the structure of a pointed set. Then
there is an isomorphism
$$\check{H}^1(\mathcal{E},A):=\underrightarrow{\textrm{lim}}\,\check{H}^1(\mathcal{U},G)\stackrel{\sim}{\longrightarrow} H^1(\mathcal{E},A)$$
of pointed sets.\\

{\bf The 1-cocycle associated to a torsor.} Let $G$ be a group in the topos $\mathcal{E}$ and let $T$ be a $G$--torsor in $\mathcal{E}$. In order to obtain a
$1$-cocycle which represents  $T$, we proceed as follows. By definition
$T\rightarrow *$ is a covering of the final object in $\mathcal{E}$ which trivializes $T$. More precisely the canonical map
$$\fonc{\mu}{T\times G}{T\times T}{(t,g)}{(t,t\cdot g)}
$$
is an isomorphism in $\mathcal{E}/T$. Here the assignment $(t,g)\mapsto (t,t\cdot g)$ makes sense on sections. Indeed, for any object $X$ in $\mathcal{E}$, the set $T(X)$ carries a right action of the  group $G(X)$ such that the map
$$\fonc{\mu(X)}{T(X)\times G(X)}{T(X)\times T(X)}{(t,g)}{(t,t\cdot g)}
$$
is a bijection, i.e. $G(X)$ acts simply and transitively on $T(X)$. Let $\mu^{-1}: T\times T\rightarrow T\times G$ be the inverse map. For any $X$ in $\mathcal{E}$, we have
$$\fonc{\mu^{-1}(X)}{T(X)\times T(X)}{T(X)\times G(X)}{(t,u)}{(t,t^{-1}u)}
$$
where, by the the notation $g=t^{-1}u$, we mean the unique element of the group $G(X)$ such that $t\cdot g=u$.
If $f: U\rightarrow T$ is a morphism of $\mathcal{E}$ such that $U\rightarrow *$ is a covering, we obtain a $1$-cocycle representing  $T$,  by considering  $$c_T\in \mathcal{Z}^1(\{U\rightarrow*\},G) \subset G(U\times U)$$ defined (on sections) by
$$\fonc{c_T}{U\times U}{G}{(t, u)}{f(t)^{-1}f(u)}
$$
where $\mathcal{Z}^1(\{U\rightarrow*\},G)$ is the pointed set of 1-cocycles with respect to the cover $\{U\rightarrow*\}$ with values in $G$. Recall that $$\mathcal{Z}^1(\{U\rightarrow*\},G):=\{s\in G(U\times U),\,(s\circ p_{23})(s\circ p_{13})^{-1}(s\circ p_{12})=1\},$$
where $p_{ij}:U\times U \times U\rightarrow U \times U$ is the projection on the $(i,j)$-components.
The cocycle $c_T$ represents the $G$--torsor $T$ in the sense that
$$\mathcal{Z}^1(\{U\rightarrow*\},G)\twoheadrightarrow \check{H}^1(\{U\rightarrow*\},G)\rightarrow \check{H}^1(\mathcal{E},G)\stackrel{\sim}{\rightarrow}H^1(\mathcal{E},G)$$
maps $c_T$ to $[T]$, where $[T]$ is the class of the $G$--torsor $T$.\\

{\bf The 2-cocycle associated to a central extension with local sections.} We consider an exact sequence of groups in $\mathcal{E}$
\begin{equation}\label{extensionC}
1\rightarrow A\rightarrow \widetilde G\rightarrow G\rightarrow 1
\end{equation}
such that $A$ is \emph{central} in $\widetilde G$. We have a boundary map $H^1(\mathcal{E},G)\rightarrow H^2(\mathcal{E},A)$ defined as follows:
$$\fonc{\delta}{H^{1}(\mathcal{E}, G)}{H^{2}(\mathcal{E}, A)}{[T]}{T^*([C])}$$
where $[C]\in H^2(B_{G},A)$ is the class defined (see Proposition \ref{prop-Giraud-sequence}) by the extension (\ref{extensionC}).
The class $\delta (T)$ is trivial if and only if the $G$--torsor can be lifted into a $\widetilde G$--torsor $\widetilde T$ (i.e. if and only if  there exists a $\widetilde G$--torsor $\widetilde T$ and an isomorphism of $G$--torsors $\widetilde T\wedge^{\widetilde{G}}G\simeq T$). Let $T$ be a $G$--torsor, let $\mathcal{U}=\{U\rightarrow *\}$ be a covering and let $c_T\in\mathcal{Z}^1(\mathcal{U},G)$ be a 1-cocycle representing $T$. Assume that there exists a lifting $\widetilde{c}_T:U\times U\rightarrow \widetilde{G}$ of $c_T$. Then
$$\check{\delta}(c_T):=\widetilde{c}_T\circ p_{23}-\widetilde{c}_T\circ p_{13}+\widetilde{c}_T\circ p_{12}$$
is a 2-cocycle with values in $A$, i.e. one has $\check{\delta}(c_T)\in\mathcal{Z}^2(\mathcal{U},A)$. The class of $\check{\delta}(c_T)$ in $H^2(\mathcal{U},A)$ does not depend on the choice of $\widetilde{c}_T$. Moreover, the image of $\check{\delta}(c_T)$ in $H^2(\mathcal{E},A)$ is $\delta([T])=T^*([C])$ (see Giraud IV.3.5.4).

\bigskip
{\bf Step 1: First reduction.}
\bigskip

The aim of this step is to prove the following result. Denote by
$$B_{r_q}:B_{\widetilde{{\bf O}}(q)}\longrightarrow B_{{\bf O}(q)}$$ the map induced by the morphism $r_q: \widetilde {\bf O}(q)\rightarrow {\bf O}(q)$ (see Step 2 for a precise definition of $r_q$).
\begin{prop}
The identity
\begin{equation}\label{whatwewant}
HW_{2}(q)=w_{2}(q)+w_{1}(q)\cup \mbox{\emph{det}}[q]+[C_{q}]
\end{equation}
in  $H^{2}(B_{{\bf O}(q)}, {\bf Z}/2{\bf Z})$ is equivalent to the identity
\begin{equation}\label{whatwehave}
B_{r_q}^*(HW_{2}(q))=B_{r_q}^*(w_{2}(q)+w_{1}(q)\cup \mbox{\emph{det}}[q]+[C_{q}])
\end{equation}
in  $H^{2}(B_{\widetilde{{\bf O}}(q)}, {\bf Z}/2{\bf Z})$.
\end{prop}
\begin{proof}
By functoriality (\ref{whatwewant}) implies (\ref{whatwehave}). Let us show the converse. Assume that (\ref{whatwehave}) holds, so that (see Lemma \ref{lem-Cdies})
$$B_{r_q}^*(HW_{2}(q)+w_{2}(q)+w_{1}(q)\cup \mbox{det}[q])=0$$
Recall that $Y=\Coprod_{\alpha}Y_{\alpha}$ is the disjoint union of its connected components. Since cohomology sends disjoint sums to direct products, we may assume $Y$ to be connected. By Lemma \ref{lem-reduc}, we have either $$HW_{2}(q)+w_{2}(q)+w_{1}(q)\cup \mbox{det}[q]=0$$ or
$$HW_{2}(q)+w_{2}(q)+w_{1}(q)\cup \mbox{det}[q]=[C_q].$$ Assume that
$$HW_{2}(q)+w_{2}(q)+w_{1}(q)\cup \mbox{det}[q]=0.$$
By Theorem \ref{mainthm1}, we would obtain
$$HW_{2}(q)+w_{1}(q)\cdot HW_{1}(q)+(w_{1}(q)\cdot w_{1}(q)+w_{2}(q))=0$$
in the polynomial ring $A[HW_1(q),...,HW_n(q)]$. This is a contradiction since $w_{1}(q),w_{2}(q)\in A$. Identity (\ref{whatwewant}) follows.
\end{proof}

\begin{lem}\label{lem-Cdies}
Consider an extension
$$1\rightarrow A\rightarrow \widetilde{G}\rightarrow G\rightarrow 1$$
of a group $G$ by an abelian group $A$ in some topos $\mathcal{E}$. Let $C\in \textrm{\emph{Ext}}(G,A)$ be the class of this extension and let $[C]$ be its cohomology class in $H^2(B_G,A)$. Then $C$ vanishes in $\textrm{\emph{Ext}}(\widetilde{G},A)$. A fortiori, $[C]$ vanishes in $H^2(B_{\widetilde{G}},A)$.
\end{lem}
\begin{proof}
The natural map
$$\textrm{Ext}(G,A)\rightarrow \textrm{Ext}(\widetilde{G},A)$$
sends the class $C$ of the extension $1\rightarrow A\rightarrow \widetilde{G}\rightarrow G\rightarrow 1$ to the class $\widetilde{C}$ of
$$1\rightarrow A\rightarrow \widetilde{G}\times_G\widetilde{G}\rightarrow \widetilde{G}\rightarrow 1$$
which is split by the diagonal $\widetilde{G}\rightarrow \widetilde{G}\times_G\widetilde{G}$, so that $\widetilde{C}=0$.

Giraud's exact sequence (Proposition \ref{prop-Giraud-sequence}) is functorial in $G$, so that one has a commutative square:
\[\xymatrix{
\textrm{Ext}(G,A)\ar[r]\ar[d]&H^2(B_G,A)\ar[d] \\
\textrm{Ext}(\widetilde{G},A)\ar[r]&H^2(B_{\widetilde{G}},A)
}\]
hence $[C]\in H^2(B_G,A)$ maps to $[\widetilde{C}]=0\in H^2(B_{\widetilde{G}},A)$.

\end{proof}
Recall that we denote by $C_q$ the class of the canonical extension
$$1\rightarrow {\bf Z}/2{\bf Z}\rightarrow \widetilde{{\bf O}}(q)\rightarrow{\bf O}(q)\rightarrow1.$$
\begin{lem}\label{lem-reduc}
The Hochschild-Serre spectral sequence associated to the above extension of group--schemes  induces an exact sequence
$$0\rightarrow H^0(Y,{\bf Z}/2{\bf Z})\rightarrow H^2(B_{{\bf O}(q)}, {\bf Z}/2{\bf Z})\rightarrow H^{2}(B_{\widetilde{{\bf O}}(q)}, {\bf Z}/2{\bf Z}).$$
Moreover, if $Y$ is connected, then ${\bf Z}/2{\bf Z}\rightarrow H^2(B_{{\bf O}(q)}, {\bf Z}/2{\bf Z})$ maps $1$ to $[C_q]$.
\end{lem}
\begin{proof}
The scheme $Y$ is the disjoint union of its connected components. Since cohomology sends disjoint unions to direct products, we may suppose $Y$ to be connected. By Corollary \ref{cor-short-exact-HS-spectralsequ} we have an exact sequence
$$0\rightarrow H^1(B_{{\bf O}(q)},{\bf Z}/2{\bf Z})\rightarrow H^1(B_{\widetilde{{\bf O}}(q)},{\bf Z}/2{\bf Z})\rightarrow
H^0(B_{{\bf O}(q)},\underline{\textrm{Hom}}({\bf Z}/2{\bf Z},{\bf Z}/2{\bf Z}))$$
$$\rightarrow H^2(B_{{\bf O}(q)},{\bf Z}/2{\bf Z})\rightarrow H^2(B_{\widetilde{{\bf O}}(q)},{\bf Z}/2{\bf Z}).$$
But $\underline{\textrm{Hom}}({\bf Z}/2{\bf Z},{\bf Z}/2{\bf Z})={\bf Z}/2{\bf Z}$ and $H^0(B_{{\bf O}(q)},{\bf Z}/2{\bf Z})={\bf Z}/2{\bf Z}$ hence we obtain an exact sequence
$${\bf Z}/2{\bf Z}\longrightarrow H^2(B_{{\bf O}(q)},{\bf Z}/2{\bf Z})\stackrel{r^*_q}{\longrightarrow} H^2(B_{\widetilde{{\bf O}}(q)},{\bf Z}/2{\bf Z}).$$
By Lemma \ref{lem-Cdies}, the class $[C_q]$ lies in $\textrm{Ker}(r^*_q)$. It remains to show that $[C_q]\neq 0$ in $H^2(B_{{\bf O}(q)},{\bf Z}/2{\bf Z})$. Let $U\rightarrow Y$ be an \'etale map (or any map) such that $q_U$ is isomorphic to the standard form $t_{n,U}$ on $U$. Then there is an isomorphism $$B_{{\bf O}(q_U)}\simeq B_{{\bf O}(q)}/U\simeq B_{{\bf O}(n)}/U\simeq B_{{\bf O}(t_{n,U})}$$ such that $[C_{q_U}]$ maps to $[C_{t_n,U}]=HW_2\in H^2(B_{{\bf O}(t_{n,U})},{\bf Z}/2{\bf Z})$, which is non-zero by Theorem \ref{thm-Jardine}. Hence $[C_q]$ maps to $[C_{q_U}]\neq 0$, hence $[C_{q}]\neq 0$.

\end{proof}

{\bf Step 2: The maps $\theta$ and $\widetilde\theta$.}\\

For a brief summary of the  Clifford algebras and Clifford groups associated to a symmetric bundle see \cite{EKV}, Section 1.9.
 Let  $(V, q)$ be a  symmetric bundle on $Y$.  We denote by $C(q)$ its
Clifford algebra  and by  $C^{*}(q)$  its Clifford group. We  then  consider
 the  sheaf of  algebras and  the sheaf of groups,  for the flat topology,  defined by the  functors $C(q): T\rightarrow C(q_{T})$ and
  $C^{*}(q): T\rightarrow C^{*}(q_{T})$.
The norm map $N: C(q)\rightarrow C(q)$ restricts to a morphism of groups
$$N: C^{*}(q)\rightarrow {\bf G}_{m}$$
where ${\bf G}_{m}$ denotes the multiplicative group. We let  $\widetilde {\bf O}(q)$ denote  the kernel of this homomorphism. This is a sheaf of groups for the flat topology
which is representable by a
  smooth group scheme over $Y$.
The group scheme $\widetilde {\bf O}(q)$
splits as  $ \widetilde {{\bf O}}_{+}(q)\coprod  \widetilde {{\bf
O}}_{-}(q)\ .$ Let $x$ be in  $ \widetilde {{\bf
O}}_{\varepsilon}(q)$ with $\varepsilon =\pm 1$, then we define $r_q(x)$
as the element of ${\bf O}(q)$
$$r_q(x): V \rightarrow V$$
$$v\mapsto \varepsilon xvx^{-1} \ .$$
This defines a group homomorphism $r_q:\widetilde {{\bf O}}(q)
\rightarrow {\bf O}(q)$. One can show  that for each $x \in
\widetilde {{\bf O}}_{\varepsilon}(q)$ the element  $r_q(x)$ belongs
to $ {\bf O}_+(q)={\bf SO}(q)$ or ${\bf  O}_-(q)={\bf O}(q)
\setminus  {\bf SO}(q)$ depending on whether  $\varepsilon=1$ or
$-1$.

We start by considering the affine case  $Y=\mbox{Spec}(R)$ where $R$ is a commutative ring in which $2$ is invertible.
In this situation  ${\bf O}(q)$ and ${\bf \widetilde O}(q)$ can be respectively considered as the orthogonal group and the Pinor group of the form $(V,q)$,  where
$V$ is a finitely generated projective $R$-module and $q$  is a non-degenerate form on $V$.

 \begin{lem}\label{lem-Ph}
 Let $(V, q)$ and $(W, f)$ be symmetric bundles over $R$ and let $t\in {\bf Isom}(q, f)(R)$.

\begin{itemize}

\item[i)]  $t$ extends in a unique way to a graded isomorphism of Clifford algebras
 $$\widetilde \psi_{t}: C(q)\rightarrow C(f)\ . $$
\item[ii)] $\widetilde \psi_{t}$ induces,  by restriction,  an  isomorphism of groups, again  denoted $\widetilde \psi_{t }$,
such that the following diagram is commutative

\[\xymatrix{\{1\}\ar[r]^{}&
({\bf Z}/2{\bf Z})(R)\ar[r]^{}\ar[d]_{\textrm{\emph{Id}}}&\widetilde {\bf O}(q)(R)\ar[r]^{r_{q}}\ar[d]_{\widetilde \psi_{t}}& {\bf O}(q)(R)\ar[d]_{\psi_{t}} \\
\{1\}\ar[r]^{}&({\bf Z}/2{\bf Z})(R)\ar[r]^{}&\widetilde {\bf O}(f)(R)\ar[r]^{r_{f}}& {\bf O}(f)(R) \\
}\]
 where $\psi_{t}$ is the group isomorphism $u\rightarrow tut^{-1}$.

\item [iii)]  For any $a\in {\bf O}(f)(R)$ and $t\in {\bf Isom}(q, f)(R)$ we have the equalities
$$\widetilde \psi_{at}= \widetilde \psi_{a}\circ \widetilde \psi_{t}\mbox{ and }\psi_{at}=a\psi_{t}a^{-1} . $$ .

\item [iv)]  Suppose that $(V, q)=(W, f)$.  We assume  that $t $ has  a lift $s(t)$ in  $\widetilde {\bf O}(q)(R)$. Then
$\widetilde \psi_{t} =i_{s(t)}$, (resp. $\varepsilon
i_{s(t)}$ on $\widetilde {{\bf O}}_{\varepsilon}(q)(R)$), if
$t \in {\bf O}_+(q)(R)$, (resp. ${\bf O}_-(q)(R)$ ).

\end{itemize}

\end{lem}

\begin{proof}  It follows from the very definition of the Clifford algebra that  $t$ extends to  a graded  isomorphism
$\widetilde \psi_{t}: C(q)\rightarrow C(f)$  which itself induces
by restriction an isomorphism $\widetilde \psi_{t}: \widetilde {\bf O}(q)\rightarrow \widetilde {\bf O}(f)$. For
$a\in \widetilde {\bf O}_{\varepsilon}(q)$ one has by definition:
$$r_{f}\circ \widetilde \psi_{t}(a)(x)=\varepsilon \widetilde \psi_{t}(a)\ x\ \widetilde \psi_{t}(a)^{-1}, \ \forall x \in W .$$
Since $\widetilde \psi_{t}$ is induced by a morphism of $R$-algebras,  the right-hand side of this equality can be written as
$$\varepsilon \widetilde \psi_{t}(a\ \widetilde \psi_{t}^{-1}(x)\ a^{-1})\ .$$  Since $\widetilde \psi_{t}$ coincides with $t$ on $V$ and since
$a\ {t}^{-1}(x)\ a^{-1} \in V$ (recall that $a\ \in \widetilde {\bf O}(q)$ we conclude that
$$r_{f}\circ \widetilde\psi_{t}(a)(x)=\psi_{t}\circ r_{q}(a)(x)\'e\ \forall  x\in W.$$ Therefore $i)$ and $ii)$ are proved.
The second equality of iii) is immediate.  In order to prove the first equality  it suffices to   observe  that both sides of this equality coincide
when restricted to $V$.

\noindent If we now assume that $t$ has a lift in $\widetilde{\bf O}(q)(R)$, then  we obtain  two automorphisms of
$C(q)$, namely $\widetilde \psi_{t} $ and $\imath _{s(t)}$.
Moreover,  since $s(t) \in \widetilde {{\bf
O}}_{\varepsilon}(q)$, it follows from   the definition of $r_q$  that
$t (x)=\varepsilon s(t) x s(t)^{-1}$ for any $x$ in $V$.
If $t\ \in {\bf O}_+(q)$, then $\varepsilon=1$ and $\widetilde \psi_{t} $
and $i_{s(t)}$   coincide on $V$ and therefore  coincide on
$ C(q)$. If   $t \in {\bf O}_-(q)$, then $ \varepsilon=-1$ and
therefore  $\widetilde \psi_{t}$ and $i_{s(t)}$ will coincide
on $ C_+(q)$ and will differ by a minus sign on  $
C_-(q)$ and the result follows.
\end{proof}

We now return to  the forms $q$ and $t_{n}$ on $Y$  and the sheaf  ${\bf Isom}(t_{n}, q)$
of $Y_{fl}$. Let us define morphisms in $Y_{fl}$:
$$\underline{\theta}:{\bf Isom}(t_{n}, q)\times {\bf O}(q)\rightarrow {\bf Isom}(t_{n}, q)\times {\bf O}(n),$$
$$\underline{\widetilde\theta}:{\bf Isom}(t_{n}, q)\times \widetilde{\bf O}(q)\rightarrow {\bf Isom}(t_{n}, q)\times \widetilde{\bf O}(n)$$
as follows. The morphisms $\underline{\theta}$ and $\underline{\widetilde\theta}$ can be defined on sections. Furthermore, the class of affine schemes yields a generating subcategory of $Y_{fl}$, hence one can define the morphisms $\underline{\theta}$ and $\underline{\widetilde\theta}$ on sections over affine schemes of the form $\textrm{Spec}(R)\rightarrow Y$.
For any $\textrm{Spec}(R)\rightarrow Y$ and $t\in {\bf Isom}(t_{n}, q)(R)$ we set  $\underline{\theta}_{t}=\psi_{t^{-1}}$ and
$\underline{\widetilde\theta}_{t}=\widetilde\psi_{t^{-1}}$.  
Then the maps
$$
\fonc{\underline{\theta}(R)}{{\bf Isom}(t_{n}, q)(R)\times {\bf O}(q)(R)}{{\bf Isom}(t_{n}, q)(R)\times {\bf O}(n)(R)}{(t, x)}{(t, \underline{\theta}_{t}(x))}
$$
and
$$
\fonc{\underline{\widetilde\theta}(R)}{{\bf Isom}(t_{n}, q)(R)\times \widetilde{\bf O}(q)(R)}{{\bf Isom}(t_{n}, q)(R)\times \widetilde{\bf O}(n)(R)}{(t, x)}{(t, \underline{\widetilde{\theta}}_{t}(x))}
$$
are both functorial in $\textrm{Spec}(R)\rightarrow Y$, and yield the morphisms $\underline{\theta}$ and $\underline{\widetilde\theta}$. We denote by $$\widetilde{\pi}:B_{\widetilde{{\bf O}}(q)}\stackrel{B_{r_q}}{\longrightarrow} B_{{\bf O}(q)} \stackrel{\pi}{\longrightarrow} Y_{fl}$$ the canonical map. Recall that
$$\Theta_q:=\pi^*{\bf Isom}(t_{n}, q).$$
Similarly, we denote by $\widetilde{\Theta}_q$ the object of  $B_{\widetilde{{\bf O}}(q)}$ given by 
$$\widetilde{\Theta}_q:=B_{r_q}^*\Theta_q=\widetilde{\pi}^*{\bf Isom}(t_{n}, q) $$
which is  the sheaf ${\bf Isom}(t_{n}, q)$ with trivial $\widetilde{{\bf O}}(q)$--action.
Pulling back $\underline{\theta}$ and $\underline{\widetilde\theta}$ along the morphism
$\widetilde{\pi}:B_{\widetilde{{\bf O}}(q)}\rightarrow Y_{fl}$ we obtain morphisms in $B_{\widetilde{{\bf O}}(q)}$:
$$\widetilde{\pi}^*\underline{\theta}: \widetilde{\Theta}_q\times \widetilde{\pi}^*{\bf O}(q)\rightarrow \widetilde{\Theta}_q\times \widetilde{\pi}^*{\bf O}(n)\mbox{ and }
\widetilde{\pi}^*\underline{\widetilde\theta}: \widetilde{\Theta}_q\times \widetilde{\pi}^*\widetilde{\bf O}(q)\rightarrow \widetilde{\Theta}_q\times \widetilde{\pi}^*\widetilde{\bf O}(n), $$
where  the above maps are just $\underline{\theta}$ and $\underline{\widetilde\theta}$ seen as equivariant map between objects of $Y_{fl}$ with trivial $\widetilde{{\bf O}}(q)$--action. Moreover,
$\widetilde{\pi}^*\underline{\theta}$ and $\widetilde{\pi}^*\widetilde{\underline{\theta}}$ are defined over $\widetilde{\Theta}_q$, i.e. $\widetilde{\pi}^*\underline{\theta}$ and $\widetilde{\pi}^*\underline{\widetilde{\theta}}$ commute with the projection to $\widetilde{\Theta}_q$. In other words, $\widetilde{\pi}^*\underline{\theta}$ and $\widetilde{\pi}^*\underline{\widetilde{\theta}}$ are maps in $B_{\widetilde{{\bf O}}(q)}/\widetilde{\Theta}_q$, which we simply denote by $\theta:=\widetilde{\pi}^*\underline{\theta}$ and $\widetilde\theta:=\widetilde{\pi}^*\underline{\widetilde{\theta}}$. We summarize what we have constructed so far.
\begin{lem}\label{lemtheta}
 The following diagram is a commutative diagram in $B_{\widetilde{{\bf O}}(q)}/\widetilde{\Theta}_q$.
\[\xymatrix{\widetilde{\Theta}_q\ar[r]^{}&
\widetilde{\Theta}_q\times {\bf Z}/2{\bf Z}\ar[r]^{}\ar[d]_{\textrm{Id}}&\widetilde{\Theta}_q\times \widetilde{\pi}^*\widetilde {\bf O}(q)\ar[r]^{r_{q}}
\ar[d]_{\widetilde \theta}& \widetilde{\Theta}_q\times \widetilde{\pi}^*{\bf O}(q)\ar[d]_{\theta}\ar[r]^{}&\widetilde{\Theta}_q \\
\widetilde{\Theta}_q\ar[r]^{}&\widetilde{\Theta}_q\times {\bf Z}/2{\bf Z}\ar[r]^{}&\widetilde{\Theta}_q\times \widetilde{\pi}^*\widetilde{\bf O}(n)\ar[r]^{r_{n}}&
\widetilde{\Theta}_q\times \widetilde{\pi}^*{\bf O}(n)\ar[r]^{}&\widetilde{\Theta}_q\\
}\]
Moreover, the horizontal rows are exact sequences of sheaves of groups on $B_{\widetilde{{\bf O}}(q)}/\widetilde{\Theta}_q$ and the vertical arrows are isomorphisms. Notice that $\widetilde{\Theta}_q$ is the trivial group in $B_{\widetilde{{\bf O}}(q)}/\widetilde{\Theta}_q$. 
\end{lem}
We now use the morphisms $\theta$ and $\widetilde \theta$ to associate to any object $p: W\rightarrow\widetilde{\Theta}_{q}$ of $B_{\widetilde{{\bf O}}(q)}/\widetilde{\Theta}_q$
isomorphisms of groups
$$\theta_{p}: \widetilde{\pi}^*{\bf O}(q)(W)\stackrel{\sim}{\longrightarrow}\widetilde{\pi}^*{\bf O}(n)(W)\ \mbox{and}\ \widetilde \theta_{p}:
\widetilde{\pi}^*\widetilde{\bf O}(q)(W)\stackrel{\sim}{\rightarrow}\widetilde{\pi}^*\widetilde{\bf O}(n)(W),$$
where $W$ is the object of $B_{\widetilde{{\bf O}}(q)}$ underlying to $p$. For example, we have  $\widetilde{\pi}^*{\bf O}(n)(W):=\textrm{Hom}_{B_{\widetilde{{\bf O}}(q)}}(W,\widetilde{\pi}^*{\bf O}(n))$. These isomorphisms can be described as follows:
\[\xymatrix{\theta_p:(f: W\rightarrow \widetilde{\pi}^*{\bf O}(q))\ar[r]^{}& (\theta_{p}(f): W
\ar[r]^{p\times f}&\widetilde{\Theta}_{q}\times \widetilde{\pi}^*{\bf O}(q)\ar[r]^{\theta}
& \widetilde{\Theta}_{q}\times \widetilde{\pi}^*{\bf O}(n)\ar[r]^{\textrm{pr}}&\widetilde{\pi}^*{\bf O}(n)) \\
}\]
\[\xymatrix{\widetilde\theta_p:(f: W\rightarrow \widetilde{\pi}^*{\bf O}(q))\ar[r]^{}& (\widetilde \theta_{p}(f): W
\ar[r]^{p\times f}&\widetilde{\Theta}_{q}\times \widetilde{\pi}^*\widetilde{\bf O}(q)\ar[r]^{\widetilde \theta}
& \widetilde{\Theta}_{q}\times \widetilde{\pi}^*\widetilde{\bf O}(n)\ar[r]^{\textrm{pr}}&\widetilde{\pi}^*\widetilde{\bf O}(n)) \\
}\] where $\textrm{pr}$ denotes  the projection on the second component.  We then have a commutative diagram of groups
\[\xymatrix{1\ar[r]^{}&
{\bf Z}/2{\bf Z}(W)\ar[r]^{}\ar[d]_{\textrm{Id}}& \widetilde{\pi}^*\widetilde{\bf O}(q)(W)\ar[r]^{r_{q, W}}
\ar[d]_{\widetilde \theta_{p}}&  \widetilde{\pi}^*{\bf O}(q)(W)\ar[d]_{\theta_{p}}\\
1\ar[r]^{}& {\bf Z}/2{\bf Z}(W)\ar[r]^{}& \widetilde{\pi}^*\widetilde{\bf O}(n)(W)\ar[r]^{r_{n, W}}&
 \widetilde{\pi}^*{\bf O}(n)(W).\\
}\]

We (also) denote by $\eta: Y_{fl}\rightarrow B_{\widetilde{{\bf O}}(q)}$ the morphism of topoi induced by the morphism of groups $1\to \widetilde{{\bf O}}(q)$. The functor $\eta^*$ forgets the $\widetilde{{\bf O}}(q)$--action. It has a left adjoint $\eta_!$, which is defined as follows: $\eta_!(Z)$ is the object of $B_{\widetilde{{\bf O}}(q)}$ given by $E_{\widetilde{{\bf O}}(q)}\times Z$ on which $\widetilde{{\bf O}}(q)$ acts via the first factor. Furthermore, we may consider the maps $\eta^*\theta$ and $\eta^*\widetilde{\theta}$ as maps in the topos $Y_{fl}/{\bf Isom}(t_{n}, q)$. For a map $t:Z\rightarrow {\bf Isom}(t_{n}, q)$ in $Y_{fl}$ we define
$$(\eta^*\theta)_t:{\bf O}(q)(Z)\rightarrow {\bf O}(n)(Z) \mbox{ and } (\eta^*\widetilde{\theta})_t:\widetilde{{\bf O}}(q)(Z)\rightarrow \widetilde{{\bf O}}(n)(Z)$$
in a similar manner. In particular, for $t:Z=\textrm{Spec}(R)\rightarrow {\bf Isom}(t_{n}, q)$, one has
\begin{equation}\label{weirdo0}
(\eta^*\theta)_t=\psi_{{t}^{-1}}\mbox{ and }(\eta^*\widetilde{\theta})_t=\widetilde\psi_{{t}^{-1}}
\end{equation}
as defined in Lemma \ref{lem-Ph}.

\begin{lem}\label{lem-weirdo}
Let $Z=\textrm{\emph{Spec}}(R)$ be an object of $B_{\widetilde{{\bf O}}(q)}$ with trivial $\widetilde{{\bf O}}(q)$--action, and let $Z\rightarrow\widetilde{\Theta}_q$ be a map in $B_{\widetilde{{\bf O}}(q)}$. By adjunction one has canonical identifications
\begin{equation}\label{weirdo1}
(-)_{\mid R}:\widetilde{\pi}^*\widetilde{{\bf O}}(n)(E_{\widetilde{{\bf O}}(q)}\times Z)\stackrel{\sim}{\longrightarrow}
\widetilde{{\bf O}}(n)(\eta^*Z)=\widetilde{{\bf O}}(n)(R)
\end{equation}
\begin{equation}\label{weirdo2}
(-)_{\mid R}:\widetilde{\pi}^*\widetilde{{\bf O}}(q)(E_{\widetilde{{\bf O}}(q)}\times Z)\stackrel{\sim}{\longrightarrow}\widetilde{{\bf O}}(q)(\eta^*Z)= \widetilde{{\bf O}}(q)(R)
\end{equation}
\begin{equation}\label{weirdo3}
(-)_{\mid R}:\widetilde{\Theta}_q(E_{\widetilde{{\bf O}}(q)}\times Z)\stackrel{\sim}{\longrightarrow}{\bf Isom}(t_{n}, q)(\eta^*Z)={\bf Isom}(t_{n}, q)(R)
\end{equation}
such that
\begin{equation}\label{weirdo5}
\theta_p(\sigma)_{\mid R}=(\eta^*\theta)_{p_{\mid R}}(\sigma_{\mid R}):= \psi_{p_{\mid R}^{-1}}(\sigma_{\mid R})\mbox{ and }
\widetilde{\theta}_p(\sigma)_{\mid R}=(\eta^*\widetilde{\theta})_{p_{\mid R}}(\sigma_{\mid R}):= \widetilde{\psi}_{p_{\mid R}^{-1}}(\sigma_{\mid R})
\end{equation}
for any $$(p,\sigma)\in\widetilde{\Theta}_q(E_{\widetilde{{\bf O}}(q)}\times Z)\times\widetilde{\pi}^*\widetilde{{\bf O}}(q)(E_{\widetilde{{\bf O}}(q)}\times Z).$$

\end{lem}

\begin{proof}
The map (\ref{weirdo1}) is defined as follows. Given $$\sigma:E_{\widetilde{{\bf O}}(q)}\times Z\longrightarrow \widetilde{\pi}^*\widetilde{{\bf O}}(q)\in \widetilde{\pi}^*\widetilde{{\bf O}}(q)(E_{\widetilde{{\bf O}}(q)}\times Z)$$ we define
$$\sigma_{\mid R}:\eta^*Z\stackrel{e}{\longrightarrow}\eta^*E_{\widetilde{{\bf O}}(q)}\times \eta^*Z\stackrel{\eta^*(\sigma)}{\longrightarrow} \eta^*\widetilde{\pi}^*\widetilde{{\bf O}}(q)=\widetilde{{\bf O}}(q)$$
where $e:\eta^*Z\longrightarrow\eta^*E_{\widetilde{{\bf O}}(q)}\times \eta^*Z$ is the map given by the unit section of $\widetilde{{\bf O}}(q)=\eta^*E_{\widetilde{{\bf O}}(q)}$. This morphism $\sigma\mapsto\sigma_{\mid R}$ is an inverse of the adjunction map:
$$\widetilde {\bf O}(n)(\eta ^*Z)=\mbox{Hom}_{Y_{fl}}(\eta ^*Z, \widetilde{\bf O}(n))=\mbox{Hom}_{Y_{fl}}(\eta ^*Z, \eta ^*\widetilde \pi^* \widetilde{\bf O}(n))=
\mbox{Hom}_{B_{\widetilde{{\bf O}}(q)}}(\eta_!\eta ^*Z, \widetilde \pi^* \widetilde{\bf O}(n))=$$
$$\mbox{Hom}_{B_{\widetilde{{\bf O}}(q)}}(E_{\widetilde {\bf O}(q)}\times Z, \widetilde \pi ^*\widetilde{\bf O}(n))=\widetilde \pi ^*\widetilde{\bf O}(n)(E_{\widetilde {\bf O}(q)}\times Z). $$
Replacing successively $\widetilde{{\bf O}}(n)$ with $\widetilde{{\bf O}}(q)$ and ${\bf Isom}(t_{n}, q)$ we obtain respectively (\ref{weirdo2}) and (\ref{weirdo3}). Then (\ref{weirdo5}) follows immediately from the definitions: for some $(p,\sigma)\in\widetilde{\Theta}_q(E_{\widetilde{{\bf O}}(q)}\times Z)\times\widetilde{\pi}^*\widetilde{{\bf O}}(q)(E_{\widetilde{{\bf O}}(q)}\times Z)$ we have
$$\widetilde{\theta}_p(\sigma): E_{\widetilde{{\bf O}}(q)}\times Z\rightarrow \widetilde{\Theta}_q\times \widetilde{\pi}^*\widetilde{{\bf O}}(q)\stackrel{\widetilde{\theta}}{\rightarrow}
\widetilde{\Theta}_q\times \widetilde{\pi}^*\widetilde{{\bf O}}(n)\rightarrow \widetilde{\pi}^*\widetilde{{\bf O}}(n)$$
and
$$\widetilde{\theta}_p(\sigma)_{\mid R}:\eta^*Z\rightarrow\eta^* E_{\widetilde{{\bf O}}(q)}\times \eta^*Z\rightarrow \eta^*\widetilde{\Theta}_q\times \eta^*\widetilde{\pi}^*\widetilde{{\bf O}}(q)\stackrel{\eta^*\widetilde{\theta}}{\rightarrow}
\eta^*\widetilde{\Theta}_q\times \eta^*\widetilde{\pi}^*\widetilde{{\bf O}}(n)\rightarrow \eta^*\widetilde{\pi}^*\widetilde{{\bf O}}(n)$$
which is just
$$(\eta^*\widetilde{\theta})_{p_{\mid R}}(\sigma_{\mid R}):\textrm{Spec}(R)\stackrel{(p_{\mid R},\sigma_{\mid R})}{\longrightarrow} {\bf Isom}(t_{n}, q)\times \widetilde{{\bf O}}(q)\stackrel{\eta^*\widetilde{\theta}}{\longrightarrow}
{\bf Isom}(t_{n}, q)\times \widetilde{{\bf O}}(n)\longrightarrow \widetilde{{\bf O}}(n).$$

\end{proof}
We shall also need the following result.
\begin{lem}\label{lem-generatingfamily}
The class of objects of the form $E_{\widetilde{{\bf O}}(q)}\times\mathrm{Spec}(R)$, where $\mathrm{Spec}(R)$ is an affine scheme endowed with its trivial $\widetilde{{\bf O}}(q)$-action, is a generating family of the topos $B_{\widetilde{{\bf O}}(q)}$.
\end{lem}
To be more precise, here we denote by $\mathrm{Spec}(R)$ an object of $B_{\widetilde{{\bf O}}(q)}$ given by a sheaf on $Y_{fl}$ represented by a $Y$-scheme of the form $\mathrm{Spec}(R)\rightarrow Y$ (the map itself is not necessarily affine) endowed with its trivial $\widetilde{{\bf O}}(q)$-action.
\begin{proof}
We need to show that, for any object $\mathcal{F}$ in $B_{\widetilde{{\bf O}}(q)}$, there exists an epimorphic family
$$\{E_{\widetilde{{\bf O}}(q)}\times\mathrm{Spec}(R_{i})\rightarrow\mathcal{F},i\in I\}$$
of morphisms in $B_{\widetilde{{\bf O}}(q)}$. Recall that we denote by $\eta: Y_{fl}\rightarrow B_{\widetilde{{\bf O}}(q)}$ the morphism of topoi induced by the morphism of groups $1\to \widetilde{{\bf O}}(q)$. The functor $\eta^*$ forgets the $\widetilde{{\bf O}}(q)$--action. It has a left adjoint $\eta_!$, which is defined as follows: $\eta_!(Z)$ is the object of $B_{\widetilde{{\bf O}}(q)}$ given by $E_{\widetilde{{\bf O}}(q)}\times Z$ on which $\widetilde{{\bf O}}(q)$ acts via the first factor.

Using the Yoneda Lemma, one sees that the class of $Y$-schemes $Y'\rightarrow Y$ forms a generating family of the topos $Y_{fl}$. Since any $Y$-scheme $Y'$ can be covered by open affine subschemes, the class of objects of the form $\mathrm{Spec}(R)\rightarrow Y$ is a generating family of $Y_{fl}$. Now let $\mathcal{F}$ be an object of $B_{\widetilde{{\bf O}}(q)}$, and let
$$\{g_{i}:\mathrm{Spec}(R_{i})\rightarrow\eta^*\mathcal{F},i\in I\}$$
be an epimorphic family in $Y_{fl}$. By adjunction we obtain a family
$$\{f_{i}:\eta_!(\mathrm{Spec}(R_{i}))\rightarrow\mathcal{F},i\in I\}$$
of morphisms in $B_{\widetilde{{\bf O}}(q)}$ such that
$$g_i=\eta^*(f_i)\circ \tau_i:\mathrm{Spec}(R_{i})\rightarrow \eta^*\eta_! (\mathrm{Spec}(R_{i}))\rightarrow\eta^*\mathcal{F}$$ 
where $\tau_{i}:\mathrm{Spec}(R_{i})\rightarrow \eta^*\eta_! (\mathrm{Spec}(R_{i}))$ is the adjunction map.

Let $u,v:\mathcal{F}\rightrightarrows\mathcal{G}$ be a pair of maps in $B_{\widetilde{{\bf O}}(q)}$ such that $u\circ f_{i}=v\circ f_{i}$ for each $i\in I$. In particular we have
$$\eta^*(u)\circ \eta^*(f_{i})\circ\tau_i=\eta^*(v)\circ \eta^*(f_{i})\circ\tau_i$$ for each $i\in I$, hence 
$\eta^*(u)\circ g_{i}=\eta^*(v)\circ g_{i}$ for each $i\in I$. It follows that $\eta^*(u)=\eta^*(v)$ since the family $\{g_i,i\in I\}$ is epimorphic. We obtain $u=v$ since $\eta^*$ is faithful. Hence the family $\{f_i,i\in I\}$ is epimorphic as well, and the result follows.

\end{proof}

\bigskip
{\bf Step 3: The cocycles $\alpha$, $\beta$ and $\gamma$.}
\bigskip

Recall that we denote by $$\widetilde{\pi}:B_{\widetilde{{\bf O}}(q)}\stackrel{B_{r_q}}{\longrightarrow} B_{{\bf O}(q)} \stackrel{\pi}{\longrightarrow} Y_{fl}$$ the canonical map. We set
$$\widetilde{\Theta}_q:=B_{r_q}^*{\Theta}_q\mbox{ , } \widetilde{T}_q:=B_{r_q}^*{T}_q \mbox{ , }  \widetilde{E}_{{\bf O}(q)}:=B_{r_q}^*E_{{\bf O}(q)}=\widetilde{{\bf O}}(q)/({\bf Z}/2{\bf Z}).$$

We now look for a covering of $B_{\widetilde{{\bf O}}(q)}$ which splits both the $\widetilde{\pi}^*{\bf O}(n)$--torsors $\widetilde{T}_{q}$ and $\widetilde{\Theta}_{q}$
and also the
$\widetilde{\pi}^*{\bf O}(q)$--torsor $\widetilde{E}_{{\bf O}(q)}$. The map
$$
\appl{\widetilde{E}_{{\bf O}(q)}\times\widetilde{\Theta}_{q}}{\widetilde{E}_{{\bf O}(q)}\times \widetilde{T}_{q}}{(x,t)}{(x,xt)}
$$
is an isomorphism of $B_{\widetilde{{\bf O}}(q)}$. It follows that
$\{U=\widetilde{E}_{{\bf O}(q)}\times\widetilde{\Theta}_{q}\rightarrow*\}$ is a covering of the final object in $B_{\widetilde{{\bf O}}(q)}$ trivializing
$\widetilde{T}_{q}$, $\widetilde{\Theta}_{q}$ and $\widetilde{E}_{{\bf O}(q)}$. We now can use the construction recalled in Step 1  to obtain $1$-cocycles of $U$ which represent each of these torsors. 
We then consider the map  $f: U\rightarrow \widetilde{T}_{q}$ defined by:
$$f: U=\widetilde{E}_{{\bf O}(q)}\times \widetilde{\Theta}_{q}\rightarrow \widetilde{E}_{{\bf O}(q)}\times \widetilde{T}_{q}\rightarrow \widetilde{T}_{q},\ (x, t)\rightarrow (x, xt)\rightarrow xt$$
in order to obtain the 1-cocycle $\gamma\in\mathcal{Z}^1(\{U\rightarrow*\},\widetilde{\pi}^*{\bf O}(n))$ which represents $\widetilde{T}_q$:
$$\fonc{\gamma}{U\times U}{\widetilde{\pi}^*{\bf O}(n)}{(x, t, y, u)}{(xt)^{-1}(yu)}
$$
We apply again this construction where the role of $f$ is now played successively by the projections  $U=\widetilde{E}_{{\bf O}(q)}\times\widetilde{\Theta}_{q}\rightarrow\widetilde{\Theta}_{q}$ and
$U=\widetilde{E}_{{\bf O}(q)}\times\widetilde{\Theta}_{q}\rightarrow \widetilde{E}_{{\bf O}(q)}$. We obtain  firstly
$$
\fonc{\beta}{U\times U}{\widetilde{\pi}^*{\bf O}(n)}{(x, t, y, u)}{t^{-1}u}
$$
for representative of  $\widetilde{\Theta}_{q}$ and secondly
$$
\fonc{\alpha}{U\times U}{\widetilde{\pi}^*{\bf O}(q)}{(x, t, y, u)}{x^{-1}y}
$$
for representative of $\widetilde{E}_{{\bf O}(q)}$. Of course we have $$\alpha\in\mathcal{Z}^1(\{U\rightarrow*\},\widetilde{\pi}^*{\bf O}(q))\mbox{ and }\beta,\gamma\in\mathcal{Z}^1(\{U\rightarrow*\},\widetilde{\pi}^*{\bf O}(n)).$$
Finally, applying the construction described in Step 2, we consider
the group isomorphism $$\theta_p:\widetilde{\pi}^*{\bf O}(q)(U\times U)\stackrel{\sim}{\longrightarrow} \widetilde{\pi}^*{\bf O}(n)(U\times U)$$ associated to
$$\fonc{p}{U\times U}{\widetilde{\Theta}_{q}. } {(x, t, y, u)}{t}$$
Considering $\alpha$ as an element of the group $\widetilde{\pi}^*{\bf O}(q)(U\times U)$ and $\beta,\gamma$ as elements of the group $\widetilde{\pi}^*{\bf O}(n)(U\times U)$, we have 
\begin{equation}\label{startingpoint}
\gamma=\theta_{p}(\alpha)\cdot \beta\in \widetilde{\pi}^*{\bf O}(n)(U\times U),
\end{equation}
since we  may write
$$\gamma(x, t, y, u)=(t^{-1}(x^{-1}y)t)(t^{-1}u)\in \widetilde{\pi}^*{\bf O}(n)(U\times U).$$
Notice that (\ref{startingpoint}) only makes sense in the group $\widetilde{\pi}^*{\bf O}(n)(U\times U)$, since $\mathcal{Z}^1(\{U\rightarrow*\},\widetilde{\pi}^*{\bf O}(n))$ only 
carries the structure of  a pointed set.\\

We continue to view  $\alpha$ and $\beta$ as elements of the groups $\widetilde{\pi}^*{\bf O}(q)(U\times U)$ and $\widetilde{\pi}^*{\bf O}(n)(U\times U)$ respectively and we consider the maps
$$r_q:\widetilde{\pi}^*\widetilde{\bf O}(q)\twoheadrightarrow\widetilde{\pi}^*{\bf O}(q)\mbox{ and }r_n:\widetilde{\pi}^*\widetilde{\bf O}(n)\twoheadrightarrow \widetilde{\pi}^*{\bf O}(n).$$

\begin{lem}
There exist an epimorphism $U'\rightarrow U$ together with  elements
$$\widetilde{\alpha_{|U'\times U'}}\in\widetilde{\pi}^*\widetilde{\bf O}(q)(U'\times U')\mbox{ and }\widetilde{\beta_{|U'\times U'}}\in\widetilde{\pi}^*\widetilde{\bf O}(n)(U'\times U')$$
such that
$$\alpha_{|U'\times U'}=  r_{q}(\widetilde{\alpha_{|U'\times U'}})\mbox{ and }\beta_{|U'\times U'}= r_{n}(\widetilde{\alpha_{|U'\times U'}}).$$
\end{lem}
\begin{proof}
First we show that $\beta$ has a lift. The map $\beta$ can be factored in the following manner:
$$\beta:{U\times U}=\widetilde{E}_{{\bf O}(q)}\times\widetilde{\Theta}_{q}\times \widetilde{E}_{{\bf O}(q)}\times\widetilde{\Theta}_{q}\longrightarrow \widetilde{\Theta}_{q}\times\widetilde{\Theta}_{q}\stackrel{b}{\longrightarrow} \widetilde{\pi}^*{\bf O}(n).$$
where
$$
\fonc{b}{\widetilde{\Theta}_{q}\times \widetilde{\Theta}_{q}}{\widetilde{\pi}^*{\bf O}(n)}{(t, u)}{t^{-1}u}
$$
By base change, it is enough to show that there exists an epimorphism  $V\rightarrow \widetilde{\Theta}_{q}$ and a commutative diagram in $B_{\widetilde{\bf O}(q)}$:
\[\xymatrix{
V\times V\ar[r]\ar[d]^{}& \widetilde{\pi}^*\widetilde{{\bf O}}(n)\ar[d]\\
\widetilde{\Theta}_{q}\times \widetilde{\Theta}_{q}\ar[r]^{}&\widetilde{\pi}^*{\bf O}(n)
}
\]
The objects $\widetilde{\Theta}_{q}\times \widetilde{\Theta}_{q}$, $\widetilde{\pi}^*{\bf O}(n)$ and $\widetilde{\pi}^*\widetilde{{\bf O}}(n)$ of  $B_{\widetilde{\bf O}(q)}$ occurring in this square are all given with the trivial action of $\widetilde{\bf O}(q)$, hence it is enough to show that there exist an epimorphism $V\rightarrow {\bf Isom}(t_{n}, q)$ in $Y_{fl}$ and a commutative diagram in $Y_{fl}$:
\[\xymatrix{
V\times V\ar[r]\ar[d]^{}&\widetilde{{\bf O}}(n)\ar[d]\\
{\bf Isom}(t_{n}, q)\times {\bf Isom}(t_{n}, q)\ar[r]^{}&{\bf O}(n)
}
\]
Take an \'etale covering $Y'\rightarrow Y$ trivializing $q$, i.e. such that there is an isometry
$f:q_{Y'}\stackrel{\sim}{\rightarrow} t_{n,Y'}$, i.e. such that there is a section
$$f:Y'\rightarrow Y'\times {\bf Isom}(q,t_{n})\mbox{ in }Y_{fl}/Y'.$$
Composition with $f$
$$Y'\times {\bf Isom}(t_{n}, q)\stackrel{(f,1)}{\longrightarrow} Y'\times {\bf Isom}(q,t_{n})\times {\bf Isom}(t_{n}, q)\stackrel{(1,-\circ-)}{\longrightarrow} Y'\times {\bf O}(n) $$
yields an isomorphism of ${\bf O}(n)$--torsors over $Y'$:
$$Y'\times {\bf Isom}(t_{n}, q)\simeq Y'\times {\bf O}(n).$$
Indeed, this map is clearly ${\bf O}(n)$--equivariant; it is an isomorphism whose inverse is induced by composition with $f^{-1}: t_{n,Y'}\stackrel{\sim}{\rightarrow}q_{Y'}$ in a similar way.
 We consider the maps
$$V=Y'\times\widetilde{{\bf O}}(n)\longrightarrow Y'\times{\bf O}(n)\simeq Y'\times{\bf Isom}(t_{n}, q)\longrightarrow {\bf Isom}(t_{n}, q) $$
and
$$
\appl{V\times V=Y'\times\widetilde{{\bf O}}(n)\times Y'\times \widetilde{{\bf O}}(n)}{\widetilde{{\bf O}}(n).}{(y',\sigma, z', \tau)}{\sigma^{-1}\tau}
$$
It is then straightforward to check that the above square  is commutative.

It remains to show that $\alpha$ has a lift. We consider the epimorphism in $B_{\widetilde{{\bf O}}(q)}$
$$U'= E_{\widetilde{{\bf O}}(q)}\times\widetilde{\Theta}_{q} \stackrel{(r,Id)}{\longrightarrow}\widetilde{E}_{{\bf O}(q)}\times\widetilde{\Theta}_{q}=U.$$
Here $r:E_{\widetilde{{\bf O}}(q)}\rightarrow \widetilde{E}_{{\bf O}(q)}$ is the map $\widetilde{{\bf O}}(q)\rightarrow{\bf O}(q)$ seen as an  $\widetilde{{\bf O}}(q)$--equivariant map, where $\widetilde{{\bf O}}(q)$ acts by left multiplication on both $\widetilde{{\bf O}}(q)$ and ${\bf O}(q)$.
Then
\[\xymatrix{
U'\times U'\ar[r]\ar[d]^{}& \widetilde{\pi}^*\widetilde{{\bf O}}(q)\ar[d]\\
U\times U\ar[r]^{}&\widetilde{\pi}^*{\bf O}(q)
}
\]
is a commutative diagram in $B_{\widetilde{\bf O}(q)}$ where the top horizontal map is defined as follows:
$$
\appl{U'\times U'=E_{\widetilde{{\bf O}}(q)}\times\widetilde{\Theta}_{q}\times E_{\widetilde{{\bf O}}(q)}\times\widetilde{\Theta}_{q}}{\widetilde{\pi}^*\widetilde{{\bf O}}(q)}{(\sigma,t,\tau, u)}{\sigma^{-1}\tau}
$$

We have shown that there exists epimorphisms  $U_{\alpha}\rightarrow U$ and $U_{\beta}\rightarrow U$
such that $\alpha_{|U_{\alpha}\times U_{\alpha}}$ and $\beta_{|U_{\beta}\times U_{\beta}}$ have lifts $\widetilde{\alpha_{|U_{\alpha}\times U_{\alpha}}}$ and $\widetilde{\beta_{|U_{\beta}\times U_{\beta}}}$. The conclusion of the Lemma with $U'=U_{\alpha}\times_UU_{\beta}\twoheadrightarrow U$ follows.

\end{proof}
It follows from (\ref{startingpoint}) that
\begin{equation}\label{anequ+}
\gamma_{|U'\times U'}=(\theta_{p}( \alpha)\cdot \widetilde \beta)_{|U'\times U'}=
\theta_{p_{|U'\times U'}}(\alpha_{|U'\times U'})\cdot  \beta_{|U'\times U'}.
\end{equation}
Using Lemma 5.11, we obtain
\begin{equation}\label{anequ}
\widetilde{\gamma_{|U'\times U'}}=\widetilde{\theta}_{p_{|U'\times U'}}(\widetilde{\alpha_{|U'\times U'}})\cdot \widetilde{\beta_{|U'\times U'}}\in\widetilde{\pi}^*\widetilde{\bf O}(q)(U'\times U')
\end{equation}
is a lift of $\gamma_{|U'\times U'}\in\widetilde{\pi}^*\widetilde{\bf O}(q)(U'\times U')$. From now on, we write $U$ for $U'$, $p$ for $p_{|U'\times U'}:U'\times U'\rightarrow W\rightarrow\widetilde{\Theta}_q$, $\alpha$ for $\alpha_{|U'\times U'}$,  $\beta$ for $\beta_{|U'\times U'}$ and  $\gamma$ for $\gamma_{|U'\times U'}$. We have lifts $\widetilde\alpha\in\widetilde{\pi}^*\widetilde{\bf O}(q)(U\times U)$, $\widetilde\beta\in\widetilde{\pi}^*\widetilde{\bf O}(n)(U\times U)$ and $\widetilde\gamma\in\widetilde{\pi}^*\widetilde{\bf O}(n)(U\times U)$ of $\alpha\in\widetilde{\pi}^*{\bf O}(q)(U\times U)$, $\beta\in\widetilde{\pi}^*{\bf O}(n)(U\times U)$ and $\gamma\in\widetilde{\pi}^*{\bf O}(n)(U\times U)$ respectively. Moreover, we have
\begin{equation}\label{anequ}
\widetilde{\gamma}=\widetilde{\theta}_{p}(\widetilde{\alpha})\cdot \widetilde{\beta}\in\widetilde{\pi}^*\widetilde{\bf O}(q)(U\times U).
\end{equation}

\bigskip
{\bf Step 4: Reduction to an identity of cocycles}
\bigskip

The extension of group objects in $Y_{fl}$
$$1\rightarrow{\bf Z}/2{\bf Z}\rightarrow  \widetilde{\bf O}(n)\rightarrow {\bf O}(n)\rightarrow 1$$
gives a morphism
$$\delta^2_n:H^1(B_{\widetilde{{\bf O}}(q)},\widetilde{\pi}^*{\bf O}(n))\rightarrow H^2(B_{\widetilde{{\bf O}}(q)},{\bf Z}/2{\bf Z}).$$
Notice that one has
\begin{equation}\label{e1}
\delta^2_n(\widetilde{T}_q)=\widetilde{T}_q^*[C_n]=B_{r_q}^*T_q^*[C_n]=B_{r_q}^*HW_2(q)
\end{equation}
and
\begin{equation}\label{e2}
\delta^2_n(\widetilde{\Theta}_q)=\widetilde{\Theta}_q^*[C_n]=B_{r_q}^*\Theta_q^*[C_n]=B_{r_q}^*w_2(q).
\end{equation}
Similarly the group extension
$$1\rightarrow{\bf Z}/2{\bf Z}\rightarrow  \widetilde{\bf O}(q)\rightarrow {\bf O}(q)\rightarrow 1$$
gives a morphism
$$\delta^2_q:H^1(B_{\widetilde{{\bf O}}(q)},\widetilde{\pi}^*{\bf O}(q))\rightarrow H^2(B_{\widetilde{{\bf O}}(q)},{\bf Z}/2{\bf Z})$$
such that
\begin{equation}\label{e3}
\delta^2_q(\widetilde{E}_{{\bf O}(q)})=\widetilde{E}_{{\bf O}(q)}^*[C_q]=B_{r_q}^*[C_q]=0.
\end{equation}

\begin{prop}
One is reduced to show
\begin{equation}\label{whatwe}
\delta^2_n(\widetilde{T}_q)=\delta^2_n(\widetilde{\Theta}_q)+ B_{r_q}^*w_{1}(q)\cup B_{r_q}^*\mbox{\emph{det}}[q] +\delta^2_q(\widetilde{E}_{{\bf O}(q)})
\end{equation}
in $H^2(B_{\widetilde{{\bf O}}(q)},{\bf Z}/2{\bf Z})$, which in turn will follow from an identity of cocycles
\begin{equation}\label{what}
\delta^2_n(\gamma)=\delta^2_n(\beta)+ (\mbox{\emph{det}}_{{\bf O}(n)}\circ \beta)\cup (\mbox{\emph{det}}_{{\bf O}(q)}\circ \alpha) +\delta^2_q(\alpha)
\end{equation}
in $\check{H}^2(\{U\rightarrow *\},{\bf Z}/2{\bf Z})$.
\end{prop}
\begin{proof}
By Proposition 5.5  it is enough to show that (\ref{what}) implies (\ref{whatwe}) and that (\ref{whatwe}) implies (\ref{whatwehave}). The fact that  (\ref{whatwe}) implies (\ref{whatwehave}) follows immediately from (\ref{e1}), (\ref{e2}), (\ref{e3}) and the fact that $B_{r_q}^*$ respects sums and cup-products.

Let us show that (\ref{what}) implies (\ref{whatwe}). The 2--cocycles $\delta^2_n(\gamma)$, $\delta^2_n(\beta)$ and $\delta^2_q(\alpha)$, all elements of $\mathcal{Z}^2(\{U\rightarrow *\},{\bf Z}/2{\bf Z})$, represent the cohomology classes $\delta^2_n(\widetilde{T}_q)$, $\delta^2_n(\widetilde{\Theta}_q)$ and $\delta^2_q(\widetilde{E}_{{\bf O}(q)})$ respectively. Then we observe that
$\mbox{det}_{{\bf O}(n)}\circ \beta$  and $\mbox{det}_{{\bf O}(q)}\circ \alpha$ are 1-cocycles representing the maps $$B_{\widetilde{{\bf O}}(q)}\stackrel{B_{r_q}}{\longrightarrow} B_{{\bf O}(q)}\stackrel{\Theta_q}{\longrightarrow} B_{{\bf O}(n)} \stackrel{\mbox{det}_{{\bf O}(n)}}{\longrightarrow}B_{{\bf Z}/2{\bf Z}}$$
and
$$B_{\widetilde{{\bf O}}(q)}\stackrel{B_{r_q}}{\longrightarrow} B_{{\bf O}(q)}\stackrel{E_{{\bf O}(q)}=Id}{\longrightarrow} B_{{\bf O}(q)} \stackrel{\mbox{det}_{{\bf O}(q)}}{\longrightarrow}B_{{\bf Z}/2{\bf Z}}$$
respectively. By definition, these two maps correspond to the cohomology classes $B_{r_q}^*w_1(q)$ and $B^*_{r_q}\mbox{det}[q]$ respectively. The result follows since the map $\check{H}^2(\{U\rightarrow *\},{\bf Z}/2{\bf Z})\rightarrow H^2(B_{\widetilde{{\bf O}}(q)},{\bf Z}/2{\bf Z})$ is compatible with cup-products.\\
\end{proof}

{\bf Step 5: Proof of (\ref{what})}\\

We still denote by $$p_{ij}:U\times U \times U\rightarrow U \times U$$
the projection on the $(i,j)$-components.
Then we have 
$$\delta_{n}^{2}(\gamma)=(\widetilde\gamma p_{23})(\widetilde\gamma p_{13})^{-1}(\widetilde\gamma p_{12})\in\mathcal{Z}^{2}(\{U\rightarrow *\},{ \bf Z}/2{ \bf Z})\subset \widetilde{\pi}^*\widetilde{\bf O}(n)(U\times U\times U),$$
a 2-cocycle representative of $\delta_{n}^{2}(\widetilde{T}_{q})\in H^2(B_{\widetilde{{\bf O}}(q)},{\bf Z}/2{\bf Z})$. Of course $\delta_{n}^{2}(\gamma)$ is only well defined in $\check{H}^{2}(\{U\rightarrow*\},{\bf Z}/2{\bf Z})$, i.e. a different choice for the lift $\widetilde\gamma$ gives a cohomologous 2-cocycle.
By (\ref{anequ}) we have
$$\delta_{n}^{2}(\gamma)=((\widetilde \theta_{p}(\widetilde \alpha)p_{23})(\widetilde\beta p_{23}))
((\widetilde \theta_{p}(\widetilde \alpha)p_{13})(\widetilde\beta p_{13}))^{-1}
((\widetilde \theta_{p}(\widetilde \alpha)p_{12})(\widetilde\beta p_{12})).$$
Our first goal is to understand the terms $\widetilde \theta_{p}(\widetilde \alpha)p_{ij}$. To this end we introduce the natural projections
$\mathfrak{p}_{i}$ for $1\leq i\leq 3$:
$$\mathfrak{p}_{i}:U\times U\times U\stackrel{\textrm{pr}_i}{\longrightarrow} U \longrightarrow \widetilde E_{{\bf O}(q)}\times \widetilde \Theta_q\stackrel{\textrm{pr}_2}{\longrightarrow}\widetilde\Theta_{q} .$$
Recall from Step 2 that
$$\widetilde \theta_{p}(\widetilde \alpha)p_{23}=\textrm{pr}\circ \widetilde\theta\circ(p\times \widetilde \alpha)\circ p_{23}.$$
where
$$p:U\times U\stackrel{\textrm{pr}_1}{\longrightarrow} U \longrightarrow \widetilde E_{{\bf O}(q)}\times\widetilde \Theta_q\stackrel{\textrm{pr}_2}{\longrightarrow}\widetilde\Theta_{q}$$
We now observe that
$$(p\times \widetilde \alpha)\circ p_{23}=(p\circ p_{23})\times (\widetilde \alpha \circ p_{23})=
\mathfrak{p}_{2}\times \widetilde \alpha \circ p_{23}. $$
Hence we deduce that
$$\widetilde \theta_{p}(\widetilde \alpha)p_{23}=\widetilde \theta_{\mathfrak{p}_{2}}(\widetilde \alpha p_{23})\in  \widetilde{\pi}^*\widetilde{\bf O}(n)(U\times U\times U).$$
Similarly, we have
$$\widetilde \theta_{p}(\widetilde \alpha)p_{13}=\widetilde \theta_{\mathfrak{p}_{1}}(\widetilde \alpha p_{13})\ \mbox{and}\
\widetilde \theta_{p}(\widetilde \alpha)p_{12}=\widetilde \theta_{\mathfrak{p}_1}(\widetilde \alpha p_{12})$$
in the group $\widetilde{\pi}^*\widetilde{\bf O}(n)(U\times U\times U)$. This yields
\begin{equation}\label{equ1}
\delta_{n}^{2}(\gamma)=\widetilde \theta_{\mathfrak{p}_{2}}(\widetilde \alpha p_{23})(\widetilde\beta p_{23})
(\widetilde\beta p_{13})^{-1}\widetilde \theta_{\mathfrak{p}_{1}}(\widetilde \alpha p_{13}^{-1})
\widetilde \theta_{\mathfrak{p}_{1}}(\widetilde \alpha p_{12})(\widetilde\beta p_{12}).
\end{equation}
Moreover, we have
$$\widetilde \theta_{\mathfrak{p}_{1}}(\widetilde \alpha p_{13}^{-1})
\widetilde \theta_{\mathfrak{p}_{1}}(\widetilde \alpha p_{12})=\widetilde \theta_{\mathfrak{p}_{1}}(\widetilde \alpha p_{13}^{-1}\widetilde \alpha p_{12})
=\widetilde \theta_{\mathfrak{p}_{1}}(\widetilde \alpha p_{23}^{-1})\widetilde \theta_{\mathfrak{p}_{1}}(\widetilde \alpha p_{23}\widetilde \alpha p_{13}^{-1}\widetilde \alpha p_{12}).$$
Since $\widetilde \alpha p_{23}\widetilde \alpha p_{13}^{-1}\widetilde \alpha p_{12}$ is in the kernel of $r_{q, Z}$ and since $\widetilde \theta_{\mathfrak{p}_{1}}$ coincides with the identity on this kernel we can write
\begin{equation}\label{equ2}
\widetilde \theta_{\mathfrak{p}_{1}}(\widetilde \alpha p_{23}^{-1})\widetilde \theta_{\mathfrak{p}_{1}}(\widetilde \alpha p_{23}\widetilde \alpha p_{13}^{-1}
\widetilde \alpha p_{12})=\widetilde \theta_{\mathfrak{p}_{1}}(\widetilde \alpha p_{23}^{-1})(\widetilde \alpha p_{23}\widetilde \alpha p_{13}^{-1}\widetilde \alpha p_{12}).
\end{equation}
Since $(\widetilde \beta p_{23})(\widetilde \beta  p_{13})^{-1}(\widetilde \beta  p_{12})$ is in the kernel of $r_{n, Z}$,
it belongs to the center of $\widetilde{\pi}^*\widetilde{\bf O}(n)(U\times U\times U)$, and it follows from (\ref{equ1}) and (\ref{equ2}) that we have
$$\delta^{2}_{n}(\gamma)=\delta^{2}_{n}(\beta) \cdot \xi \cdot \delta^{2}_{q}(\alpha)\in\widetilde{\pi}^*\widetilde{\bf O}(n)(U\times U\times U)$$
where $\xi$ is defined as follows:
$$\xi= \widetilde \theta_{\mathfrak{p}_{2}}(\widetilde \alpha p_{23})(\widetilde \beta p_{12})^{-1}\widetilde \theta_{\mathfrak{p}_{1}}
(\widetilde \alpha p_{23}^{-1})(\widetilde \beta p_{12}) \in\mathcal{Z}^{2}(\{U\rightarrow *\},{ \bf Z}/2{ \bf Z})$$
Clearly the result (\ref{what}) would follow from an identity
$$\xi=\mbox{det}_{{\bf O}(n)}(\beta)\cup \mbox{det}_{{\bf O}(q)}(\alpha)\in \mathcal{Z}^2(\{U\rightarrow *\},{\bf Z}/2{\bf Z})$$
in the group of 2--cocycles $\mathcal{Z}^2(\{U\rightarrow *\},{\bf Z}/2{\bf Z})$. Since $\mathcal{Z}^2(\{U\rightarrow *\},{\bf Z}/2{\bf Z})\subset {\bf Z}/2{\bf Z}(U\times U\times U)$, it is of course equivalent to show that
\begin{equation}\label{thelast}
\xi=\mbox{det}_{{\bf O}(n)}(\beta)\cup \mbox{det}_{{\bf O}(q)}(\alpha)\in {\bf Z}/2{\bf Z}(U\times U\times U).
\end{equation}
Let us first make the cup product $\mbox{det}_{{\bf O}(n)}(\beta)\cup \mbox{det}_{{\bf O}(q)}(\alpha)$ more explicit: It is given by
$$m\circ(\mbox{det}_{{\bf O}(n)}(\beta) p_{12},\mbox{det}_{{\bf O}(q)}(\alpha) p_{23}):U\times U\times U \longrightarrow
{\bf Z}/2{\bf Z}\times {\bf Z}/2{\bf Z}\longrightarrow{\bf Z}/2{\bf Z}$$
where
$m:{\bf Z}/2{\bf Z}\times {\bf Z}/2{\bf Z}\rightarrow{\bf Z}/2{\bf Z}$ is the standard multiplication,
$$
\fonc{\mbox{det}_{{\bf O}(n)}(\beta):=\mbox{det}_{{\bf O}(n)}\circ\beta}{U\times U}{\mu_2={\bf Z}/2{\bf Z}}{(x_{1}, t_{1}, x_{2}, t_{2})}{\mbox{det}_{{\bf O}(n)}(t_{1}^{-1}t_{2})}$$
and
$$
\fonc{\mbox{det}_{{\bf O}(q)}(\alpha):=\mbox{det}_{{\bf O}(q)}\circ\alpha}{U\times U}{\mu_2={\bf Z}/2{\bf Z}}{(x_{1}, t_{1}, x_{2}, t_{2})}{\mbox{det}_{{\bf O}(q)}(x^{-1}_{1}x_{2})}.$$
By Lemma \ref{lem-generatingfamily}, the class of objects of the form $E_{\widetilde{{\bf O}}(q)}\times\mbox{Spec}(R)\rightarrow Y$, where $\mbox{Spec}(R)$ is an affine $Y$-scheme endowed with its trivial $\widetilde{{\bf O}}(q)$--action, is a generating family of the topos $B_{\widetilde{{\bf O}}(q)}$. Therefore, in order to prove (\ref{thelast})
it is enough to show
\begin{equation}\label{xicircu}
\xi\circ u=(\mbox{det}_{{\bf O}(n)}(\beta) \cup \mbox{det}_{{\bf O}(q)}(\alpha))\circ u\in {\bf Z}/2{\bf Z}(E_{\widetilde{{\bf O}}(q)}\times\mbox{Spec}(R))
\end{equation}
for any map $u$ in $B_{\widetilde{{\bf O}}(q)}$ of the form $$u: E_{\widetilde{{\bf O}}(q)}\times\mbox{Spec}(R)\rightarrow U\times U\times U,$$ where $\mbox{Spec}(R)$ is an affine scheme. Moreover, by adjunction we have an isomorphism (see the proof of Lemma \ref{lem-generatingfamily})
$$
\appl{\textrm{Hom}_{B_{\widetilde{{\bf O}}(q)}}(E_{\widetilde{{\bf O}}(q)}\times\mbox{Spec}(R), {\bf Z}/2{\bf Z})}{\textrm{Hom}_{Y_{fl}}(\mbox{Spec}(R), {\bf Z}/2{\bf Z})}{f}{f_{\mid R}}
$$
sending $f:E_{\widetilde{{\bf O}}(q)}\times\mbox{Spec}(R)\rightarrow {\bf Z}/2{\bf Z}$ to
$$f_{\mid R}:\mbox{Spec}(R)\rightarrow \eta^*(E_{\widetilde{{\bf O}}(q)}\times\mbox{Spec}(R))\stackrel{\eta^*f}{\rightarrow}{\bf Z}/2{\bf Z}.$$
Using the bijection $f\mapsto f_{\mid R}$ above, we are   reduced to showing the identity
\begin{equation}\label{almost}
(\xi\circ u)_{\mid R}= m\circ(\mbox{det}_{{\bf O}(n)}(\beta p_{12} u)_{\mid R},\mbox{det}_{{\bf O}(q)}(\alpha p_{23} u)_{\mid R})
\end{equation}
in ${\bf Z}/2{\bf Z}(\textrm{Spec}(R))={\bf Z}/2{\bf Z}^{\pi_0(\textrm{Spec}(R))}$. Note by the way that one may suppose $\textrm{Spec}(R)$ connected and reduced. Indeed, (\ref{almost}) can be shown after restriction to $\textrm{Spec}(R_j)^{\textrm{red}}\rightarrow \textrm{Spec}(R)$ for any connected component $\textrm{Spec}(R_j)^{\textrm{red}}$ of $\textrm{Spec}(R)$, given with its unique structure of reduced closed affine subscheme (note that a connected component is always closed but not necessarily open). By Lemma \ref{lem-weirdo}, we have
\begin{eqnarray}
(\xi\circ u)_{\mid R}&=&(\widetilde \theta_{\mathfrak{p}_{2}u}(\widetilde \alpha p_{23}u)\cdot(\widetilde \beta p_{12}u)^{-1}\cdot\widetilde \theta_{\mathfrak{p}_{1}u}
(\widetilde\alpha p_{23}u)^{-1}\cdot(\widetilde \beta p_{12}u))_{\mid R}\\
&=&(\widetilde \theta_{\mathfrak{p}_{2}u}(\widetilde \alpha p_{23}u))_{\mid R}\cdot (\widetilde \beta p_{12}u)_{\mid R}^{-1}\cdot(\widetilde \theta_{\mathfrak{p}_{1}u}
(\widetilde\alpha p_{23}u))_{\mid R}^{-1}\cdot(\widetilde \beta p_{12}u)_{\mid R}\\
\label{fau}&=& \eta^*\widetilde \theta_{\mathfrak{p}_{2}u_{\mid R}}(\widetilde \alpha p_{23}u_{\mid R})\cdot \widetilde \beta p_{12}u_{\mid R}^{-1}\cdot \eta^*\widetilde \theta_{\mathfrak{p}_{1}u_{\mid R}}
(\widetilde\alpha p_{23}u_{\mid R})^{-1}\cdot \widetilde \beta p_{12}u_{\mid R}
\end{eqnarray}
where $\widetilde\alpha p_{ij}u_{\mid R}\in\widetilde{\bf O}(q)(R)$, $\widetilde\beta p_{ij}u_{\mid R}\in\widetilde{\bf O}(n)(R)$ and $\eta^*\widetilde \theta_{\mathfrak{p}_{i}u_{\mid R}}:\widetilde{\bf O}(q)(R)\rightarrow \widetilde{\bf O}(n)(R)$ is the map  induced  by $\mathfrak{p}_{i}u_{\mid R}\in  {\bf Isom}(t_{n}, q)(R)$, see Step 2. Moreover, one has
$$\mathfrak{p}_{2}=\mathfrak{p}_{1} \star \beta p_{12}$$
where $\star:\widetilde{\Theta}_q\times \widetilde{\pi}^*{\bf O}(n)\rightarrow \widetilde{\Theta}_q$ is the $\widetilde{\pi}^*{\bf O}(n)$--torsor structure map of $\widetilde{\Theta}_q$. Applying successively $(-)\circ u$ and $(-)_{\mid R}$ we obtain
$$\mathfrak{p}_{2}u_{\mid R}=(\mathfrak{p}_{1}u_{\mid R})\circ(\beta p_{12}u_{\mid R})$$
where the right-hand side $(\mathfrak{p}_{1}u_{\mid R})\circ(\beta p_{12}u_{\mid R})\in {\bf Isom}(t_{n}, q)(R)$ is the composition of $\mathfrak{p}_{1}u_{\mid R}\in {\bf Isom}(t_{n}, q)(R)$ and  $\beta p_{12}u_{\mid R}\in{\bf O}(n)(R)$. It then follows from (\ref{weirdo0}) and Lemma \ref{lem-Ph} that
\begin{equation}\label{derder}
\eta^*\widetilde \theta_{\mathfrak{p}_{2}u_{\mid R}}:= \widetilde \psi_{\mathfrak{p}_{2}u_{\mid R}^{-1}}
=\widetilde \psi_{(\beta p_{12}u_{\mid R})^{-1}(\mathfrak{p}_{1}u_{\mid R})^{-1}}=\widetilde \psi_{(\beta p_{12}u_{\mid R})^{-1}} \circ\widetilde \psi_{(\mathfrak{p}_{1}u_{\mid R})^{-1}}=:
\widetilde \psi_{\beta p_{12}u_{\mid R}^{-1}}\circ \eta^*\widetilde\theta_{\mathfrak{p}_{1}u_{\mid R}}.
\end{equation}
By Lemma \ref{lem-Ph}  iii) and iv), if
$\beta p_{12}u_{\mid R} \in {\bf O}_{+}(n)(R)$, then $\widetilde \psi_{(\beta p_{12}u_{\mid R})^{-1}}=i_{(\widetilde \beta p_{12}u_{\mid R})^{-1}}$, since $\widetilde \beta p_{12}u_{\mid R}$ is a lift of $\beta p_{12}u_{\mid R}$.
For $\beta p_{12}u_{\mid R}\ \in {\bf O}_{+}(n)(R)$ we obtain
$$\eta^*\widetilde \theta_{\mathfrak{p}_{2}u_{\mid R}}(-)=(\widetilde \beta p_{12}u_{\mid R})^{-1}\cdot\eta^*\widetilde\theta_{\mathfrak{p}_{1}u_{\mid R}}(-)\cdot(\widetilde \beta p_{12}u_{\mid R})$$ hence, by (\ref{fau}),
$(\xi\circ u) _{\mid R}= 0$. We now suppose that $\beta p_{12}u_{\mid R} \in {\bf O}_{-}(n)(R)$ and  $\alpha p_{23}u_{\mid R}
\in {\bf O}_{+}(q)(R)$. It follows that $\widetilde \alpha p_{23}u_{\mid R}
\in \widetilde {\bf O}_{+}(q)(R)$ and that
$\eta^*\widetilde \theta_{\mathfrak{p}_{1}u_{\mid R}}(\widetilde \alpha p_{23}u_{\mid R})\ \in \widetilde{{\bf O}}_{+}(n)(R)$.
Since $\widetilde \psi_{(\beta p_{12}u_{\mid R})^{-1}}$ coincides with $i_{(\widetilde \beta p_{12}u_{\mid R})^{-1}}$ on $\widetilde {\bf O}_{+}(n)(R)$ we deduce from (\ref{fau}) and (\ref{derder}) that $(\xi \circ u)_{\mid R}=0$. We now assume that  $\alpha p_{23}u_{\mid R}\in {\bf O}_{-}(q)(R)$ and $\widetilde \alpha p_{23}u_{\mid R}\in \widetilde {\bf O}_{-}(q)(R)$. Using
$\widetilde \psi_{(\beta p_{12}u_{\mid R})^{-1}}=-i_{(\widetilde \beta p_{12}u_{\mid R})^{-1}}$ on $\widetilde {\bf O}_{-}(n)(R)$, we conclude that
$(\xi\circ u)_{\mid R}=1$ in this last case. A  comparison in each case of the values of  $(\xi \circ u)_{\mid R}$ and
$$((\mbox{det}_{{\bf O}(n)}(\beta)\cup\mbox{det}_{{\bf O}(q)}(\alpha))\circ u)_{\mid R}=\mbox{det}_{{\bf O}(n)}(\beta p_{12}u_{\mid R})\cdot \mbox{det}_{{\bf O}(q)}(\alpha p_{23}u_{\mid R})$$
yields
$$(\xi \circ u)_{\mid R}=((\mbox{det}_{{\bf O}(n)}(\beta)\cup\mbox{det}_{{\bf O}(q)}(\alpha))\circ u)_{\mid R}\in {\bf Z}/2{\bf Z}(\mbox{Spec}(R))={\bf Z}/2{\bf Z}$$
for any $\textrm{Spec}(R)$ connected. The result follows.

\end{proof}

\begin{rem}

Let us define
$$H(B_{{\bf O}(q)}, {\bf Z}/2{\bf Z})^{*}=\{1+a_{1}+a_{2}\'e\in  \bigoplus_{0\leq i\leq 2}H^{i}(B_{{\bf O}(q)}, {\bf Z}/2{\bf Z});
a_{i}\in H^{i}(B_{{\bf O}(q)}, {\bf Z}/2{\bf Z})\}\ . $$
The operation
$$(1+a_{1}+a_{2})(1+b_{1}+b_{2})=1+(a_{1}+b_{1})+a_{2}+b_{2}+a_{1}\cup b_{1} .$$
turn $H(B_{{\bf O}(q)}, {\bf Z}/2{\bf Z})^{*}$ into an abelian group. Moreover, $T_q$ and $\Theta_q$ induce morphisms of abelian group from $H(B_{{\bf O}(n)}, {\bf Z}/2{\bf Z})^{*}$ to $H(B_{{\bf O}(q)}, {\bf Z}/2{\bf Z})^{*}$. We associate to $(V,q)$ the element
$$s_{q}=1+\mathrm{det}[q]+[C_{q}]\ \in H(B_{{\bf O}(q)}, {\bf Z}/2{\bf Z})^{*}$$
and we simply write $s_{n}$ for $s_{t_{n}}$. Then Theorem \ref{mainthm} yields the identity
$$s_q=T_{q}^{*}(s_{n}){\Theta_{q}}^{*}(s_{n})^{-1}.$$
\end{rem}

\section{Consequences of the main theorem}

\subsection{Serre's formula}
Our aim is to deduce comparison formulas from Theorem 5.1  which extend the work of Serre (see [15], Chapitre III, Annexe, (2.2.1), (2.2.2))  to symmetric bundles over an arbitrary base scheme. This formula is also refered to as the "real Fr\"ohlich-Kahn-Snaith formula" in [17]. A direct proof of this result is given in \cite{CNET1}, Theorem 0.2.

We consider an ${\bf O}(q)$--torsor $\alpha$ of $Y_{fl}$. We also denote by $\alpha:Y_{fl}\rightarrow B_{{\bf O}(q)}$ the classifying map for this torsor, and by $[\alpha]$ its class in  $H^{1}(Y_{fl}, {\bf O}(q))$. We denote by
$$\delta_{q}^{1}:H^{1}(Y_{fl}, {\bf O}(q))\rightarrow H^{1}(Y_{fl}, {\bf Z}/2{\bf Z})$$
the map induced by the determinant map $\mathrm{det}_q:{\bf O}(q)\rightarrow  {\bf Z}/2{\bf Z}$,  and by
$$\delta_{q}^{2}: H^{1}(Y_{fl}, {\bf O}(q))\rightarrow H^{2}(Y_{fl}, {\bf Z}/2{\bf Z})$$
the boundary map associated to the group--extension $C_q$ in $Y_{fl}$
\begin{equation}\label{ext-Cq}
1\rightarrow {\bf Z}/2{\bf Z}\rightarrow \widetilde
{\bf O}(q)\rightarrow {\bf O}(q)\rightarrow 1.
\end{equation}
In other words, we have
\begin{equation}\label{hehehe}
\delta_q^1[\alpha]:=\alpha^*(\mbox{det}[q]) \mbox{ and }  \delta_q^2[\alpha]:=\alpha^*[C_q].
\end{equation}
As in Section \ref{sect-twisted-forms} we associate to $\alpha$  a symmetric bundle $(V_{\alpha},q_{\alpha})$ on $Y$.
\begin{cor}\label{cor-one}
For any ${\bf O}(q)$--torsor $\alpha$ of $Y_{fl}$, we have
\begin{itemize}
\item[i)] $w_{1}(q_{\alpha})=w_{1}(q)+\delta_q^1[\alpha]$ \mbox{ in } $H^1(Y,{\bf Z}/2{\bf Z})$;
\item[ii)] $ w_{2}(q_{\alpha})=w_{2}(q)+w_{1}(q)\cdot \delta_q^1[\alpha] +\delta_q^2[\alpha] $ \mbox{ in } $H^2(Y,{\bf Z}/2{\bf Z})$.
\end{itemize}
\end{cor}
\begin{proof}  We define   $\{q_{\alpha}\}: Y_{fl}\rightarrow B_{{\bf O}(n)}$  to be  the morphism of topoi associated  to the ${\bf O}(n)$--torsor ${\bf Isom }(t_{n}, q_{\alpha})$ of $Y_{fl}$ and we let
$T_q: B_{{\bf O}(q)}\rightarrow B_{{\bf O}(n)}$ be the morphism  of topoi defined in \ref{defTq}.
\begin{lem}The following triangle is commutative 
\[\xymatrix{
Y_{fl}\ar[r]^{\alpha}\ar[rd]_{\{q_{\alpha}\}}&B_{{\bf O}(q)}\ar[d]^{T_q}\\
&B_{{\bf O}(n)}
}\]
\end{lem}
\begin{proof}
It will suffice  to describe  an isomorphism $$\{q_{\alpha}\}^*E_{{\bf O}(n)}\simeq \alpha^*T_q^*E_{{\bf O}(n)}$$ of ${\bf O}(n)$--torsors of $Y_{fl}$. It follows from the definitions
that  $\{q_{\alpha}\}^*E_{{\bf O}(n)}={\bf Isom}(t_{n}, q_{\alpha})$ and that
$$\alpha^*T_q^*E_{{\bf O}(n)}=\alpha^*{\bf Isom}(t_{n}, q)={\bf Isom}(q, q_{\alpha})\wedge^{{\bf O}(q)}{\bf Isom}(t_{n}, q).$$
The lemma then follows from the fact that the map
$${\bf Isom}(q, q_{\alpha})\times {\bf Isom}(t_{n}, q)\rightarrow {\bf Isom}(t_{n}, q_{\alpha}),$$
given by composition, induces an ${\bf O}(n)$--equivariant isomorphism
$${\bf Isom}(q, q_{\alpha})\wedge^{{\bf O}(q)} {\bf Isom}(t_{n}, q)\simeq {\bf Isom}(t_{n}, q_{\alpha}).$$
\end{proof}
As a consequence of the lemma we obtain that
\begin{equation}\label{here}
\alpha^{*}T_{q}^{*}(HW_{i})=\{q_{\alpha}\}^{*}(HW_{i})=w_{i}(q_{\alpha})\mbox{ in } H^i(Y,{\bf Z}/2{\bf Z}), \mbox{ for } i=1,2.
\end{equation} We now
observe that,  since $\pi\circ \alpha\simeq id$,  we have
\begin{equation}\label{hehe}
\alpha^{*}(w_{i}(q))=\alpha^{*}\pi^{*}(w_{i}(q))=w_{i}(q).
\end{equation}
Using (\ref{hehehe}), (\ref{here}) and (\ref{hehe}), the corollary is just the pull-back of Theorem \ref{mainthm} via $\alpha^*$.

\end{proof}

\subsection{Comparison formulas for Hasse-Witt invariants of orthogonal representations.}
Let $(V,q,\rho)$ be an orthogonal representation of $G$. To be more precise, $G$ is a group-scheme over $Y$, $(V, q)$ is  a symmetric bundle over $Y$,  and $\rho:G\rightarrow {\bf O}(q)$ is a morphism of $Y$-group-schemes.

We denote by $B_{\rho}: B_{G}\rightarrow B_{{\bf O}(q)}$ the morphism of classifying topoi induced by the group homomorphism $\rho$.
The Hasse-Witt invariants $w_i(q, \rho)$ of $(V,q,\rho)$ lie in $H^i(B_G,{\bf Z}/2{\bf Z})$. Indeed, there is a morphism
$$B_G\stackrel{B_{\rho}}{\longrightarrow}B_{{\bf O}(q)}\stackrel{T_q}{\longrightarrow}B_{{\bf O}(n)}$$
canonically associated to $(V,q,\rho)$ and $w_i(q, \rho)$ is simply the pull-back of $HW_i$ along this map:
\begin{equation}\label{almostnew}
w_i(q, \rho):=(T_q\circ B_{\rho})^*(HW_i)=B_{\rho}^*(HW_i(q)).
\end{equation}
On the other hand the morphism of groups
$$\mbox{det}_{q}\circ \rho:G\longrightarrow {\bf O}(q)\longrightarrow {\bf Z}/2{\bf Z}$$
defines (see Proposition \ref{prop-Giraud-sequence}) a cohomology class $w_1(\rho)\in H^1(B_G,{\bf Z}/2{\bf Z})$. Note that one has
\begin{equation}\label{ehe}
w_{1}(\rho)=B_{\rho}^{*}(\mbox{det}[q]).
\end{equation}
Pulling back the group extension $C_q$ along the map $\rho:G\rightarrow {\bf O}(q)$, we obtain a group--extension
$$1\rightarrow {\bf Z}/2{\bf Z}\rightarrow \widetilde{G}\rightarrow G\rightarrow 1$$
where $\widetilde{G}:=\widetilde{\bf O}(q)\times_{{\bf O}(q)}G$. We denote by $C_{G}\in\mathrm{Ext}_Y(G, {\bf Z}/2{\bf Z})$ the class of this extension and by $[C_G]$ its cohomology class in $ H^2(B_G,{\bf Z}/2{\bf Z})$  (see Proposition \ref{prop-Giraud-sequence}), so that
\begin{equation}\label{ehe2}
B_{\rho}^{*}([C_{q}])=[C_{G}].
\end{equation}

\begin{cor}\label{cor-two} Let $G$ be a group scheme on $Y$ and let $(V, q, \rho)$ be an orthogonal representation of $G$. Then in $H^{*}(B_{G}, {\bf Z}/2{\bf Z})$
we have
\begin{itemize}
\item[i)] $w_{1}(q, \rho)=w_{1}(q)+w_{1}(\rho)$
\item[ii)] $ w_{2}(q, \rho)=w_{2}(q)+w_{1}(q)\cdot w_{1}(\rho) +[C_{G}]$.
\end{itemize}
\end{cor}

\begin{proof} Let us  denote by  $\mu$ and $\pi$  the morphisms of classifying topoi associated to the group--morphisms $G_Y\rightarrow1$ and ${\bf O}(q)\rightarrow 1$ respectively. We have $\mu\simeq \pi\circ  B_{\rho}$ and therefore
\begin{equation}\label{ehe3}
\mu^{*}(w_{i}(q))=B_{\rho}^{*}\pi^{*}(w_{i}(q)),\  i =1,2.
\end{equation}
As in subsection 4.5,  we identify $H^{i}(Y_{fl}, {\bf Z}/2{\bf Z})$ as a direct summand of $H^{i}(B_{G}, {\bf Z}/2{\bf Z})$ (resp. $H^{i}(B_{{\bf O}(q)}, {\bf Z}/2{\bf Z})$) via $\mu^*$
 (resp. $\pi^*$). In view of (\ref{almostnew}), (\ref{ehe}), (\ref{ehe2}) and (\ref{ehe3}), the corollary is just the pull-back of Theorem \ref{mainthm} via $B_{\rho}^{*}$.
\end{proof}

\subsection{Fr\"ohlich twists}
In this section we extend the work of Fr\"ohlich \cite{Fro} and the results of \cite{CNET1}, Theorem 0.4, to twists of quadratic forms by  $G$-torsors when the group scheme $G$ is  not necessarily constant. Let $G$ be a group-scheme over $Y$ and let $(V, q,\rho)$ be an orthogonal representation of  $G$.
\begin{definition}
For any $G$-torsor $X$ on $Y$, we define the twist $(V_X,q_X)$ to be  the symmetric bundle on $Y$ associated to the morphism
$$\{q_X\}:Y_{fl}\stackrel{X}{\longrightarrow} B_G \stackrel{B_{\rho}}{\longrightarrow} B_{{\bf O}(q)} \stackrel{T_q}{\longrightarrow} B_{{\bf O}(n)}.$$ Equivalently, $(V_X,q_X)$ is the twist of $(V,q)$  given by the morphism
$$Y_{fl}\stackrel{X}{\longrightarrow} B_G \stackrel{B_{\rho}}{\longrightarrow} B_{{\bf O}(q)}.$$
\end{definition}
The twist $(V_X,q_X)$ can be made explicit in a number of situations (see for example Section \ref{sect-explicit-twists} below). In order to compare $w(q)$ with $w(q_{X})$, we denote by $$\delta_{q, \rho}^{1}:H^{1}(Y_{fl}, G)\longrightarrow H^{1}(Y_{fl}, {\bf Z}/2{\bf Z})$$
the map induced by the group homomorphism $\mbox{det}_{q}\circ \rho$ and by
$$\delta_{q, \rho}^{2}: H^{1}(Y_{fl}, G)\rightarrow H^{2}(Y_{fl}, {\bf Z}/2{\bf Z})$$
the boundary map associated to the group-extension $C_{G}$ in $Y_{fl}$
$$1\rightarrow {\bf Z}/2{\bf Z}\rightarrow \widetilde
G\rightarrow G\rightarrow 1.$$

\begin{cor}\label{cor-three} Let $(V, q, \rho)$ be a $G$--equivariant symmetric bundle and let $X$ be a $G$-torsor over $Y$. Then we have
\begin{itemize}
\item[i)] $w_{1}(q_X)=w_{1}(q)+\delta_{q, \rho}^1[X]; $
\item[ii)] $ w_{2}(q_X)=w_{2}(q)+w_{1}(q)\cdot \delta_{q, \rho}^1[X]+\delta_{q, \rho}^2[X]. $
\end{itemize}
\end{cor}
\begin{proof}
By definition $\delta_{q, \rho}^1[X]$ is the cohomology class associated to the morphism
$$B_{\mathrm{det}_{q}}\circ B_{\rho}\circ X: Y_{fl}\longrightarrow B_G\stackrel{B_{\rho}}{\longrightarrow} B_{{\bf O}(q)}\stackrel{B_{\mathrm{det}_{q}}}{\longrightarrow} B_{{\bf Z}/2{\bf Z}}.$$
It follows that
\begin{equation}\label{a}
\delta_{q, \rho}^1[X]=(B_{\rho}\circ X)^*(\mbox{det}[q]).
\end{equation}
Moreover one has
\begin{equation}\label{b}
\delta_{q, \rho}^2[X]:=X^*([C_G])=X^*B_{\rho}^*([C_q])=(B_{\rho}\circ X)^*([C_q])
\end{equation}
and
\begin{equation}\label{c}
w_i(q_X):=(T_q\circ B_{\rho}\circ X)^*(HW_i)=(B_{\rho}\circ X)^*(HW_i(q)).
\end{equation}
and finally, we  have
\begin{equation}\label{d}
w_i(q)=(B_{\rho}\circ X)^*(w_i(q))
\end{equation}
since $B_{\rho}\circ X$ is defined over $Y_{fl}$. In view of (\ref{a}), (\ref{b}), (\ref{c}) and (\ref{d}) the corollary now follows by pulling back the equality in Theorem \ref{mainthm} along the morphism $B_{\rho}\circ X$.

\end{proof}
\begin{rem}
We have associated to any orthogonal representation $\rho: G\rightarrow {\bf O}(q)$ an exact
sequence of $Y$--group--schemes:
$$1\rightarrow {\bf Z}/2{\bf Z}\rightarrow \widetilde G\rightarrow G\rightarrow 1 .$$
For any $G$--torsor $X$ over $Y$ the class  $\delta^{2}_{q, \rho}[X]$ may be seen as the obstruction of the embedding problem associated to the scheme $X$ and the exact sequence. Corollary 6.4  provides us with a formula for this obstruction in terms of Hasse-Witt invariants of symmetric bundles:
$$\delta^{2}_{q, \rho}[X]= w_{2}(q_{X})+w_{2}(q)+w_{1}(q)^2+w_1(q)w_1(q_X) . $$
In the particular case where $\rho: G_{Y}\rightarrow {\bf SO}(q)$ we have this remarkabley simple formula
$$\delta^{2}_{q, \rho}[X]= w_{2}(q_{X})-w_{2}(q), $$
which expresses this obstruction as a difference of two Hasse-Witt invariants of quadratic forms.
\end{rem}

\subsection{Equivariant symmetric bundles}

\subsubsection{Equivariant symmetric bundles}
Let $G$ be a discrete group acting on a scheme $X$. An equivariant symmetric bundle on $(G,X)$ is a symmetric bundle $(E,r)$ on $X$ endowed with a family of isomorphisms of symmetric bundles $\{\phi_g:(E,r)\rightarrow g^*(E,r),\,g\in G\}$ such that $\phi_{hg}=g^*(\phi_h)\circ\phi_g$ and $\phi_1=\textrm{Id}_{(V,q)}$.

One way to generalize this notion for arbitrary group--schemes is the following:
let $G$ be a $Y$--group--scheme acting (on the left) on the $Y$--scheme $X$. An equivariant symmetric bundle $(E,r,\sigma)$ on $(G,X)$ is a symmetric bundle $(E,r)$ on $X$ such that ${\bf Isom}_X(t_n,r)$, seen as object of $X_{fl}\stackrel{\sim}{\rightarrow}Y_{fl}/X$, is endowed with a left action $\sigma$ of $G$ such that ${\bf Isom}_X(t_n,r)\rightarrow X$ is $G$--equivariant. A $G$--equivariant symmetric bundle $(E,r,\sigma)$ on $(G,X)$ may be seen as an ${\bf O}(n)$--torsor of $B_G/X$, or equivalently, as a map $B_G/X\rightarrow B_{{\bf O}(n)_Y}$. Recall that there is a canonical map $l:X_{fl}\rightarrow B_G/X$, whose inverse image forgets the $G$-action.

\begin{definition}
The category of $G$--equivariant symmetric bundles on $(G,X)$ of rank $n$  is the opposite category to ${\bf Homtop}_{Y_{fl}}(B_G/X,B_{{\bf O}(n)_Y})$. In particular, a $G$--equivariant symmetric bundle $(E,r,\sigma)$ on $(G,X)$ is a map $$\{r,\sigma\}:B_G/X\longrightarrow B_{{\bf O}(n)_Y}.$$
The underlying symmetric bundle $(E,r)$ is the symmetric bundle on $X$ determined (see Proposition \ref{prop-define-Quad-as-Homtop}) by the map
$$\{r\}:X_{fl}\stackrel{l}{\longrightarrow} B_G/X\stackrel{\{r,\sigma\}}{\longrightarrow} B_{{\bf O}(n)_Y}.$$
We define its Hasse-Witt invariants $w_i(r,\sigma)$ as follows:
$$\fonc{\{r,\sigma\}^*}{H^i(B_{{\bf O}(n)_Y},{\bf Z}/2{\bf Z})}{H^i(B_G/X,{\bf Z}/2{\bf Z})}{HW_i}{w_i(r,\sigma)}.$$
\end{definition}

 We consider   the following example:  assume that the quotient scheme $Y=X/G$ exists, and let $(V,q)$ be a symmetric bundle on $Y$ endowed with an orthogonal representation $\rho:G\rightarrow {\bf O}(q)$. Then the pull-back $(V_X,q_X)$ of $(V,q)$ along $X\rightarrow Y$ has a natural $G$--equivariant structure. Indeed, ${\bf Isom}_X(t_n,q)$ is, as an object of  $Y_{fl}/X$, given by ${\bf Isom}_Y(t_n,q)\times_Y X\rightarrow X$. We let $G$ act diagonally on ${\bf Isom}_Y(t_n,q)\times_Y X$ so that the second projection $${\bf Isom}_Y(t_n,q)\times_Y X\longrightarrow X$$ is $G$--equivariant. In terms of morphisms, this $G$--equivariant symmetric bundle is the map
\begin{equation}\label{equi-s-b}
B_G/X\stackrel{l_X}{\longrightarrow} B_G \stackrel{B_{\rho}}{\longrightarrow} B_{{\bf O}(q)_Y} \stackrel{T_q}{\longrightarrow} B_{{\bf O}(n)_Y}
\end{equation}
where $l_X$ is the localization map.

\begin{definition}
Let $G$ be a group-scheme acting on $X$ with quotient scheme $Y$ and let $(V,q,\rho)$ be an orthogonal representation on $Y$. We denote by $(V_X,q_X)$ the $G$--equivariant symmetric bundle on $X$ defined by (\ref{equi-s-b}), i.e. as the pull-back of $(V,q)$ along $X\rightarrow Y$ endowed with its natural $G$--equivariant structure.
\end{definition}
In case that $X$ is a $G$-torsor, the notation $(V_X,q_X)$ for the $G$--equivariant symmetric bundle on $X$ is compatible with the notation for the Fr\"ohlich twisted form $(V_X,q_X)$ on $Y$, since they can be identified via the equivalence $B_G/X\simeq Y$. Indeed, we have a commutative diagram:
\[
\xymatrix{
Y_{fl}\ar[d]^{\simeq}\ar[r]^{X}& B_G\ar[d]^=\ar[r]^{B_{\rho}}&B_{{\bf O}(q)}\ar[d]^=\ar[r]^{T_q}&B_{{\bf O}(n)}\ar[d]^=\\
B_G/X\ar[r]^{l_X}& B_G\ar[r]^{B_{\rho}}&B_{{\bf O}(q)}\ar[r]^{T_q}&B_{{\bf O}(n)}
}\]
where the top row  $(T_q\circ B_{\rho}\circ X)$ defines Fr\"ohlich twisted form $(V_X,q_X)$ on $Y$ and the lower row  $(l_X\circ B_{\rho}\circ X)$ defines by the $G$--equivariant symmetric bundle $(V_X,q_X)$ on $X$. In other words, Fr\"ohlich twisted form on $Y$ may be viewed as the $G$--equivariant symmetric bundle $(V_X,q_X)$ on $X$.

\subsubsection{Comparison formulas}
Let $G$ be a group--scheme acting on the scheme $X$ with property  that the quotient scheme $Y$ exists and such that $X\rightarrow Y$ has weak odd ramification in the sense of Definition \ref{def-weak-odd} below. Let $(V,q)$ be a symmetric bundle on $Y$ together with an orthogonal representation $G\rightarrow {\bf O}(q)$, and let $(V_X,q_X)$ be the induced $G$--equivariant symmetric bundle on $(G,X)$. By definition, we have $H^*(B_G/X,{\bf Z}/2{\bf Z})\simeq H^*(Y_{fl},{\bf Z}/2{\bf Z})$ hence $(V_X,q_X)$ has Hasse-Witt invariants
$$w_i(V_X,q_X)\in H^i(B_G/X,{\bf Z}/2{\bf Z})\simeq H^i(Y_{fl},{\bf Z}/2{\bf Z}).$$
We consider the map
$$f^*:H^*(B_{{\bf O}(q)},{\bf Z}/2{\bf Z})\rightarrow H^*(B_G/X,{\bf Z}/2{\bf Z})\simeq H^*(Y_{fl},{\bf Z}/2{\bf Z}), $$
where $f$ is the morphism $B_\rho\circ l_X$. 

\begin{cor}\label{cor-four}
We have
\begin{itemize}
\item[i)] $w_{1}(q_X)=w_{1}(q)+f^*(\mathrm{det}[q]); $
\item[ii)] $ w_{2}(q_X)=w_{2}(q)+w_{1}(q)\cdot f^*(\mathrm{det}[q])+ f^*([C_q])$.
\end{itemize}
in the ring $H^*(Y_{fl},{\bf Z}/2{\bf Z})$.
\end{cor}
\begin{proof}
Apply $f^*$ to Theorem \ref{mainthm}.
\end{proof}

\subsection{An explicit description of the twisted form}\label{sect-explicit-twists}
We shall use our previous work to obtain an explicit description of the twists which, in this geometric context,  generalizes Fr\"ohlich's construction \cite{Fro}. Recall that $(V_X,q_X)$ is the symmetric bundle on $Y$ associated to the morphism
$$Y_{fl}\stackrel{X}{\longrightarrow} B_G \stackrel{B_{\rho}}{\longrightarrow} B_{{\bf O}(q)}. $$
In other words, $(V_X,q_X)$ is determined by the fact that we have an isomorphism of ${\bf O}(q)$--torsors
$${\bf Isom}(q, q_{X})\simeq X\wedge ^{G}{\bf O}(q). $$
Our goal is to provide  a  concrete description  of $(V_{X}, q_{X})$  at  least when $G$ satisfies  some  additional hypotheses.
\noindent We start by recalling the results of \cite{CCMT1} in the affine case.

\begin{definition}
Let $R$ be a commutative noetherian  integral domain with fraction field $K$.  A finite and flat $R$-algebra $A$ is said to satisfy  ${\bf H_{2}}$ when $A_{K}$ is a commutative separable $K$-algebra and the image under the counit
of the set of integral of $A$ is the square of a principal ideal of $R$.
\end{definition}

Let $S=\mbox {Spec}(R)$.  We assume that $G\rightarrow S$ is a group scheme associated to a Hopf algebra $A$ which
 satisfies ${\bf H_{2}}$ (we will say that {\it $G$ satisfies ${\bf H_{2}}$}). We consider a $G$--equivariant symmetric bundle
 $(V, q, \rho)$ given by  a projective  $R$-module $V$ endowed with a non-degenerate quadratic form $q$ and a group homomorphism
 $\rho: G\rightarrow {\bf O}(q)$.
 We have proved in \cite{CCMT1} Theorem 3.1  that  for any $G$--torsor $X=\mbox{Spec}(B) \rightarrow S$
 where  $B$ is a commutative and finite $R$-algebra,   the twist of $(V, q, \rho)$ by $X$ is defined by
 $$(V_{X}, q_{X})=(\mathcal{D}^{-1/2}(B)\otimes_{R}V, Tr\otimes q)^{A}. $$
 This twist  can be roughly described as the symmetric bundle obtained by taking  the fixed point by $A$ of the tensor product of
 $(V, q)$ by the square root of the different of $B$,  endowed with the trace form,  (see \cite{CCMT1}, Sections 1 and 2 for the precise definitions).

We now come back to the general situation. We assume that $G\rightarrow S$ satisfies ${\bf H_{2}}$. Moreover, 
 for the sake of simplicity,  we suppose that $Y$ is  integral and flat over $S$. Let $(V, q, \rho)$ be a $G_{Y}$--equivariant symmetric bundle and let $X$ be
  a $G_{Y}$--torsor. For any affine open subscheme  $U\rightarrow Y$ we set $X_{U}=X\times _{Y}U$ and
  $G_{U}=G_{Y}\times_{Y}U$. We know that by base change $X_{U}\rightarrow U$ is a $G_{U}$--torsor. Moreover  by restriction $(V, q)$
  defines  an equivariant $G_{U}$-symmetric bundle $(V\mid {U}, q\mid U,  \rho \mid U)$ over $U$.  Using the functoriality
  properties of the different maps involved it is easy to check that
  $$(V_{X}, q_{X})\mid U\simeq ((V\mid {U})_{X_{U}}, (q\mid U)_{X_{U}}) . $$
  Since the properties of $G$ are preserved by flat base change we conclude that $G_{U}\rightarrow U$ satisfies ${\bf H_{2}}$ and
  therefore that  by the above $((V\mid {U})_{X_{U}}, (q\mid U)_{X_{U}})$ has now been  explicitly described.

\section{Appendix}

\subsection{Odd ramification}

Let $G$ be a finite constant group acting admissibly on the $S$-scheme $X$, so that $X/G=Y$ exists as a scheme. Here a geometric point of the scheme $X$  is always algebraic and separable (i.e. given by a separable closure of the residue field $k(x)$ for some $x\in X$). Let $\alpha:\overline{\alpha}\rightarrow Y$ be a geometric point and consider the set $\Lambda_{\alpha}$ of geometric points $\overline{\alpha}\rightarrow X$ lying over $\alpha$.
Then $G$ acts transitively on the left (by composition) on the finite set $\Lambda_{\alpha}$. For an element of $\Lambda_{\alpha}$, $\beta:\overline{\alpha}\rightarrow X$, the inertia subgroup $I_{\beta}$ of  $\beta$ is defined as the stabilizer of the geometric point $\beta$.

\begin{definition}\label{def-oddramif-finitecase}
We say that the cover $X\rightarrow Y=X/G$ has \emph{odd ramification} if the order of $I_{\beta}$ is odd for any geometric point $\beta$ of $X$. 
\end{definition}

It follows from the following lemma that the cover $X\rightarrow Y=X/G$ has odd ramification if and only if $H^n(I_{\beta},{\bf Z}/2{\bf Z})=0$ for any $n\geq1$ and any  geometric point $\beta$ of $X$.

\begin{lem}\label{lem-coh-finite-grps}
Let $I$ be a finite group. Then $I$ has odd order if and only if $H^*(I,{\bf Z}/2{\bf Z})={\bf Z}/2{\bf Z}$.
\end{lem}

\begin{proof}
This follows from \cite{Quillen71} Corollary 7.8. 
\end{proof}
Let $G$ again denote a finite (constant) group acting admissibly on the $S$-scheme $X$ with quotient scheme $Y$. The quotient map
$X\rightarrow Y=X/G$ is finite, and hence proper. We have a commutative diagram:
\[\xymatrix{
B_G/X\ar[r]^g\ar[d]_j&Y_{fl}\ar[d]_{i} \\
B^{et}_G/X\ar[r]^{g^{Et}}\ar[d]_p&Y_{Et}\ar[d]_q \\
\mathcal{S}_{et}(G,X)\ar[r]^{g^{et}}&Y_{et}
}\]

We denote by $$Rg_*:\mathcal{D}(Ab(B_G/X))\longrightarrow \mathcal{D}(Ab(Y_{fl}))$$ the total derived functor of the left exact functor $g_*:Ab(B_G/X)\rightarrow Ab(Y_{fl})$, where $Ab(B_G/X)$ (respectively $Ab(Y_{fl})$) is the abelian category of abelian objects in $B_G/X$ (respectively in $Y_{fl}$), and $\mathcal{D}(Ab(B_G/X))$ (respectively $\mathcal{D}(Ab(Y_{fl}))$) is its  derived category. Similarly, we denote by $Rg_*^{Et}$ and $Rg_*^{et}$ the total derived functors of $g_*^{Et}$ and $g_*^{et}$ respectively.

\begin{lem}\label{lemhe}
The following statements are equivalent.
\begin{enumerate}
\item The cover $X\rightarrow X/G$ has odd ramification in the sense of Definition \ref{def-oddramif-finitecase}.
\item The natural map is a quasi-isomorphism ${\bf Z}/2{\bf Z}\simeq Rg^{et}_*{\bf Z}/2{\bf Z}$.
\item The natural map is a quasi-isomorphism ${\bf Z}/2{\bf Z}\simeq Rg^{Et}_*{\bf Z}/2{\bf Z}$.
\item The natural map is a quasi-isomorphism ${\bf Z}/2{\bf Z}\simeq Rg_*{\bf Z}/2{\bf Z}$.
\end{enumerate}
\end{lem}

\begin{proof}

$(1)\Leftrightarrow(2)$: Let $\alpha:\overline{\alpha}\rightarrow Y$ be an algebraic separable geometric point of $Y$ and let $\beta:\overline{\alpha}\rightarrow X$ be a geometric point over $\alpha$. Then for any $n\geq0$ one has
$$(R^ng^{et}_*{\bf Z}/2{\bf Z})_{\alpha}\simeq H^n(I_{\beta},{\bf Z}/2{\bf Z});$$
this follows from of proper base change and from the fact that $G$ is finite. It then follows from Lemma \ref{lem-coh-finite-grps} that $R^ng^{et}_*{\bf Z}/2{\bf Z}=0$ for any $n\geq1$ if and only if $X\rightarrow X/G$ has odd ramification.

$(2)\Leftrightarrow(3)$: Recall that $q_*$ and $p_*$ are both exact and that $p^*$ and $q^*$ are both fully faithful. We have
$$Rg_*^{et}Rp_*{\bf Z}/2{\bf Z}\simeq Rg_*^{et}p_*{\bf Z}/2{\bf Z}\simeq Rg_*^{et}{\bf Z}/2{\bf Z}$$
$$\simeq Rq_*Rg_*^{Et}{\bf Z}/2{\bf Z}\simeq q_*Rg_*^{Et}{\bf Z}/2{\bf Z}.$$
In particular, if $Rg_*^{Et}{\bf Z}/2{\bf Z}\simeq {\bf Z}/2{\bf Z}$,  then $Rg_*^{et}{\bf Z}/2{\bf Z}\simeq q_*{\bf Z}/2{\bf Z}\simeq{\bf Z}/2{\bf Z}$. Hence we have $(3)\Rightarrow (2)$. Conversely, assume that  $Rg_*^{et}{\bf Z}/2{\bf Z}\simeq{\bf Z}/2{\bf Z}$. 
Then we have $$q_*Rg_{Y,*}^{Et}{\bf Z}/2{\bf Z}=\textrm{Res}_Y(Rg_{Y,*}^{Et}{\bf Z}/2{\bf Z})\simeq{\bf Z}/2{\bf Z},$$ where the functor $\textrm{Res}_Y$ maps a big \'etale sheaf on $Y$ to its restriction to the small \'etale site of $Y$ and $g_Y^{Et}$ is just $g^{Et}$. Moreover,  the cover $X\rightarrow X/G=Y$ has weak odd ramification. For any $Y$-scheme $Y'$ let $G$ act on $X':=X\times_YY'$ via its action on $X$. Then $G$ acts admissibly on  $X'$ and the cover $X'\rightarrow X'/G$ has weak odd ramification. If we denote  the canonical morphism (as defined above)
$g^{et}_{Y'}:  S_{et}(G, X')\rightarrow Y'_{et}$, then $$\textrm{Res}_{Y'}(Rg_{Y,*}^{Et}{\bf Z}/2{\bf Z})\simeq\textrm{Res}_{Y'}(Rg_{Y',*}^{Et}{\bf Z}/2{\bf Z})\simeq Rg_{Y',*}^{et}{\bf Z}/2{\bf Z}\simeq{\bf Z}/2{\bf Z}.$$
But the system of functors $\textrm{Res}_{Y'}:Y_{Et}\rightarrow Y'_{et}$, for $Y'$ running over the category of $Y$-schemes, is conservative. Hence
the canonical map
$${\bf Z}/2{\bf Z}\rightarrow Rg_{Y,*}^{Et}{\bf Z}/2{\bf Z}$$
is an isomorphism.

$(3)\Leftrightarrow(4)$ is shown below in a more general setting.

\end{proof}

The following definition generalizes the notion of odd ramification to the case where the group $G$ is an arbitrary group-scheme (rather than a finite constant group).

\begin{prop-definition}\label{def-weak-odd}
Let $G$ be an $S$-group scheme acting on an $S$-scheme $X$ such that the quotient $X/G=Y$ exists as a scheme. 
The cover $X\rightarrow Y$ is said to have \emph{odd ramification} if the following equivalent conditions are satisfied:
\begin{enumerate}
\item The natural map is a quasi-isomorphism ${\bf Z}/2{\bf Z}\simeq Rg^{Et}_*{\bf Z}/2{\bf Z}$.
\item The natural map is a quasi-isomorphism ${\bf Z}/2{\bf Z}\simeq Rg_*{\bf Z}/2{\bf Z}$.
\end{enumerate}
\end{prop-definition}
\begin{proof}
For any $n\geq0$ the sheaf $R^n(g_*^{Et}){\bf Z}/2{\bf Z}$ (resp. the sheaf $R^n(g_*){\bf Z}/2{\bf Z}$) is the sheaf associated to the presheaf
$$\appl{{\bf Sch}/Y}{{\bf Ab}}{Y'}{H^n(B^{Et}_G/X\times_YY',{\bf Z}/2{\bf Z})\simeq H^n(B_G/X\times_YY',{\bf Z}/2{\bf Z})} $$
for the \'etale topology (resp. for the flat topology), where the isomorphism is given by Lemma \ref{compra-etale-flat}. First we observe that $R^0g_*^{Et}{\bf Z}/2{\bf Z}={\bf Z}/2{\bf Z}$ if and only if $R^0g_*{\bf Z}/2{\bf Z}={\bf Z}/2{\bf Z}$, which follows from the fact that the above presheaf  is already a sheaf for the flat ( and hence the \'etale) topology. Following \cite{SGA4}, we denote by $\widehat{C}$ (respectively $\widetilde{C}$) the category of presheaves (respectively of sheaves) on a site $C$. The inverse image functor of the morphism $\widetilde{C}\rightarrow \widehat{C}$ coincides with the associated functor. Considering the sequence of topoi $Y_{fl}\rightarrow Y_{Et}\rightarrow \widehat{\bf Sch}/Y$ we see that sheafifying a presheaf $P$ on ${\bf Sch}/Y$ for the \'etale  topology and then sheafifying for the flat topology is the same thing as sheafifying $P$ for the flat topology directly. In particular, if the \'etale sheafification of the presheaf above is zero,  then so is its flat sheafification. Therefore, if  $R^ng_*^{Et}{\bf Z}/2{\bf Z}=0$,  then $R^ng_*{\bf Z}/2{\bf Z}=0$; hence $(1)\Rightarrow (2)$.

Conversely, assume that $Rg_*{\bf Z}/2{\bf Z}\simeq {\bf Z}/2{\bf Z}$. Then one has
$$\mathbb{Z}/2\simeq Ri_*\mathbb{Z}/2\simeq Ri_*Rg_*\mathbb{Z}/2\simeq Rg_*^{Et}Rj_*\mathbb{Z}/2\simeq  Rg_*^{Et}\mathbb{Z}/2$$
where $i$ and $j$ are the morphisms of the diagram prior to  Lemma \ref{lemhe}, and the fact that $\mathbb{Z}/2\simeq Rj_*\mathbb{Z}/2$ follows from Lemma \ref{compra-etale-flat}. Hence we have shown $(2)\Rightarrow (1)$.

\end{proof}
\begin{example}\label{ex--torsor}
Let $X$ be an $S$-scheme which supports  a left action by  the $S$-group scheme $G$, with the property  that $X$ is a left $G_Y$--torsor over $Y$. Then $X\rightarrow X/G=Y$ has odd ramification.
\end{example}
\begin{proof}
We have an equivalence $B_G/X\simeq (B_G/Y)/X\simeq (B_{G_Y})/X$
and the map $$g:B_G/X\simeq (B_{G_Y})/X\rightarrow Y_{fl}$$ is an equivalence, since $yX$ is a (left) $yG_Y$--torsor in $Y_{fl}$. We therefore obtain $$Rg_*{\bf Z}/2{\bf Z}=g_*{\bf Z}/2{\bf Z}={\bf Z}/2{\bf Z}.$$
\end{proof}
\begin{example}\label{ex-locodd}
Let $G$ be an $S$-group scheme which is locally a finite constant group scheme in  the fppf--topology. Let $G$ acts on a scheme $X$ with quotient scheme $Y$. Assume that there is a covering family $\{U_i\rightarrow S,\,i\in I\}$ with the property  that $G_{U_i}$ is a finite constant group, $X_{U_i}/G_{U_i}=Y_{U_i}$ and the cover $X_{U_i}\rightarrow Y_{U_i}$ has odd ramification in the sense of Definition \ref{def-oddramif-finitecase} for all $i\in I$. Then  $X\rightarrow X/G$ has odd ramification in the sense of Definition \ref{def-weak-odd}.
\end{example}
\begin{proof}Localizing over $Y$ if  necessary, one can assume that $Y=S$. Consider $g:B_G/X\rightarrow Y_{fl}$ and let $U=U_i$ be a $Y$-scheme.
The following diagram consists of pull--backs:
\[\xymatrix{
B_{G_{U}}/X_{U}\ar[r]\ar[d]&B_{G_{U}}\ar[r]\ar[d] &U_{fl}\ar[d]\\
B_G/X\ar[r]^{}&B_{G}\ar[r]&Y_{fl}
}\]
where $G_U=G\times _YU$ and $X_U=X\times _YU$. Since the total pull--back square above is obtained by localization we have (by Lemma \ref{lemhe}) $$Rg_*{\bf Z}/2{\bf Z}\mid U_i\simeq Rg_{U_i,*}{\bf Z}/2{\bf Z}\simeq{\bf Z}/2{\bf Z}$$
for any $i\in I$, where $$g_{U_i}:B_{G_{U_i}}/X_{U_i}\rightarrow (X_{U_i}/G_{U_i})_{fl}=U_{i,fl}$$
is the map induced by $g:B_G/X\rightarrow Y_{fl}$ by localization. The natural map ${\bf Z}/2{\bf Z}\rightarrow Rg_*{\bf Z}/2{\bf Z}$ is therefore a quasi-isomorphism, since $\{U_i\rightarrow S=Y,\,i\in I\}$ is a covering family.
\end{proof}
\begin{prop}
Let $X$ be an $S$-scheme endowed  with a left action of the $S$-group scheme $G$. Let $1=G_0\subseteq G_1\subseteq ...\subseteq G_n=G$ be a filtration by normal subgroup-schemes such that:
\begin{itemize}
\item  For all $i\geq j$ the quotients $G_{i}/G_j$ and $X_i:=X/G_i$ exist as schemes;
\item For all $i$ we have $X_i/(G_{i+1}/G_i)=X_{i+1}$.
\item  For all $i\geq j$ the map $G_{i}\rightarrow G_{i}/G_j$ is faithfully flat locally of finite presentation.
\end{itemize}
If the cover $X_i\rightarrow X_i/(G_{i+1}/G_i)$ has odd ramification  for each $i$, then the cover $X\rightarrow X/G$ has odd ramification.
 \end{prop}
\begin{proof}
We consider the decomposition
$$f:B_G/X\stackrel{f_1}{\longrightarrow} B_{G/G_1}/X_1\stackrel{f_2}{\longrightarrow} B_{G/G_2}/X_2\stackrel{f_3}{\longrightarrow} ...\stackrel{f_{n}}{\longrightarrow} B_{G/G_n}/X_n\simeq X_{n,fl}.$$
Since we have
$$Rf_*{\bf Z}/2{\bf Z}\simeq Rf_{n,*}\cdot\cdot\cdot Rf_{2,*}Rf_{1,*}{\bf Z}/2{\bf Z}$$
it is enough to show that $Rf_{i,*}{\bf Z}/2{\bf Z}\simeq{\bf Z}/2{\bf Z}$ for all $i$. For all $i\geq j$ the map $G_{i}\rightarrow G_{i}/G_j$ is a covering for the fppf-topology by assumption. It follows that the Yoneda embedding preserves the colimit of the effective equivalence relation $G_j\times G_i\rightrightarrows G_i$,
where the double map is defined by multiplication and projection. In other words, one has $yG_{i}/yG_{j}=y(G_i/G_{j})$. We obtain
\begin{eqnarray*}
f_i^*(y(G/G_{i})\times X_i)&=&yG/yG_{i}\times X_{i-1}\\
&=&(yG/yG_{i-1})/(yG_i/yG_{i-1})\times X_{i-1}\\
&=&y(G/G_{i-1})/y(G_i/G_{i-1})\times X_{i-1}.
\end{eqnarray*}
Hence
\begin{equation}\label{onelocalization}
(B_{G/G_{i-1}}/X_{i-1})/f_i^*(y(G/G_{i})\times X_i)\simeq B_{G_i/G_{i-1}}/X_{i-1}, 
\end{equation}
and of course we  have
\begin{equation}\label{anotherlocalization}
(B_{G/G_i}/X_i)/(y(G/G_{i})\times X_i)\simeq X_{i,fl}.
\end{equation}
Localizing over $y(G/G_{i})\times X_i$  we obtain the pull--back
\[\xymatrix{
B_{G_i/G_{i-1}}/X_{i-1}\ar[r]^{f'_i}\ar[d]&X_{i,fl}\ar[d]\\
B_{G/G_{i-1}}/X_{i-1}\ar[r]^{f_i}&B_{G/G_i}/X_i
}\]
where the vertical maps are localization maps (see  (\ref{onelocalization}) and (\ref{anotherlocalization})). Moreover,  by assumption $Rf'_{i,*}{\bf Z}/2{\bf Z}\simeq{\bf Z}/2{\bf Z}$
 and  so we obtain $Rf_{i,*}{\bf Z}/2{\bf Z}\simeq{\bf Z}/2{\bf Z}$ using the standard base change isomorphism for the localization pull--back square above.
\end{proof}

\subsection{A review of basic topos theory}
In this section we recall some elementary notions concerning Grothendieck topoi.
\subsubsection{Giraud's axioms}\label{sect-Gir-Ax}
A Grothendieck topos is a category that is equivalent to the category of sheaves of sets on a Grothendieck site. Many non-equivalent sites give rise to the same topos, however  Giraud's Theorem [SGA 4][IV, Theorem 1.2] provides an intrinsic definition: a topos is a category satisfying Giraud's axioms. To be more precise, a category $\mathcal{S}$ is a topos if and only if it satisfies the following four properties.
\begin{itemize}
\item $\mathcal{S}$ has finite projective limits.
\item $\mathcal{S}$ has direct sums, which are moreover disjoint and universal.
\item Equivalence relations in $\mathcal{S}$ are effective and universal.
\item $\mathcal{S}$ has a small set of generators.
\end{itemize}
Roughly speaking, these axioms tell us that a topos behaves as the category of sets. For instance, arbitrary projective limits (such as a final object $e_{\mathcal{S}}$, products, fiber products, kernel of pairs of maps) and arbitrary inductive limits (such as an initial object $\emptyset_{\mathcal{S}}$, sums, amalgamated sums, quotient by equivalence relations) exist and are universal (i.e. stable by base change) in the topos $\mathcal{S}$. Therefore, the reader should think of a topos as the category of sets and perform standard operations (such as the ones mentioned above) in a topos as if it were the category of sets. The basic example the reader should keep in mind is the topos ${\bf Sets}$ of sets. Note that ${\bf Sets}$ is the category of sheaves of sets on the point space $\{*\}$; it is called the \emph{punctual topos}.

\subsubsection{Sites and the Yoneda functor}
For a category $C$ endowed with a Grothendieck topology, we denote by $\widetilde{C}$ the topos of sheaves of sets on $C$.
We have the \emph{Yoneda functor}:
$$\fonc{y_C}{C}{\widetilde{C}}{X}{y_C(X)}$$
where $y_C(X)$ is the sheaf associated to the presheaf $\textrm{Hom}_C(-,X)$. The Yoneda functor always commutes with finite projective limits, i.e. it is \emph{left exact}, since the associated sheaf functor is left exact. The site $C$ is said to be \emph{subcanonical} if the presheaves $\textrm{Hom}_C(-,X)$ are all sheaves already. The Yoneda functor $y_C$ is fully faithful if and only if the site $C$ is subcanonical.

A site $C$ is said to be \emph{left exact} if $C$ is given by a category having finite projective limits (or equivalently, having a final object and fiber products) endowed with a subcanonical topology. If the site $C$ is left exact, then $C$ can be seen as a left exact full subcategory of $\widetilde{C}$, since the Yoneda functor $y_C$ is fully faithful and left exact.

\subsubsection{Algebraic structure} A group--object in a topos $\mathcal{S}$ (more generally, in any category with finite projective limits) is an object $G$ endowed with maps $e_\mathcal{S}\rightarrow G$ (unit), $G\times G\rightarrow G$ (multiplication) and $i:G\stackrel{\sim}{\rightarrow} G$ (inverse) satisfying the standard axioms. Let $C$ be a site whose underlying category has finite projective limits and let $G$ be a group--object in $C$. The group--object $G$ is defined in terms of finite projective limits (a finite product and the final object are involved) and the Yoneda functor commutes with these limits, hence $y_C(G)$ is a group in $\widetilde{C}$. Equivalently, $y_C(G)(X)$ is a group -- in the usual sense -- for any object $X$ in $C$ and the restriction map $y_C(G)(X)\rightarrow y_C(G)(Y)$ is a morphism of groups for any map $Y\rightarrow X$ in $C$. An abelian object $\mathcal{A}$ is a group object with the property  that $\mathcal{A}(X)$ is abelian for any $X$ in $\mathcal{S}$. It is straightforward to define rings and modules in a topos. The category $Ab(\mathcal{S})$ of abelian objects in a topos $\mathcal{S}$ (more generally, the category of modules over a ring) is an abelian category with enough injective objects.

\subsubsection{Geometric morphisms} Let $\mathcal{S}$ and $\mathcal{S}'$ be topoi. A morphism $u:\mathcal{S}\rightarrow\mathcal{S}'$ is a pair of functors $(u^{*}, u_{*})$ together with an adjunction which makes $u^*:\mathcal{S}'\rightarrow\mathcal{S}$ left adjoint to $u_*:\mathcal{S}\rightarrow\mathcal{S}'$ and such that $u^*$ is left exact (i.e. commutes with finite projective limits). It follows that the functor $u^*$ is exact (since a functor having a right adjoint is right exact); it is called the inverse image of $u$. The direct image $u_*$ is left exact only.

Given a functor $f^*:\mathcal{S}'\rightarrow\mathcal{S}$ which commutes with finite projective limits and arbitrary inductive limits, there exists an (essentially) unique morphism $f:\mathcal{S}\rightarrow\mathcal{S}'$ with inverse image $f^*$ (since such a functor has a uniquely determined right adjoint).

As a first example, let $\mathcal{S}$ be a topos and let $X$ be an object of $\mathcal{S}$. The slice category $\mathcal{S}/X$ is the category whose objects are maps $\mathcal{F}\rightarrow X$ and morphisms are maps over $X$. It follows from Giraud's axioms that $\mathcal{S}/X$ is a topos: $\mathcal{S}/X$ is said to be a \emph{slice} or \emph{induced} topos. The functor
$$\fonc{l_X^*}{\mathcal{S}}{\mathcal{S}/X}{\mathcal{F}}{\mathcal{F}\times X}$$
commutes with arbitrary projective and inductive limits (precisely because these limits are universal in $\mathcal{S}$);  hence $l^*_X$ is the inverse image of a morphism of topoi
$$l_X:\mathcal{S}/X\rightarrow \mathcal{S}. $$
The morphism $l_X$ is said to be a \emph{localization morphism} or a \emph{local homeomorphism}. The inverse image $l^*_X$ commutes with arbitrary projective limits, hence has a left adjoint $l_{X,!}$. One sees immediately that $l_{X,!}(\mathcal{F}\rightarrow X)=\mathcal{F}$. Note that $l_{X,!}$ is not left exact (it does not preserve the final object, unless $X=e_{\mathcal{S}}$). But the induced functor $l_{X,!}:Ab(\mathcal{S}/X)\rightarrow Ab(\mathcal{S})$ is exact, so that $l^*_X:Ab(\mathcal{S})\rightarrow Ab(\mathcal{S}/X)$ preserves injective objects.

For any topos $\mathcal{S}$ there is a unique morphism $e:\mathcal{S}\rightarrow{\bf Sets}$, hence ${\bf Sets}$ is a final object in the 2-category of topoi. Indeed, let $e_\mathcal{S}$ be a final object in $\mathcal{S}$ (which is unique up to a unique isomorphism) and let $\{*\}$ be the set with one element, i.e. the final object in ${\bf Sets}$. Since $e^*$ is left exact,  we have $e^*\{*\}=e_\mathcal{S}$. For any set $I$ we have $I=\coprod_I\{*\}$ hence $e^*I=\coprod_Ie^*\{*\}=\coprod_Ie_{\mathcal{S}}$ since $e^*$ commutes with inductive limits (hence with sums). The functor $e^*$ is therefore uniquely determined. Conversely, the functor defined above commutes with arbitrary inductive limits and finite projective limits; it is the inverse image of a morphism. The direct image $e_*$ is the \emph{global sections functor} $e_*(X)=\textrm{Hom}_{\mathcal{S}}(e_{\mathcal{S}},X)$ while $e^*$ is refered to as the \emph{constant sheaf functor}. If $\mathcal{S}=\textrm{Sh}(T)$ is the category of sheaves of sets on a topological space $T$, then the unique morphism $e:\textrm{Sh}(T)\rightarrow{\bf Sets}$ is induced by the unique continuous map $T\rightarrow \{*\}$.

A morphism of left exact sites $f^*:C'\rightarrow C$ is a left exact functor (and so  $f^*$ commutes with finite projective limits) $f^*:C'\rightarrow C$ sending covering families to covering families. Such a morphism $f^*$ of left exact sites induces a morphism of topoi $(f^*,f_*):\widetilde{C}\rightarrow \widetilde{C'}$ whose inverse image functor extends $f^*:C'\rightarrow C$, where $C$ is seen as a full subcategory of $\widetilde{C}$ via $y_C:C\hookrightarrow\widetilde{C}$. For example, let ${\bf Sch}/S$ be the category of $S$-schemes endowed  with the \'etale or the fppf-topology. The identity functor $({\bf Sch}/S)_{et}\rightarrow ({\bf Sch}/S)_{fppf}$ is a morphism of left exact sites.

\subsubsection{Topoi as spaces}
Let $f:T\rightarrow T'$ be a continuous map between topological spaces. Then the usual inverse and direct image functors yield a morphism $(f^*,f_*):\textrm{Sh}(T)\rightarrow \textrm{Sh}(T')$, where $\textrm{Sh}(T)$ is the category of sheaves of sets on $T$ (recall that $\textrm{Sh}(T)$ is equivalent to the category \'etal\'e spaces on $T$). Moreover, if the spaces $T$ and $T'$ are sober (a rather weak condition meaning that any irreducible closed subspace has a unique generic point) then any morphism of topoi $\textrm{Sh}(T)\rightarrow \textrm{Sh}(T')$ is induced by a unique continuous map. Therefore a topos may be  thought of as a generalized topological space and a morphism as a continuous map. Adopting  this  point of view, one can then define subtopoi, open and closed subtopoi, proper, connected, locally connected morphisms, cohomology, fundamental groups, fundamental groupoids and higher homotopy groups in the context of topos theory, generalizing the usual notions for nice topological spaces. However, the great majority  of this theory is not used in this paper.

\subsubsection{Cohomology} Let $\mathcal{S}$ be a topos. The unique map
$e_{\mathcal{S}}:\mathcal{S}\rightarrow {\bf Sets}$ yields a left exact functor
$\Gamma_{\mathcal{S}}:=e_{\mathcal{S},*}:Ab(\mathcal{S})\rightarrow {\bf Ab}$,
and one defines
$$H^i(\mathcal{S},\mathcal{A}):=R^i(\Gamma_{\mathcal{S}})\mathcal{A}$$
for any abelian object $\mathcal{A}$ in $\mathcal{S}$. This generalizes sheaf cohomology of topological spaces -- take $\mathcal{S}=Sh(T)$ -- on the one hand, and the cohomology of discrete groups -- take $\mathcal{S}$ to be the category of $G$--sets -- on the other.

If $f:\mathcal{E}\rightarrow\mathcal{S}$ is a morphism of topoi, then $e_{\mathcal{S}}\circ f=e_{\mathcal{E}}$ because $e_{\mathcal{E}}$ is the unique map from $\mathcal{E}$ to the final topos, so that $\Gamma_{\mathcal{S}}\circ f_*=\Gamma_{\mathcal{E}}$. Moreover, $f_*$ has an exact left adjoint, hence $f_*$ preserves injective objects. The Grothendieck spectral sequence for the composite functor $\Gamma_{\mathcal{S}}\circ f_*=\Gamma_{\mathcal{E}}$ is called the \emph{Leray spectral sequence for $f$}:
$$H^i(\mathcal{S},R^jf_*\mathcal{A})\Longrightarrow H^{i+j}(\mathcal{E},\mathcal{A}).$$

\subsubsection{The 2--category of topoi}\label{subsect-2catoftopos}
Let $\mathcal{S}$ and $\mathcal{E}$ be topoi. The class ${\bf Homtop}(\mathcal{E},\mathcal{S})$ of morphisms from $\mathcal{E}$ to $\mathcal{S}$ is a category. Indeed, let $f,g:\mathcal{E}\rightrightarrows\mathcal{S}$ be morphisms. A map $\alpha:f\rightarrow g$ in the category ${\bf Homtop}(\mathcal{E},\mathcal{S})$ is a natural transformation $\alpha:f_*\rightarrow g_*$. This structure is compatible with the composition of morphisms of topoi. In other words, the class of topoi forms a 2-category ${\bf Top}$. In this context, the notion of a commutative diagram has to be refined. Roughly speaking, diagrams are not required to commute on the nose, but only up to a coherent family of isomorphisms.  If $I$ is a small category, a pseudo-functor
$$\fonc{\mathcal{T}}{I}{{\bf Top}}{i}{\mathcal{T}(i)}$$
consists in a family of topoi $\mathcal{T}(i)$ indexed over the set of objects of $I$ together with morphisms $\mathcal{T}(f):\mathcal{T}(i)\rightarrow \mathcal{T}(j)$ for any map $f:i\rightarrow j$. For a pair of composable maps $(f,g)$ in $I$, we do not demand that the morphisms $\mathcal{T}(g)\circ \mathcal{T}(f)$ and $\mathcal{T}(g\circ f)$ be equal to one another. Instead, the pseudo-functor $\mathcal{T}$ contains additional information, which is given by a family of isomorphisms $\sigma_{(g,f)}:\mathcal{T}(g)\circ \mathcal{T}(f)\simeq \mathcal{T}(g\circ f)$ satisfying certain associativity conditions.
Such a pseudo-functor $\mathcal{T}$ is called a \emph{commutative diagram of topoi} indexed over $I$. For example, a commutative square of topoi is a diagram
\[\xymatrix{
\mathcal{E}'\ar[d]^{g'}\ar[r]^{f'}&\mathcal{S}'\ar[d]^g \\
\mathcal{E}\ar[r]^f&\mathcal{S}
}\]
given with an isomorphism $\alpha:g_*\circ f'_*\simeq f_*\circ g'_*$.

Such a commutative square is said to be 2-cartesian if, for any topos $\mathcal{T}$, the functor
$$
\appl{{\bf Homtop}(\mathcal{T},\mathcal{E}')}{{\bf Homtop}(\mathcal{T},\mathcal{S'})\times_{{\bf Homtop}(\mathcal{T},\mathcal{S})}{\bf Homtop}(\mathcal{T},\mathcal{E})}{a:\mathcal{T}\rightarrow\mathcal{E}'}{(f'\circ a,g'\circ a,\alpha\star a_*)}
$$
is an equivalence. Here the 2--fiber product category ${\bf Homtop}(\mathcal{T},\mathcal{S'})\times_{{\bf Homtop}(\mathcal{T},\mathcal{S})}{\bf Homtop}(\mathcal{T},\mathcal{E})$ is the category of triples $(b_1,b_2,\beta)$ where $b_1:\mathcal{T}\rightarrow \mathcal{S'}$ and $b_2:\mathcal{T}\rightarrow \mathcal{E}$ are morphisms and $\beta:g\circ b_1\simeq f\circ b_2$ is an isomorphism. A \emph{pull--back of topoi} is a commutative square which is 2--cartesian. If the diagram above is a pull--back, we write
$$\mathcal{E}'\simeq\mathcal{E}\times_{\mathcal{S}}\mathcal{S}'.$$
For example, for any morphism $f:\mathcal{E}\rightarrow \mathcal{S}$ and for any object $X$ of $\mathcal{S}$, the diagram
\[\xymatrix{
\mathcal{E}/f^*X\ar[r]^{f_{/X}}\ar[d]&\mathcal{S}/X\ar[d] \\
\mathcal{E}\ar[r]^f&\mathcal{S}
}\]
is a pull--back \cite{SGA4}[IV.5.10], where the vertical maps are localization morphisms (as defined above) and $f_{/X}$ is given by $f_{/X}^*(\mathcal{F}\rightarrow X):=(f^*\mathcal{F}\rightarrow f^*X)$. In other words, one has
$$\mathcal{E}/f^*X\simeq\mathcal{E}\times_{\mathcal{S}}\mathcal{S}/X.$$

\end{document}